\documentclass[10pt]{amsart}

\usepackage[utf8x]{inputenc}
\usepackage[english]{babel}
\usepackage{amsmath, amsfonts, amssymb, amsthm}

\usepackage[dvipsnames]{xcolor}
\usepackage[hyperindex=true,frenchlinks=true,colorlinks=true,
citecolor=NavyBlue,linkcolor=Mulberry,urlcolor=LimeGreen,linktocpage]{hyperref}

\usepackage{tikz}
\usetikzlibrary{fit}
\usetikzlibrary{decorations.pathmorphing}

\usepackage{stmaryrd}
\usepackage{enumitem}
\usepackage{hyperref}
\usepackage{subfigure}
\usepackage{multicol}

\author{Samuele Giraudo}
\date{\today}
\address{Institut Gaspard-Monge, université Paris-Est Marne-la-Vallée,
5 Boulevard Descartes, Champs-sur-Marne, 77454, Marne-la-Vallée cedex 2, France}
\email{samuele.giraudo@univ-mlv.fr}
\title[Intervals of balanced binary trees in the Tamari lattice]
{Intervals of balanced binary trees \\ in the Tamari lattice}

\newtheorem{Theoreme}{Theorem}[section]
\newtheorem{Proposition}[Theoreme]{Proposition}
\newtheorem{Lemme}[Theoreme]{Lemma}
\newtheorem{Definition}[Theoreme]{Definition}
\newtheorem{Corollaire}[Theoreme]{Corollary}
\newtheorem{Remarque}[Theoreme]{Remark}

\numberwithin{equation}{section}

\newcommand{\ArbreVide}{\perp}
\newcommand{\Ht}{\operatorname{h}}
\newcommand{\EnsNat}{\mathbb{N}}
\newcommand{\EnsRel}{\mathbb{Z}}

\newcommand{\EnsAB}{\mathcal{T}}
\newcommand{\EnsEq}{\mathcal{B}}
\newcommand{\ABCons}{\wedge}
\newcommand{\CouvTam}{\rightthreetimes}
\newcommand{\OrdTam}{\leq_{\operatorname{T}}}
\newcommand{\Tam}{\mathbb{T}}
\newcommand{\Des}{\operatorname{i}}
\newcommand{\HyperCube}{\mathbb{H}}
\newcommand{\ADroite}{\rightsquigarrow}
\newcommand{\EnsAdmissible}{\mathcal{A}}
\newcommand{\MotHt}{\operatorname{hw}}
\newcommand{\HtMot}{\Omega}
\newcommand{\Lang}{\mathcal{L}}

\newcommand{\EnsBud}{\mathcal{D}}
\newcommand{\Buds}{\operatorname{Buds}}
\newcommand{\Eval}{\operatorname{ev}}
\newcommand{\Nd}{\operatorname{n}}
\newcommand{\WDes}{\operatorname{wi}}
\newcommand{\Canop}{\operatorname{cnp}}
\newcommand{\EnsCanop}{\mathcal{C}}
\newcommand{\GramSerie}{\mathcal{S}}
\newcommand{\LettreB}{{\tt b}}
\newcommand{\Expr}{\operatorname{subs}}
\newcommand{\GenGraph}{\mathcal{G}}
\newcommand{\Nar}{\operatorname{nar}}

\newcommand{\PBT}{{\bf PBT}}
\newcommand{\Sym}{{\bf Sym}}
\newcommand{\PP}{{\bf P}}
\newcommand{\InvolAB}{\sim}
\newcommand{\Sloane}[1]{\href{http://oeis.org/#1}{{\bf #1}}}
\newcommand{\Wrt}{w.r.t.~}

\definecolor{Noir}{RGB}{0,0,0}
\definecolor{Rouge}{RGB}{205,35,38}
\definecolor{Bleu}{RGB}{2,60,195}
\definecolor{Vert}{RGB}{23,183,1}
\definecolor{Orange}{RGB}{255,113,15}
\definecolor{Marron}{RGB}{170,100,30}
\definecolor{Blanc}{RGB}{255,255,255}

\tikzstyle{Noeud} = [circle,draw=Bleu!100,fill=Bleu!20,thick,inner sep=0pt,
minimum size=10mm,line width=2pt,font=\Huge]
\tikzstyle{Feuille} = [rectangle,draw=Noir!100,fill=Noir!30,thick,
inner sep=0pt,minimum size=3mm]
\tikzstyle{Arete} = [Rouge!80, thick, draw, line width = 2pt]
\tikzstyle{SArbre} = [rectangle,draw=Orange!100,fill=Orange!30,thick,
inner sep=0pt,minimum size=16mm,line width=2pt,font=\Huge]
\tikzstyle{EtiqClair} = [draw = Bleu!100, fill = Blanc!100]
\tikzstyle{EtiqFonce} = [draw = Bleu!100, fill = Bleu!15]
\tikzstyle{Bourgeon} = [Noeud, draw = Vert!100, fill = Vert!15]
\tikzstyle{NoeudR} = [Noeud,rectangle]
\tikzstyle{Cercle} = [circle, draw = Rouge!100, fill = Rouge!70, thick, inner sep = 0pt, minimum size = 3mm]
\tikzstyle{AreteCube} = [line width=1pt, draw=Bleu]
\tikzstyle{AreteTam} = [Bleu!80, thick, draw, line width = 8pt]
\tikzstyle{NoeudTam} = [Noeud, minimum size = 10mm]
\tikzstyle{AreteGen} = [Bleu!80, thick, draw, line width = 3pt, ->]
\tikzstyle{Marque1} = [draw = Orange!100, fill = Orange!15]

\newcommand{\Feuille}[1]{%
    \scalebox{#1}{%
    \raisebox{.3em}{%
    \begin{tikzpicture}
        \node[Feuille]{};
    \end{tikzpicture}}}}

\newcommand{\Noeud}[2]{%
    \scalebox{#1}{%
    \raisebox{-.1em}{%
    \begin{tikzpicture}%
        \node[Noeud]{#2};%
    \end{tikzpicture}}}}

\newcommand{\NoeudR}[2]{%
    \scalebox{#1}{%
    \begin{tikzpicture}%
        \node[NoeudR]{#2};%
    \end{tikzpicture}}}

\newcommand{\Bourgeon}[2]{%
    \scalebox{#1}{\raisebox{-.2em}{%
    \begin{tikzpicture}%
        \node[Bourgeon]{\huge #2};%
    \end{tikzpicture}}}}

\newcommand{\BourgeonA}[4]{%
    \scalebox{#1}{%
    \begin{tikzpicture}%
        \node[Bourgeon](1)at(0,0){\huge #2};%
        \node[Noeud,EtiqClair](2)at(1,1.25){\LARGE #3};%
        \node[Bourgeon](3)at(2,0){\huge #4};%
        \draw[Arete](1)--(2);%
        \draw[Arete](2)--(3);%
    \end{tikzpicture}}}

\newcommand{\BourgeonACarre}[4]{%
    \scalebox{#1}{%
    \begin{tikzpicture}%
        \node[Bourgeon](1)at(0,0){\huge #2};%
        \node[NoeudR,EtiqClair](2)at(1,1.25){\LARGE #3};%
        \node[Bourgeon](3)at(2,0){\huge #4};%
        \draw[Arete](1)--(2);%
        \draw[Arete](2)--(3);%
    \end{tikzpicture}}}

\newcommand{\MotifA}[3]{%
    \scalebox{#1}{%
    \begin{tikzpicture}%
        \node[Noeud,EtiqClair](0)at(0,-1){#2};%
        \node[Noeud,EtiqClair](1)at(1,0){#3};%
        \draw[Arete](1)--(0);%
    \end{tikzpicture}}}

\newcommand{\MotifB}[3]{%
    \scalebox{#1}{%
    \begin{tikzpicture}%
        \node[Noeud,EtiqClair](0)at(0,0){#2};%
        \node[Noeud,EtiqClair](1)at(1,-1){#3};%
        \draw[Arete](0)--(1);
    \end{tikzpicture}}}

\newcommand{\ArbreA}[1]{%
    \scalebox{#1}{%
    \begin{tikzpicture}%
        \node[Feuille](0)at(0,-1){};%
        \node[Noeud](1)at(1,0){};%
        \node[Feuille](2)at(2,-1){};%
        \draw[Arete](1)--(0);%
        \draw[Arete](1)--(2);%
    \end{tikzpicture}}}

\newcommand{\ArbreB}[1]{%
    \scalebox{#1}{%
    \begin{tikzpicture}%
        \node[Feuille](0)at(0,-2){};%
        \node[Noeud](1)at(1,-1){};%
        \node[Feuille](2)at(2,-2){};%
        \draw[Arete](1)--(0);%
        \draw[Arete](1)--(2);%
        \node[Noeud](3)at(3,0){};%
        \node[Feuille](4)at(4,-1){};%
        \draw[Arete](3)--(1);%
        \draw[Arete](3)--(4);%
    \end{tikzpicture}}}

\begin{document}

\begin{abstract}
    We show that the set of balanced binary trees is closed by interval
    in the Tamari lattice. We establish that the intervals $[T, T']$ where
    $T$ and $T'$ are balanced binary trees are isomorphic as posets to a
    hypercube. We introduce synchronous grammars that allow to generate
    tree-like structures and obtain fixed-point functional equations to
    enumerate these. We also introduce imbalance tree patterns and show
    that they can be used to describe some sets of balanced binary trees
    that play a particular role in the Tamari lattice. Finally, we investigate
    other families of binary trees that are also closed by interval in the
    Tamari lattice.
\end{abstract}

\maketitle
{\small \tableofcontents}

\section{Introduction}

Binary search trees are used as data structures to represent dynamic totally
ordered sets (see~\cite{AU93, KNU398, CLR04}). The algorithms solving classical
related problems such as the insertion, the deletion or the search of a
given element can be performed in linear time in terms of the depth of the
encoding binary tree, and, if the binary tree is \emph{balanced}, these
operations can be made in logarithmic time in terms of the cardinality of
the represented set. Recall that a binary tree is balanced if for each node
$x$, the heights of the left and the right subtrees of $x$ differ by at most one.
\smallskip

The algorithmic of balanced binary trees fundamentally relies on the so-called
\emph{rotation} operation. An insertion or a deletion of an element in a dynamic
ordered set modifies the binary tree encoding it and can imbalance it. The
efficiency of these algorithms comes from the fact that binary search trees
can be rebalanced very quickly after the insertion or the deletion, using
no more than two rotations~\cite{AVL62}.
\smallskip

Surprisingly, this operation appears in a different context since it defines
a partial order on the set of binary trees of a given size. A binary tree
$T_0$ is smaller than a binary tree $T_1$ if it is possible to transform
$T_0$ into $T_1$ by performing a succession of right rotations. This partial
order, known as the \emph{Tamari order}~\cite{Tam62, STA99, KNU44}, defines
a lattice structure on the set of binary trees of a given size.
\smallskip

Since binary trees are naturally equipped by this order structure induced
by rotations, and the balance of balanced binary trees is maintained doing
rotations, we would like to investigate if balanced binary trees play a
particular role in the Tamari lattice. Our goal is to combine the two points
of view of the rotation operation. Computer trials show that the intervals
$[T, T']$ where $T$ and $T'$ are balanced binary trees are only made of
balanced binary trees. The main goal of this paper is to prove this property.
As a consequence, we give a characterization on the shape of these intervals
and, using grammars allowing the generation of tree-like structures, enumerate
these ones.
\smallskip

This article is organized as follows. In Section~\ref{sec:Preliminaries},
we set the essential notions about binary trees and balanced binary trees,
and we give the definition of the Tamari lattice in our setting. Section~\ref{sec:Closure}
is devoted to establish the main result: The set of balanced binary trees
is closed by interval in the Tamari lattice. In Section~\ref{sec:GramSync},
we define \emph{synchronous grammars}. This new sort of grammars allows to generate
sets of tree-like structures and gives a way to obtain a fixed-point functional
equation for the generating series enumerating these. In Section~\ref{sec:Motifs},
we introduce a notion of binary tree pattern, namely the imbalance tree patterns,
and a notion of pattern avoidance. We also define subsets of balanced binary
trees whose elements hold a particular position in the Tamari lattice. These
sets can also be defined as the balanced binary trees avoiding some given
imbalance tree patterns. In Section~\ref{sec:Shape}, we look at balanced
binary tree intervals and show that they are, as posets, isomorphic to hypercubes.
Encoding balanced binary tree intervals by kind of tree-like structures,
and by constructing the synchronous grammar generating these trees, we give
a fixed-point functional equation satisfied by the generating series enumerating
balanced binary tree intervals. We do the same for maximal balanced binary
tree intervals. Finally, in Section~\ref{sec:AutresFamillesCloses}, we
investigate three other families of binary trees that are closed by interval
in the Tamari lattice: The weight balanced binary trees, the binary trees
with a given canopy and the $k$-Narayana binary trees. We also look at a
generalization of balanced binary trees and prove, among other, that the
set of usual balanced binary trees is the only set among this generalization
that is both closed by interval in the Tamari lattice and the subposet of
the Tamari lattice induced by it has nontrivial intervals.
\medskip

This paper is an extended version of~\cite{Gir10} where only
Sections~\ref{sec:Preliminaries},~\ref{sec:Closure},~\ref{sec:Motifs} and~\ref{sec:Shape}
were developed.

\subsection*{Acknowledgments}
The author would like to thank Florent Hivert for introducing him to the
problem addressed in this paper, and Jean-Christophe Novelli and Florent
Hivert for their invaluable advice and their improvement suggestions. The
computations of this work have been done with the open-source mathematical
software Sage~\cite{SAGE}.

\section{Preliminaries} \label{sec:Preliminaries}

\subsection{Complete rooted planar binary trees}

In this article, we consider complete rooted planar binary trees and we
call these simply \emph{binary trees}. Recall that a binary tree $T$ is
either a \emph{leaf} (also called \emph{empty tree}) denoted by $\ArbreVide$,
or a node that is attached through two edges to two binary trees, called
respectively the \emph{left subtree} and the \emph{right subtree} of $T$.
The (unique) binary tree which has $L$ as left subtree and $R$ as right
subtree is denoted by $L \ABCons R$. Let also $\EnsAB_n$ be the set of binary
trees with $n$ nodes and $\EnsAB$ be the set of all binary trees. We use
in the sequel the standard terminology (\emph{i.e.}, \emph{parent}, \emph{child},
\emph{ancestor}, \emph{path}, \emph{etc.}) about binary trees~\cite{AU93}.
\medskip

In our graphical representations, nodes are represented by circles \Noeud{.35}{},
leaves by squares \Feuille{.5}, and edges by segments
\scalebox{.35}{
\begin{tikzpicture}
    \draw[Arete](0,0)--(-1,-.8);
\end{tikzpicture}}
or
\scalebox{.35}{
\begin{tikzpicture}
    \draw[Arete](0,0)--(1,-.8);
\end{tikzpicture}}.
Besides, we shall represent arbitrary subtrees by big squares like
\scalebox{.2}{\raisebox{.3em}{
    \begin{tikzpicture}
        \node[SArbre]{};
    \end{tikzpicture}
    }
},
and arbitrary paths by zigzag lines
\scalebox{.35}{
\begin{tikzpicture}
    \draw[Arete, decorate, decoration = zigzag] (0,0) -- (0,-.8);
\end{tikzpicture}}.
\medskip

Recall that the \emph{infix reading order} of the nodes of a binary tree $T$
consists in recursively visiting its left subtree, then its root, and finally
its right subtree. We say that a node $x$ of $T$ is the \emph{leftmost node}
if $x$ is the first visited node in the infix order. We also say that a
node $y$ is \emph{to the right \Wrt}a node $x$ if $x$ appears strictly
before $y$ in the infix order and we denote that by $x \ADroite_T y$. We
extend this notation to subtrees, saying that a subtree $S$ of $T$ is
\emph{to the right \Wrt}a node $x$ if for all nodes $y$ of $S$ we have
$x \ADroite_T y$.
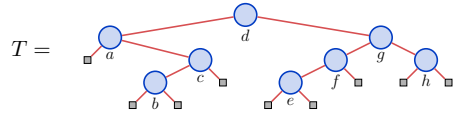
\begin{figure}[ht]
    \centering
    \scalebox{.3}{\begin{tikzpicture}
        \node[Feuille](0)at(0.0,-2){};
        \node[Noeud,label=below:\Huge $a$](1)at(1.0,-1){};
        \node[Feuille](2)at(2.0,-4){};
        \node[Noeud,label=below:\Huge $b$](3)at(3.0,-3){};
        \node[Feuille](4)at(4.0,-4){};
        \draw[Arete](3)--(2);
        \draw[Arete](3)--(4);
        \node[Noeud,label=below:\Huge $c$](5)at(5.0,-2){};
        \node[Feuille](6)at(6.0,-3){};
        \draw[Arete](5)--(3);
        \draw[Arete](5)--(6);
        \draw[Arete](1)--(0);
        \draw[Arete](1)--(5);
        \node[Noeud,label=below:\Huge $d$](7)at(7.0,0){};
        \node[Feuille](8)at(8.0,-4){};
        \node[Noeud,label=below:\Huge $e$](9)at(9.0,-3){};
        \node[Feuille](10)at(10.0,-4){};
        \draw[Arete](9)--(8);
        \draw[Arete](9)--(10);
        \node[Noeud,label=below:\Huge $f$](11)at(11.0,-2){};
        \node[Feuille](12)at(12.0,-3){};
        \draw[Arete](11)--(9);
        \draw[Arete](11)--(12);
        \node[Noeud,label=below:\Huge $g$](13)at(13.0,-1){};
        \node[Feuille](14)at(14.0,-3){};
        \node[Noeud,label=below:\Huge $h$](15)at(15.0,-2){};
        \node[Feuille](16)at(16.0,-3){};
        \draw[Arete](15)--(14);
        \draw[Arete](15)--(16);
        \draw[Arete](13)--(11);
        \draw[Arete](13)--(15);
        \draw[Arete](7)--(1);
        \draw[Arete](7)--(13);
        \node at (-2.5, -1.5) {\scalebox{2.8}{$T = $}};
    \end{tikzpicture}}
    \caption{An example of binary tree.}
    \label{fig:ExempleAB}
\end{figure}
For example, consider the binary tree $T$ depicted in Figure~\ref{fig:ExempleAB}.
The sequence $(a, b, c, d, e, f, g, h)$ is the sequence of all nodes of $T$
visited in the infix order. Hence, $a$ is the leftmost node of $T$ and we
have, among other, $a \ADroite_T b$ and $c \ADroite_T f$. Consider the
subtree $S$ of root $g$. It contains the nodes $e$, $f$, $g$ and $h$.
Hence, we have $a \ADroite_T S$, $b \ADroite_T S$, $c \ADroite_T S$ and
$d \ADroite_T S$. However, we neither have the relation $e \ADroite_T S$
since $S$ contains $e$, nor the relation $f \ADroite_T S$ since $S$ contains
$e$ and $f \ADroite_T e$ does not hold.

\subsection{Balanced binary trees}

If $T$ is a binary tree, we shall denote by $\Ht(T)$ its \emph{height},
that is the length of the longest path connecting its root to one of its
leaves. More formally,
\begin{equation}
    \Ht(T) :=
    \begin{cases}
        1 + \max \{\Ht(L), \Ht(R)\} & \mbox{if $T = L \ABCons R$,} \\
        0                           & \mbox{otherwise ($T = \ArbreVide$).}
    \end{cases}
\end{equation}
For example, we have $\Ht(\ArbreVide) = 0$, $\Ht \left( \ArbreA{.17} \right) = 1$,
and $\Ht \left( \raisebox{-.3em}{\ArbreB{.17}} \right) = 2$.
\medskip

Let us define the \emph{imbalance mapping} $\Des_T$ which associates an element
of $\EnsRel$ with a node $x$ of~$T$. It is defined by
\begin{equation}
    \Des_T(x) := \Ht(R) - \Ht(L),
\end{equation}
where $L$ (resp. $R$) is the left (resp. right) subtree of $x$. For example,
the imbalance values of the nodes of the binary tree $T$ shown in
Figure~\ref{fig:ExempleAB} satisfy $\Des_T(a) = 2$, $\Des_T(b) = 0$, $\Des_T(c) = -1$,
$\Des_T(d) = 0$, $\Des_T(e) = 0$, $\Des_T(f) = -1$, $\Des_T(g) = -1$ and $\Des_T(h) = 0$.
\medskip

A node $x$ is \emph{balanced} if
\begin{equation}
    \Des_T(x) \in \{-1, 0, 1\}.
\end{equation}
Balanced binary trees form a subset of $\EnsAB$ composed of binary trees
which have the property of being balanced:
\begin{Definition}
    A binary tree $T$ is \emph{balanced} if all nodes of $T$ are balanced.
\end{Definition}

Let us denote by $\EnsEq_n$ the set of balanced binary trees with $n$ nodes
(see Figure~\ref{fig:ArbresEq} for the first sets) and $\EnsEq$ the set of
all balanced binary trees. The number of balanced binary trees enumerated
according to their number of nodes is Sequence \Sloane{A006265} of~\cite{SLO08}
and begins as
\begin{equation}
    1, 1, 2, 1, 4, 6, 4, 17, 32, 44, 60, 70, 184, 476, 872, 1553, 2720, 4288, 6312, 9004.
\end{equation}
\begin{figure}[ht]
    \centering
    \begin{tabular}{c|l}
        $n$ & $\EnsEq_n$ \\
        $0$ & \scalebox{.17}{
            \begin{tikzpicture}
                \node[Feuille]{};
            \end{tikzpicture}} \\
        $1$ & \scalebox{.17}{
            \begin{tikzpicture}
                \node[Feuille](0)at(0,-1){};
                \node[Noeud](1)at(1,0){};
                \node[Feuille](2)at(2,-1){};
                \draw[Arete](1)--(0);
                \draw[Arete](1)--(2);
            \end{tikzpicture}} \\
        $2$ & \scalebox{.17}{
            \begin{tikzpicture}
                \node[Feuille](0)at(0,-2){};
                \node[Noeud](1)at(1,-1){};
                \node[Feuille](2)at(2,-2){};
                \draw[Arete](1)--(0);
                \draw[Arete](1)--(2);
                \node[Noeud](3)at(3,0){};
                \node[Feuille](4)at(4,-1){};
                \draw[Arete](3)--(1);
                \draw[Arete](3)--(4);
            \end{tikzpicture}}
            \scalebox{.17}{
            \begin{tikzpicture}
                \node[Feuille](0)at(0,-1){};
                \node[Noeud](1)at(1,0){};
                \node[Feuille](2)at(2,-2){};
                \node[Noeud](3)at(3,-1){};
                \node[Feuille](4)at(4,-2){};
                \draw[Arete](3)--(2);
                \draw[Arete](3)--(4);
                \draw[Arete](1)--(0);
                \draw[Arete](1)--(3);
            \end{tikzpicture}} \\
        $3$ & \scalebox{.17}{
            \begin{tikzpicture}
                \node[Feuille](0)at(0,-2){};
                \node[Noeud](1)at(1,-1){};
                \node[Feuille](2)at(2,-2){};
                \draw[Arete](1)--(0);
                \draw[Arete](1)--(2);
                \node[Noeud](3)at(3,0){};
                \node[Feuille](4)at(4,-2){};
                \node[Noeud](5)at(5,-1){};
                \node[Feuille](6)at(6,-2){};
                \draw[Arete](5)--(4);
                \draw[Arete](5)--(6);
                \draw[Arete](3)--(1);
                \draw[Arete](3)--(5);
            \end{tikzpicture}} \\
        $4$ & \scalebox{.17}{
            \begin{tikzpicture}
                \node[Feuille](0)at(0,-2){};
                \node[Noeud](1)at(1,-1){};
                \node[Feuille](2)at(2,-2){};
                \draw[Arete](1)--(0);
                \draw[Arete](1)--(2);
                \node[Noeud](3)at(3,0){};
                \node[Feuille](4)at(4,-3){};
                \node[Noeud](5)at(5,-2){};
                \node[Feuille](6)at(6,-3){};
                \draw[Arete](5)--(4);
                \draw[Arete](5)--(6);
                \node[Noeud](7)at(7,-1){};
                \node[Feuille](8)at(8,-2){};
                \draw[Arete](7)--(5);
                \draw[Arete](7)--(8);
                \draw[Arete](3)--(1);
                \draw[Arete](3)--(7);
            \end{tikzpicture}}
            \scalebox{.17}{
            \begin{tikzpicture}
                \node[Feuille](0)at(0,-2){};
                \node[Noeud](1)at(1,-1){};
                \node[Feuille](2)at(2,-2){};
                \draw[Arete](1)--(0);
                \draw[Arete](1)--(2);
                \node[Noeud](3)at(3,0){};
                \node[Feuille](4)at(4,-2){};
                \node[Noeud](5)at(5,-1){};
                \node[Feuille](6)at(6,-3){};
                \node[Noeud](7)at(7,-2){};
                \node[Feuille](8)at(8,-3){};
                \draw[Arete](7)--(6);
                \draw[Arete](7)--(8);
                \draw[Arete](5)--(4);
                \draw[Arete](5)--(7);
                \draw[Arete](3)--(1);
                \draw[Arete](3)--(5);
            \end{tikzpicture}}
            \scalebox{.17}{
            \begin{tikzpicture}
                \node[Feuille](0)at(0,-3){};
                \node[Noeud](1)at(1,-2){};
                \node[Feuille](2)at(2,-3){};
                \draw[Arete](1)--(0);
                \draw[Arete](1)--(2);
                \node[Noeud](3)at(3,-1){};
                \node[Feuille](4)at(4,-2){};
                \draw[Arete](3)--(1);
                \draw[Arete](3)--(4);
                \node[Noeud](5)at(5,0){};
                \node[Feuille](6)at(6,-2){};
                \node[Noeud](7)at(7,-1){};
                \node[Feuille](8)at(8,-2){};
                \draw[Arete](7)--(6);
                \draw[Arete](7)--(8);
                \draw[Arete](5)--(3);
                \draw[Arete](5)--(7);
            \end{tikzpicture}}
            \scalebox{.17}{
            \begin{tikzpicture}
                \node[Feuille](0)at(0,-2){};
                \node[Noeud](1)at(1,-1){};
                \node[Feuille](2)at(2,-3){};
                \node[Noeud](3)at(3,-2){};
                \node[Feuille](4)at(4,-3){};
                \draw[Arete](3)--(2);
                \draw[Arete](3)--(4);
                \draw[Arete](1)--(0);
                \draw[Arete](1)--(3);
                \node[Noeud](5)at(5,0){};
                \node[Feuille](6)at(6,-2){};
                \node[Noeud](7)at(7,-1){};
                \node[Feuille](8)at(8,-2){};
                \draw[Arete](7)--(6);
                \draw[Arete](7)--(8);
                \draw[Arete](5)--(1);
                \draw[Arete](5)--(7);
            \end{tikzpicture}} \\
        $5$ & \scalebox{.17}{
            \begin{tikzpicture}
                \node[Feuille](0)at(0,-3){};
                \node[Noeud](1)at(1,-2){};
                \node[Feuille](2)at(2,-3){};
                \draw[Arete](1)--(0);
                \draw[Arete](1)--(2);
                \node[Noeud](3)at(3,-1){};
                \node[Feuille](4)at(4,-3){};
                \node[Noeud](5)at(5,-2){};
                \node[Feuille](6)at(6,-3){};
                \draw[Arete](5)--(4);
                \draw[Arete](5)--(6);
                \draw[Arete](3)--(1);
                \draw[Arete](3)--(5);
                \node[Noeud](7)at(7,0){};
                \node[Feuille](8)at(8,-2){};
                \node[Noeud](9)at(9,-1){};
                \node[Feuille](10)at(10,-2){};
                \draw[Arete](9)--(8);
                \draw[Arete](9)--(10);
                \draw[Arete](7)--(3);
                \draw[Arete](7)--(9);
            \end{tikzpicture}}
            \scalebox{.17}{
            \begin{tikzpicture}
                \node[Feuille](0)at(0,-2){};
                \node[Noeud](1)at(1,-1){};
                \node[Feuille](2)at(2,-2){};
                \draw[Arete](1)--(0);
                \draw[Arete](1)--(2);
                \node[Noeud](3)at(3,0){};
                \node[Feuille](4)at(4,-3){};
                \node[Noeud](5)at(5,-2){};
                \node[Feuille](6)at(6,-3){};
                \draw[Arete](5)--(4);
                \draw[Arete](5)--(6);
                \node[Noeud](7)at(7,-1){};
                \node[Feuille](8)at(8,-3){};
                \node[Noeud](9)at(9,-2){};
                \node[Feuille](10)at(10,-3){};
                \draw[Arete](9)--(8);
                \draw[Arete](9)--(10);
                \draw[Arete](7)--(5);
                \draw[Arete](7)--(9);
                \draw[Arete](3)--(1);
                \draw[Arete](3)--(7);
            \end{tikzpicture}}
            \scalebox{.17}{
            \begin{tikzpicture}
                \node[Feuille](0)at(0,-3){};
                \node[Noeud](1)at(1,-2){};
                \node[Feuille](2)at(2,-3){};
                \draw[Arete](1)--(0);
                \draw[Arete](1)--(2);
                \node[Noeud](3)at(3,-1){};
                \node[Feuille](4)at(4,-2){};
                \draw[Arete](3)--(1);
                \draw[Arete](3)--(4);
                \node[Noeud](5)at(5,0){};
                \node[Feuille](6)at(6,-3){};
                \node[Noeud](7)at(7,-2){};
                \node[Feuille](8)at(8,-3){};
                \draw[Arete](7)--(6);
                \draw[Arete](7)--(8);
                \node[Noeud](9)at(9,-1){};
                \node[Feuille](10)at(10,-2){};
                \draw[Arete](9)--(7);
                \draw[Arete](9)--(10);
                \draw[Arete](5)--(3);
                \draw[Arete](5)--(9);
            \end{tikzpicture}}
            \scalebox{.17}{
            \begin{tikzpicture}
                \node[Feuille](0)at(0,-3){};
                \node[Noeud](1)at(1,-2){};
                \node[Feuille](2)at(2,-3){};
                \draw[Arete](1)--(0);
                \draw[Arete](1)--(2);
                \node[Noeud](3)at(3,-1){};
                \node[Feuille](4)at(4,-2){};
                \draw[Arete](3)--(1);
                \draw[Arete](3)--(4);
                \node[Noeud](5)at(5,0){};
                \node[Feuille](6)at(6,-2){};
                \node[Noeud](7)at(7,-1){};
                \node[Feuille](8)at(8,-3){};
                \node[Noeud](9)at(9,-2){};
                \node[Feuille](10)at(10,-3){};
                \draw[Arete](9)--(8);
                \draw[Arete](9)--(10);
                \draw[Arete](7)--(6);
                \draw[Arete](7)--(9);
                \draw[Arete](5)--(3);
                \draw[Arete](5)--(7);
            \end{tikzpicture}}
            \scalebox{.17}{
            \begin{tikzpicture}
                \node[Feuille](0)at(0,-2){};
                \node[Noeud](1)at(1,-1){};
                \node[Feuille](2)at(2,-3){};
                \node[Noeud](3)at(3,-2){};
                \node[Feuille](4)at(4,-3){};
                \draw[Arete](3)--(2);
                \draw[Arete](3)--(4);
                \draw[Arete](1)--(0);
                \draw[Arete](1)--(3);
                \node[Noeud](5)at(5,0){};
                \node[Feuille](6)at(6,-3){};
                \node[Noeud](7)at(7,-2){};
                \node[Feuille](8)at(8,-3){};
                \draw[Arete](7)--(6);
                \draw[Arete](7)--(8);
                \node[Noeud](9)at(9,-1){};
                \node[Feuille](10)at(10,-2){};
                \draw[Arete](9)--(7);
                \draw[Arete](9)--(10);
                \draw[Arete](5)--(1);
                \draw[Arete](5)--(9);
            \end{tikzpicture}}
            \scalebox{.17}{
            \begin{tikzpicture}
                \node[Feuille](0)at(0,-2){};
                \node[Noeud](1)at(1,-1){};
                \node[Feuille](2)at(2,-3){};
                \node[Noeud](3)at(3,-2){};
                \node[Feuille](4)at(4,-3){};
                \draw[Arete](3)--(2);
                \draw[Arete](3)--(4);
                \draw[Arete](1)--(0);
                \draw[Arete](1)--(3);
                \node[Noeud](5)at(5,0){};
                \node[Feuille](6)at(6,-2){};
                \node[Noeud](7)at(7,-1){};
                \node[Feuille](8)at(8,-3){};
                \node[Noeud](9)at(9,-2){};
                \node[Feuille](10)at(10,-3){};
                \draw[Arete](9)--(8);
                \draw[Arete](9)--(10);
                \draw[Arete](7)--(6);
                \draw[Arete](7)--(9);
                \draw[Arete](5)--(1);
                \draw[Arete](5)--(7);
            \end{tikzpicture}} \\
        $6$ & \scalebox{.17}{
            \begin{tikzpicture}
                \node[Feuille](0)at(0,-3){};
                \node[Noeud](1)at(1,-2){};
                \node[Feuille](2)at(2,-3){};
                \draw[Arete](1)--(0);
                \draw[Arete](1)--(2);
                \node[Noeud](3)at(3,-1){};
                \node[Feuille](4)at(4,-3){};
                \node[Noeud](5)at(5,-2){};
                \node[Feuille](6)at(6,-3){};
                \draw[Arete](5)--(4);
                \draw[Arete](5)--(6);
                \draw[Arete](3)--(1);
                \draw[Arete](3)--(5);
                \node[Noeud](7)at(7,0){};
                \node[Feuille](8)at(8,-3){};
                \node[Noeud](9)at(9,-2){};
                \node[Feuille](10)at(10,-3){};
                \draw[Arete](9)--(8);
                \draw[Arete](9)--(10);
                \node[Noeud](11)at(11,-1){};
                \node[Feuille](12)at(12,-2){};
                \draw[Arete](11)--(9);
                \draw[Arete](11)--(12);
                \draw[Arete](7)--(3);
                \draw[Arete](7)--(11);
            \end{tikzpicture}}
            \scalebox{.17}{
            \begin{tikzpicture}
                \node[Feuille](0)at(0,-3){};
                \node[Noeud](1)at(1,-2){};
                \node[Feuille](2)at(2,-3){};
                \draw[Arete](1)--(0);
                \draw[Arete](1)--(2);
                \node[Noeud](3)at(3,-1){};
                \node[Feuille](4)at(4,-3){};
                \node[Noeud](5)at(5,-2){};
                \node[Feuille](6)at(6,-3){};
                \draw[Arete](5)--(4);
                \draw[Arete](5)--(6);
                \draw[Arete](3)--(1);
                \draw[Arete](3)--(5);
                \node[Noeud](7)at(7,0){};
                \node[Feuille](8)at(8,-2){};
                \node[Noeud](9)at(9,-1){};
                \node[Feuille](10)at(10,-3){};
                \node[Noeud](11)at(11,-2){};
                \node[Feuille](12)at(12,-3){};
                \draw[Arete](11)--(10);
                \draw[Arete](11)--(12);
                \draw[Arete](9)--(8);
                \draw[Arete](9)--(11);
                \draw[Arete](7)--(3);
                \draw[Arete](7)--(9);
            \end{tikzpicture}}
            \scalebox{.17}{
            \begin{tikzpicture}
                \node[Feuille](0)at(0,-3){};
                \node[Noeud](1)at(1,-2){};
                \node[Feuille](2)at(2,-3){};
                \draw[Arete](1)--(0);
                \draw[Arete](1)--(2);
                \node[Noeud](3)at(3,-1){};
                \node[Feuille](4)at(4,-2){};
                \draw[Arete](3)--(1);
                \draw[Arete](3)--(4);
                \node[Noeud](5)at(5,0){};
                \node[Feuille](6)at(6,-3){};
                \node[Noeud](7)at(7,-2){};
                \node[Feuille](8)at(8,-3){};
                \draw[Arete](7)--(6);
                \draw[Arete](7)--(8);
                \node[Noeud](9)at(9,-1){};
                \node[Feuille](10)at(10,-3){};
                \node[Noeud](11)at(11,-2){};
                \node[Feuille](12)at(12,-3){};
                \draw[Arete](11)--(10);
                \draw[Arete](11)--(12);
                \draw[Arete](9)--(7);
                \draw[Arete](9)--(11);
                \draw[Arete](5)--(3);
                \draw[Arete](5)--(9);
            \end{tikzpicture}}
            \scalebox{.17}{
            \begin{tikzpicture}
                \node[Feuille](0)at(0,-2){};
                \node[Noeud](1)at(1,-1){};
                \node[Feuille](2)at(2,-3){};
                \node[Noeud](3)at(3,-2){};
                \node[Feuille](4)at(4,-3){};
                \draw[Arete](3)--(2);
                \draw[Arete](3)--(4);
                \draw[Arete](1)--(0);
                \draw[Arete](1)--(3);
                \node[Noeud](5)at(5,0){};
                \node[Feuille](6)at(6,-3){};
                \node[Noeud](7)at(7,-2){};
                \node[Feuille](8)at(8,-3){};
                \draw[Arete](7)--(6);
                \draw[Arete](7)--(8);
                \node[Noeud](9)at(9,-1){};
                \node[Feuille](10)at(10,-3){};
                \node[Noeud](11)at(11,-2){};
                \node[Feuille](12)at(12,-3){};
                \draw[Arete](11)--(10);
                \draw[Arete](11)--(12);
                \draw[Arete](9)--(7);
                \draw[Arete](9)--(11);
                \draw[Arete](5)--(1);
                \draw[Arete](5)--(9);
            \end{tikzpicture}}
    \end{tabular}
    \caption{The first balanced binary trees.}
    \label{fig:ArbresEq}
\end{figure}
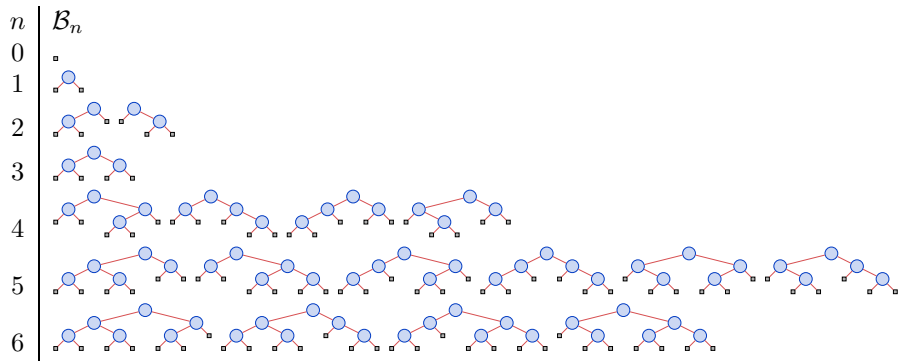

\subsection{The Tamari lattice}

The Tamari lattice can be defined in several ways depending on which kind
of Catalan object (\emph{i.e.}, objects in bijection with binary trees) the
order relation is defined. The most common definitions are made on integer
vectors with some conditions~\cite{STA99}, on forests and binary trees~\cite{KNU44},
and on Dyck paths~\cite{BEB07}. We give here the most convenient definition
for our use. First, let us recall the right rotation operation:
\begin{Definition}
    Let $T_0$ be a binary tree and $y$ be a node of $T_0$ having a nonempty
    left subtree. Let $S_0 := (A \ABCons B) \ABCons C$ be the subtree of
    root $y$ of $T_0$ and $T_1$ be the binary tree obtained by replacing
    $S_0$ by $A \ABCons (B \ABCons C)$ (see Figure~\ref{fig:Rotation}).
    Then the \emph{right rotation of root} $y$ sends $T_0$ to $T_1$.
\end{Definition}

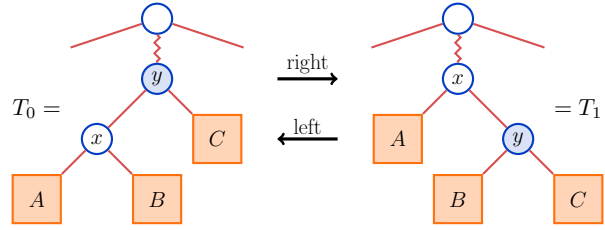
\begin{figure}[ht]
    \centering
    \scalebox{.4}{
    \begin{tikzpicture}
        \node[Noeud, EtiqClair] (racine) at (0, 0) {};
        \node (g) at (-3, -1) {};
        \node (d) at (3, -1) {};
        \node[Noeud, EtiqFonce] (r) at (0, -2) {\Huge $y$};
        \node[Noeud, EtiqClair] (q) at (-2, -4) {\Huge $x$};
        \node[SArbre] (A) at (-4, -6) {$A$};
        \node[SArbre] (B) at (0, -6) {$B$};
        \node[SArbre] (C) at (2, -4) {$C$};
        \draw[Arete] (racine) -- (g);
        \draw[Arete] (racine) -- (d);
        \draw[Arete, decorate, decoration = zigzag] (racine) -- (r);
        \draw[Arete] (r) -- (q);
        \draw[Arete] (r) -- (C);
        \draw[Arete] (q) -- (A);
        \draw[Arete] (q) -- (B);
        \node at (-4, -3) {\scalebox{2.2}{$T_0 = $}};
        \path (4, -2) edge[line width=3pt, ->] node[anchor=south,above,font=\Huge,Noir!100]{right} (6, -2);
        \path (6, -4) edge[line width=3pt, ->] node[anchor=south,above,font=\Huge,Noir!100]{left} (4, -4);
        \node[Noeud, EtiqClair] (racine') at (10, 0) {};
        \node (g') at (7, -1) {};
        \node (d') at (13, -1) {};
        \node[Noeud, EtiqFonce] (r') at (12, -4) {\Huge $y$};
        \node[Noeud, EtiqClair] (q') at (10, -2) {\Huge $x$};
        \node[SArbre] (A') at (8, -4) {$A$};
        \node[SArbre] (B') at (10, -6) {$B$};
        \node[SArbre] (C') at (14, -6) {$C$};
        \draw[Arete] (racine') -- (g');
        \draw[Arete] (racine') -- (d');
        \draw[Arete, decorate, decoration = zigzag] (racine') -- (q');
        \draw[Arete] (q') -- (r');
        \draw[Arete] (q') -- (A');
        \draw[Arete] (r') -- (B');
        \draw[Arete] (r') -- (C');
        \node at (14, -3) {\scalebox{2.2}{$ = T_1$}};
    \end{tikzpicture}}
    \caption{The right rotation of root $y$.}
    \label{fig:Rotation}
\end{figure}

We write $T_0 \CouvTam T_1$ if $T_1$ can be obtained by a right rotation
from $T_0$. We call the relation $\CouvTam$ the \emph{partial Tamari relation}.
Note that the application of a right rotation to a binary tree does not
change the infix order of its nodes. In the sequel, we mainly talk about
right rotations, so we call these simply \emph{rotations}. We are now in
a position to give our definition of the Tamari order.
\begin{Definition} \label{def:OrdreTamari}
    The \emph{Tamari relation} $\OrdTam$ is the reflexive and transitive closure
    of the partial Tamari relation $\CouvTam$. In other words, we have $T_0 \OrdTam T_k$
    if there exists a sequence $T_1, \dots, T_{k-1}$ of binary trees such that
    \begin{equation}
        T_0 \CouvTam T_1 \CouvTam \cdots \CouvTam T_{k-1} \CouvTam T_k.
    \end{equation}
\end{Definition}

The Tamari relation is an order relation. Indeed, $\OrdTam$ is reflexive
and transitive by definition. To prove that $\OrdTam$ is antisymmetric,
consider the statistic $\phi : \EnsAB \rightarrow \EnsNat$ where $\phi(T)$
is the sum for all nodes $x$ of $T$ of the number of the nodes constituting
the right subtree of $x$. It is plain that if $T_0 \CouvTam T_1$ then
$\phi(T_0) < \phi(T_1)$, showing that $\OrdTam$ is antisymmetric.
\medskip

For $n \geq 0$, the set $\EnsAB_n$ with the order relation $\OrdTam$ defines
a lattice, namely the Tamari lattice (see~\cite{HT72}). We denote by
$\Tam_n := (\EnsAB_n, \OrdTam)$ the Tamari lattice of order $n$ (see
Figure~\ref{fig:TreillisTamari} for some examples).
\begin{figure}[ht]
    \centering
    \subfigure[$\Tam_3$]{\makebox[5cm]{\scalebox{.14}{
    \begin{tikzpicture}
        \node[Feuille](a000n0)at(0,-3){};
        \node[NoeudTam](a000n1)at(1,-2){};
        \node[Feuille](a000n2)at(2,-3){};
        \draw[Arete](a000n1)--(a000n0);
        \draw[Arete](a000n1)--(a000n2);
        \node[NoeudTam](a000n3)at(3,-1){};
        \node[Feuille](a000n4)at(4,-2){};
        \draw[Arete](a000n3)--(a000n1);
        \draw[Arete](a000n3)--(a000n4);
        \node[NoeudTam](a000n5)at(5,0){};
        \node[Feuille](a000n6)at(6,-1){};
        \draw[Arete](a000n5)--(a000n3);
        \draw[Arete](a000n5)--(a000n6);
        \node[fit=(a000n0) (a000n1) (a000n2) (a000n3) (a000n4) (a000n5) (a000n6)] (c000) {};

        \node[Feuille](a100n0)at(-7,-9){};
        \node[NoeudTam](a100n1)at(-6,-8){};
        \node[Feuille](a100n2)at(-5,-10){};
        \node[NoeudTam](a100n3)at(-4,-9){};
        \node[Feuille](a100n4)at(-3,-10){};
        \draw[Arete](a100n3)--(a100n2);
        \draw[Arete](a100n3)--(a100n4);
        \draw[Arete](a100n1)--(a100n0);
        \draw[Arete](a100n1)--(a100n3);
        \node[NoeudTam](a100n5)at(-2,-7){};
        \node[Feuille](a100n6)at(-1,-8){};
        \draw[Arete](a100n5)--(a100n1);
        \draw[Arete](a100n5)--(a100n6);
        \node[fit=(a100n0) (a100n1) (a100n2) (a100n3) (a100n4) (a100n5) (a100n6)] (c100) {};

        \node[Feuille](a200n0)at(-7,-15){};
        \node[NoeudTam](a200n1)at(-6,-14){};
        \node[Feuille](a200n2)at(-5,-17){};
        \node[NoeudTam](a200n3)at(-4,-16){};
        \node[Feuille](a200n4)at(-3,-17){};
        \draw[Arete](a200n3)--(a200n2);
        \draw[Arete](a200n3)--(a200n4);
        \node[NoeudTam](a200n5)at(-2,-15){};
        \node[Feuille](a200n6)at(-1,-16){};
        \draw[Arete](a200n5)--(a200n3);
        \draw[Arete](a200n5)--(a200n6);
        \draw[Arete](a200n1)--(a200n0);
        \draw[Arete](a200n1)--(a200n5);
        \node[fit=(a200n0) (a200n1) (a200n2) (a200n3) (a200n4) (a200n5) (a200n6)] (c200) {};

        \node[Feuille](a010n0)at(7,-13){};
        \node[NoeudTam](a010n1)at(8,-12){};
        \node[Feuille](a010n2)at(9,-13){};
        \draw[Arete](a010n1)--(a010n0);
        \draw[Arete](a010n1)--(a010n2);
        \node[NoeudTam](a010n3)at(10,-11){};
        \node[Feuille](a010n4)at(11,-13){};
        \node[NoeudTam](a010n5)at(12,-12){};
        \node[Feuille](a010n6)at(13,-13){};
        \draw[Arete](a010n5)--(a010n4);
        \draw[Arete](a010n5)--(a010n6);
        \draw[Arete](a010n3)--(a010n1);
        \draw[Arete](a010n3)--(a010n5);
        \node[fit=(a010n0) (a010n1) (a010n2) (a010n3) (a010n4) (a010n5) (a010n6)] (c010) {};

        \node[Feuille](a210n0)at(0,-22){};
        \node[NoeudTam](a210n1)at(1,-21){};
        \node[Feuille](a210n2)at(2,-23){};
        \node[NoeudTam](a210n3)at(3,-22){};
        \node[Feuille](a210n4)at(4,-24){};
        \node[NoeudTam](a210n5)at(5,-23){};
        \node[Feuille](a210n6)at(6,-24){};
        \draw[Arete](a210n5)--(a210n4);
        \draw[Arete](a210n5)--(a210n6);
        \draw[Arete](a210n3)--(a210n2);
        \draw[Arete](a210n3)--(a210n5);
        \draw[Arete](a210n1)--(a210n0);
        \draw[Arete](a210n1)--(a210n3);
        \node[fit=(a210n0) (a210n1) (a210n2) (a210n3) (a210n4) (a210n5) (a210n6)] (c210) {};

        \draw[AreteTam] (c000)--(c100);
        \draw[AreteTam] (c000)--(c010);
        \draw[AreteTam] (c100)--(c200);
        \draw[AreteTam] (c200)--(c210);
        \draw[AreteTam] (c010)--(c210);
    \end{tikzpicture}}}}
    \subfigure[$\Tam_4$]{\makebox[8cm]{\scalebox{.14}{
    \begin{tikzpicture}
        \node[Feuille](a0000n0)at(0,-4){};
        \node[NoeudTam](a0000n1)at(1,-3){};
        \node[Feuille](a0000n2)at(2,-4){};
        \draw[Arete](a0000n1)--(a0000n0);
        \draw[Arete](a0000n1)--(a0000n2);
        \node[NoeudTam](a0000n3)at(3,-2){};
        \node[Feuille](a0000n4)at(4,-3){};
        \draw[Arete](a0000n3)--(a0000n1);
        \draw[Arete](a0000n3)--(a0000n4);
        \node[NoeudTam](a0000n5)at(5,-1){};
        \node[Feuille](a0000n6)at(6,-2){};
        \draw[Arete](a0000n5)--(a0000n3);
        \draw[Arete](a0000n5)--(a0000n6);
        \node[NoeudTam](a0000n7)at(7,0){};
        \node[Feuille](a0000n8)at(8,-1){};
        \draw[Arete](a0000n7)--(a0000n5);
        \draw[Arete](a0000n7)--(a0000n8);
        \node[fit=(a0000n0) (a0000n1) (a0000n2) (a0000n3) (a0000n4) (a0000n5) (a0000n6) (a0000n7) (a0000n8)] (c0000) {};

        \node[Feuille](a1000n0)at(-14,-7){};
        \node[NoeudTam](a1000n1)at(-13,-6){};
        \node[Feuille](a1000n2)at(-12,-8){};
        \node[NoeudTam](a1000n3)at(-11,-7){};
        \node[Feuille](a1000n4)at(-10,-8){};
        \draw[Arete](a1000n3)--(a1000n2);
        \draw[Arete](a1000n3)--(a1000n4);
        \draw[Arete](a1000n1)--(a1000n0);
        \draw[Arete](a1000n1)--(a1000n3);
        \node[NoeudTam](a1000n5)at(-9,-5){};
        \node[Feuille](a1000n6)at(-8,-6){};
        \draw[Arete](a1000n5)--(a1000n1);
        \draw[Arete](a1000n5)--(a1000n6);
        \node[NoeudTam](a1000n7)at(-7,-4){};
        \node[Feuille](a1000n8)at(-6,-5){};
        \draw[Arete](a1000n7)--(a1000n5);
        \draw[Arete](a1000n7)--(a1000n8);
        \node[fit=(a1000n0) (a1000n1) (a1000n2) (a1000n3) (a1000n4) (a1000n5) (a1000n6) (a1000n7) (a1000n8)] (c1000) {};

        \node[Feuille](a0100n0)at(14,-11){};
        \node[NoeudTam](a0100n1)at(15,-10){};
        \node[Feuille](a0100n2)at(16,-11){};
        \draw[Arete](a0100n1)--(a0100n0);
        \draw[Arete](a0100n1)--(a0100n2);
        \node[NoeudTam](a0100n3)at(17,-9){};
        \node[Feuille](a0100n4)at(18,-11){};
        \node[NoeudTam](a0100n5)at(19,-10){};
        \node[Feuille](a0100n6)at(20,-11){};
        \draw[Arete](a0100n5)--(a0100n4);
        \draw[Arete](a0100n5)--(a0100n6);
        \draw[Arete](a0100n3)--(a0100n1);
        \draw[Arete](a0100n3)--(a0100n5);
        \node[NoeudTam](a0100n7)at(21,-8){};
        \node[Feuille](a0100n8)at(22,-9){};
        \draw[Arete](a0100n7)--(a0100n3);
        \draw[Arete](a0100n7)--(a0100n8);
        \node[fit=(a0100n0) (a0100n1) (a0100n2) (a0100n3) (a0100n4) (a0100n5) (a0100n6) (a0100n7) (a0100n8) ] (c0100) {};

        \node[Feuille](a2000n0)at(-28,-11){};
        \node[NoeudTam](a2000n1)at(-27,-10){};
        \node[Feuille](a2000n2)at(-26,-13){};
        \node[NoeudTam](a2000n3)at(-25,-12){};
        \node[Feuille](a2000n4)at(-24,-13){};
        \draw[Arete](a2000n3)--(a2000n2);
        \draw[Arete](a2000n3)--(a2000n4);
        \node[NoeudTam](a2000n5)at(-23,-11){};
        \node[Feuille](a2000n6)at(-22,-12){};
        \draw[Arete](a2000n5)--(a2000n3);
        \draw[Arete](a2000n5)--(a2000n6);
        \draw[Arete](a2000n1)--(a2000n0);
        \draw[Arete](a2000n1)--(a2000n5);
        \node[NoeudTam](a2000n7)at(-21,-9){};
        \node[Feuille](a2000n8)at(-20,-10){};
        \draw[Arete](a2000n7)--(a2000n1);
        \draw[Arete](a2000n7)--(a2000n8);
        \node[fit=(a2000n0) (a2000n1) (a2000n2) (a2000n3) (a2000n4) (a2000n5) (a2000n6) (a2000n7) (a2000n8) ] (c2000) {};

        \node[Feuille](a2100n0)at(27,-17){};
        \node[NoeudTam](a2100n1)at(28,-16){};
        \node[Feuille](a2100n2)at(29,-18){};
        \node[NoeudTam](a2100n3)at(30,-17){};
        \node[Feuille](a2100n4)at(31,-19){};
        \node[NoeudTam](a2100n5)at(32,-18){};
        \node[Feuille](a2100n6)at(33,-19){};
        \draw[Arete](a2100n5)--(a2100n4);
        \draw[Arete](a2100n5)--(a2100n6);
        \draw[Arete](a2100n3)--(a2100n2);
        \draw[Arete](a2100n3)--(a2100n5);
        \draw[Arete](a2100n1)--(a2100n0);
        \draw[Arete](a2100n1)--(a2100n3);
        \node[NoeudTam](a2100n7)at(34,-15){};
        \node[Feuille](a2100n8)at(35,-16){};
        \draw[Arete](a2100n7)--(a2100n1);
        \draw[Arete](a2100n7)--(a2100n8);
        \node[fit=(a2100n0) (a2100n1) (a2100n2) (a2100n3) (a2100n4) (a2100n5) (a2100n6) (a2100n7) (a2100n8) ] (c2100) {};

        \node[Feuille](a0010n0)at(0,-18){};
        \node[NoeudTam](a0010n1)at(1,-17){};
        \node[Feuille](a0010n2)at(2,-18){};
        \draw[Arete](a0010n1)--(a0010n0);
        \draw[Arete](a0010n1)--(a0010n2);
        \node[NoeudTam](a0010n3)at(3,-16){};
        \node[Feuille](a0010n4)at(4,-17){};
        \draw[Arete](a0010n3)--(a0010n1);
        \draw[Arete](a0010n3)--(a0010n4);
        \node[NoeudTam](a0010n5)at(5,-15){};
        \node[Feuille](a0010n6)at(6,-17){};
        \node[NoeudTam](a0010n7)at(7,-16){};
        \node[Feuille](a0010n8)at(8,-17){};
        \draw[Arete](a0010n7)--(a0010n6);
        \draw[Arete](a0010n7)--(a0010n8);
        \draw[Arete](a0010n5)--(a0010n3);
        \draw[Arete](a0010n5)--(a0010n7);
        \node[fit=(a0010n0) (a0010n1) (a0010n2) (a0010n3) (a0010n4) (a0010n5) (a0010n6) (a0010n7) (a0010n8) ] (c0010) {};

        \node[Feuille](a0200n0)at(14,-22){};
        \node[NoeudTam](a0200n1)at(15,-21){};
        \node[Feuille](a0200n2)at(16,-22){};
        \draw[Arete](a0200n1)--(a0200n0);
        \draw[Arete](a0200n1)--(a0200n2);
        \node[NoeudTam](a0200n3)at(17,-20){};
        \node[Feuille](a0200n4)at(18,-23){};
        \node[NoeudTam](a0200n5)at(19,-22){};
        \node[Feuille](a0200n6)at(20,-23){};
        \draw[Arete](a0200n5)--(a0200n4);
        \draw[Arete](a0200n5)--(a0200n6);
        \node[NoeudTam](a0200n7)at(21,-21){};
        \node[Feuille](a0200n8)at(22,-22){};
        \draw[Arete](a0200n7)--(a0200n5);
        \draw[Arete](a0200n7)--(a0200n8);
        \draw[Arete](a0200n3)--(a0200n1);
        \draw[Arete](a0200n3)--(a0200n7);
        \node[fit=(a0200n0) (a0200n1) (a0200n2) (a0200n3) (a0200n4) (a0200n5) (a0200n6) (a0200n7) (a0200n8) ] (c0200) {};

        \node[Feuille](a1010n0)at(-14,-25){};
        \node[NoeudTam](a1010n1)at(-13,-24){};
        \node[Feuille](a1010n2)at(-12,-26){};
        \node[NoeudTam](a1010n3)at(-11,-25){};
        \node[Feuille](a1010n4)at(-10,-26){};
        \draw[Arete](a1010n3)--(a1010n2);
        \draw[Arete](a1010n3)--(a1010n4);
        \draw[Arete](a1010n1)--(a1010n0);
        \draw[Arete](a1010n1)--(a1010n3);
        \node[NoeudTam](a1010n5)at(-9,-23){};
        \node[Feuille](a1010n6)at(-8,-25){};
        \node[NoeudTam](a1010n7)at(-7,-24){};
        \node[Feuille](a1010n8)at(-6,-25){};
        \draw[Arete](a1010n7)--(a1010n6);
        \draw[Arete](a1010n7)--(a1010n8);
        \draw[Arete](a1010n5)--(a1010n1);
        \draw[Arete](a1010n5)--(a1010n7);
        \node[fit=(a1010n0) (a1010n1) (a1010n2) (a1010n3) (a1010n4) (a1010n5) (a1010n6) (a1010n7) (a1010n8) ] (c1010) {};

        \node[Feuille](a0210n0)at(0,-30){};
        \node[NoeudTam](a0210n1)at(1,-29){};
        \node[Feuille](a0210n2)at(2,-30){};
        \draw[Arete](a0210n1)--(a0210n0);
        \draw[Arete](a0210n1)--(a0210n2);
        \node[NoeudTam](a0210n3)at(3,-28){};
        \node[Feuille](a0210n4)at(4,-30){};
        \node[NoeudTam](a0210n5)at(5,-29){};
        \node[Feuille](a0210n6)at(6,-31){};
        \node[NoeudTam](a0210n7)at(7,-30){};
        \node[Feuille](a0210n8)at(8,-31){};
        \draw[Arete](a0210n7)--(a0210n6);
        \draw[Arete](a0210n7)--(a0210n8);
        \draw[Arete](a0210n5)--(a0210n4);
        \draw[Arete](a0210n5)--(a0210n7);
        \draw[Arete](a0210n3)--(a0210n1);
        \draw[Arete](a0210n3)--(a0210n5);
        \node[fit=(a0210n0) (a0210n1) (a0210n2) (a0210n3) (a0210n4) (a0210n5) (a0210n6) (a0210n7) (a0210n8) ] (c0210) {};

        \node[Feuille](a3000n0)at(-28,-29){};
        \node[NoeudTam](a3000n1)at(-27,-28){};
        \node[Feuille](a3000n2)at(-26,-32){};
        \node[NoeudTam](a3000n3)at(-25,-31){};
        \node[Feuille](a3000n4)at(-24,-32){};
        \draw[Arete](a3000n3)--(a3000n2);
        \draw[Arete](a3000n3)--(a3000n4);
        \node[NoeudTam](a3000n5)at(-23,-30){};
        \node[Feuille](a3000n6)at(-22,-31){};
        \draw[Arete](a3000n5)--(a3000n3);
        \draw[Arete](a3000n5)--(a3000n6);
        \node[NoeudTam](a3000n7)at(-21,-29){};
        \node[Feuille](a3000n8)at(-20,-30){};
        \draw[Arete](a3000n7)--(a3000n5);
        \draw[Arete](a3000n7)--(a3000n8);
        \draw[Arete](a3000n1)--(a3000n0);
        \draw[Arete](a3000n1)--(a3000n7);
        \node[fit=(a3000n0) (a3000n1) (a3000n2) (a3000n3) (a3000n4) (a3000n5) (a3000n6) (a3000n7) (a3000n8) ] (c3000) {};

        \node[Feuille](a3100n0)at(27,-35){};
        \node[NoeudTam](a3100n1)at(28,-34){};
        \node[Feuille](a3100n2)at(29,-37){};
        \node[NoeudTam](a3100n3)at(30,-36){};
        \node[Feuille](a3100n4)at(31,-38){};
        \node[NoeudTam](a3100n5)at(32,-37){};
        \node[Feuille](a3100n6)at(33,-38){};
        \draw[Arete](a3100n5)--(a3100n4);
        \draw[Arete](a3100n5)--(a3100n6);
        \draw[Arete](a3100n3)--(a3100n2);
        \draw[Arete](a3100n3)--(a3100n5);
        \node[NoeudTam](a3100n7)at(34,-35){};
        \node[Feuille](a3100n8)at(35,-36){};
        \draw[Arete](a3100n7)--(a3100n3);
        \draw[Arete](a3100n7)--(a3100n8);
        \draw[Arete](a3100n1)--(a3100n0);
        \draw[Arete](a3100n1)--(a3100n7);
        \node[fit=(a3100n0) (a3100n1) (a3100n2) (a3100n3) (a3100n4) (a3100n5) (a3100n6) (a3100n7) (a3100n8) ] (c3100) {};

        \node[Feuille](a3010n0)at(-14,-37){};
        \node[NoeudTam](a3010n1)at(-13,-36){};
        \node[Feuille](a3010n2)at(-12,-39){};
        \node[NoeudTam](a3010n3)at(-11,-38){};
        \node[Feuille](a3010n4)at(-10,-39){};
        \draw[Arete](a3010n3)--(a3010n2);
        \draw[Arete](a3010n3)--(a3010n4);
        \node[NoeudTam](a3010n5)at(-9,-37){};
        \node[Feuille](a3010n6)at(-8,-39){};
        \node[NoeudTam](a3010n7)at(-7,-38){};
        \node[Feuille](a3010n8)at(-6,-39){};
        \draw[Arete](a3010n7)--(a3010n6);
        \draw[Arete](a3010n7)--(a3010n8);
        \draw[Arete](a3010n5)--(a3010n3);
        \draw[Arete](a3010n5)--(a3010n7);
        \draw[Arete](a3010n1)--(a3010n0);
        \draw[Arete](a3010n1)--(a3010n5);
        \node[fit=(a3010n0) (a3010n1) (a3010n2) (a3010n3) (a3010n4) (a3010n5) (a3010n6) (a3010n7) (a3010n8) ] (c3010) {};

        \node[Feuille](a3200n0)at(14,-40){};
        \node[NoeudTam](a3200n1)at(15,-39){};
        \node[Feuille](a3200n2)at(16,-41){};
        \node[NoeudTam](a3200n3)at(17,-40){};
        \node[Feuille](a3200n4)at(18,-43){};
        \node[NoeudTam](a3200n5)at(19,-42){};
        \node[Feuille](a3200n6)at(20,-43){};
        \draw[Arete](a3200n5)--(a3200n4);
        \draw[Arete](a3200n5)--(a3200n6);
        \node[NoeudTam](a3200n7)at(21,-41){};
        \node[Feuille](a3200n8)at(22,-42){};
        \draw[Arete](a3200n7)--(a3200n5);
        \draw[Arete](a3200n7)--(a3200n8);
        \draw[Arete](a3200n3)--(a3200n2);
        \draw[Arete](a3200n3)--(a3200n7);
        \draw[Arete](a3200n1)--(a3200n0);
        \draw[Arete](a3200n1)--(a3200n3);
        \node[fit=(a3200n0) (a3200n1) (a3200n2) (a3200n3) (a3200n4) (a3200n5) (a3200n6) (a3200n7) (a3200n8) ] (c3200) {};

        \node[Feuille](a3210n0)at(0,-44){};
        \node[NoeudTam](a3210n1)at(1,-43){};
        \node[Feuille](a3210n2)at(2,-45){};
        \node[NoeudTam](a3210n3)at(3,-44){};
        \node[Feuille](a3210n4)at(4,-46){};
        \node[NoeudTam](a3210n5)at(5,-45){};
        \node[Feuille](a3210n6)at(6,-47){};
        \node[NoeudTam](a3210n7)at(7,-46){};
        \node[Feuille](a3210n8)at(8,-47){};
        \draw[Arete](a3210n7)--(a3210n6);
        \draw[Arete](a3210n7)--(a3210n8);
        \draw[Arete](a3210n5)--(a3210n4);
        \draw[Arete](a3210n5)--(a3210n7);
        \draw[Arete](a3210n3)--(a3210n2);
        \draw[Arete](a3210n3)--(a3210n5);
        \draw[Arete](a3210n1)--(a3210n0);
        \draw[Arete](a3210n1)--(a3210n3);
        \node[fit=(a3210n0) (a3210n1) (a3210n2) (a3210n3) (a3210n4) (a3210n5) (a3210n6) (a3210n7) (a3210n8) ] (c3210) {};

        \draw[AreteTam] (c0000)--(c1000);
        \draw[AreteTam] (c0000)--(c0100);
        \draw[AreteTam] (c1000)--(c2000);
        \draw[AreteTam] (c0100)--(c2100);
        \draw[AreteTam] (c2000)--(c2100);
        \draw[AreteTam] (c0000)--(c0010);
        \draw[AreteTam] (c0100)--(c0200);
        \draw[AreteTam] (c1000)--(c1010);
        \draw[AreteTam] (c0010)--(c1010);
        \draw[AreteTam] (c0010)--(c0210);
        \draw[AreteTam] (c0200)--(c0210);
        \draw[AreteTam] (c2000)--(c3000);
        \draw[AreteTam] (c2100)--(c3100);
        \draw[AreteTam] (c3000)--(c3100);
        \draw[AreteTam] (c3000)--(c3010);
        \draw[AreteTam] (c1010)--(c3010);
        \draw[AreteTam] (c0200)--(c3200);
        \draw[AreteTam] (c3100)--(c3200);
        \draw[AreteTam] (c3010)--(c3210);
        \draw[AreteTam] (c0210)--(c3210);
        \draw[AreteTam] (c3200)--(c3210);
    \end{tikzpicture}}}}
    \caption{The Tamari lattices $\Tam_3$ and $\Tam_4$. The smallest elements are at the top.}
    \label{fig:TreillisTamari}
\end{figure}
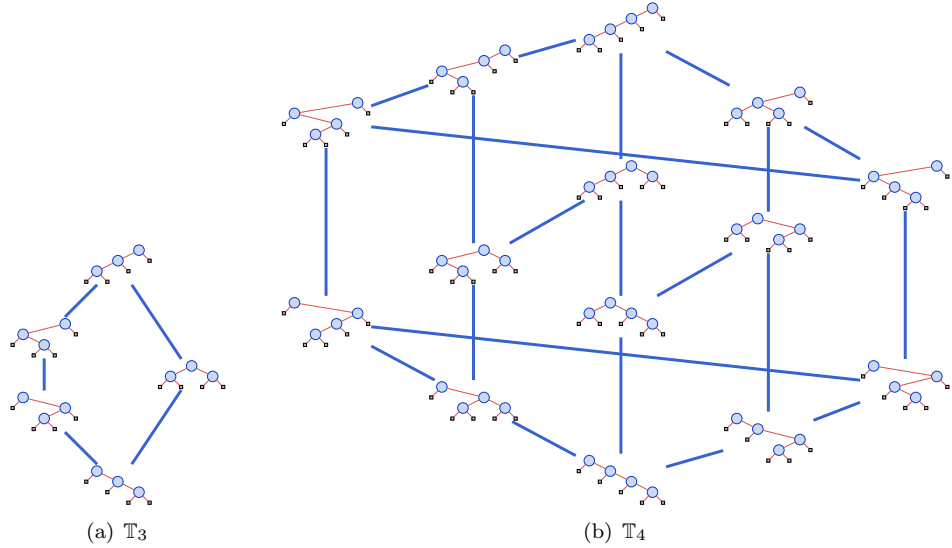

\section{Closure by interval of the set of balanced binary trees} \label{sec:Closure}

\subsection{Rotations and balance} \label{subsec:RotationsEquilibre}

Let us first consider the modifications of the imbalance values of the nodes
of a balanced binary tree $T_0 := (A \ABCons B) \ABCons C$ when a rotation
at its root is applied. Let $T_1$ be the binary tree obtained by this rotation,
$y$ be the root of $T_0$ and $x$ be the left child of $y$ in $T_0$ (see again
Figure~\ref{fig:Rotation}, considering now that $y$ is the root of $T_0$ and $x$
is the root of $T_1$). Note first that the imbalance values of the nodes
of the subtrees $A$, $B$ and $C$ are not modified by this rotation. Indeed,
only the imbalance values of $x$ and $y$ are changed. Since $T_0$ is balanced,
we have $\Des_{T_0}(x) \in \{-1, 0, 1\}$ and $\Des_{T_0}(y) \in \{-1, 0, 1\}$.
Thus, the pair $(\Des_{T_0}(x), \Des_{T_0}(y))$ can take nine different
values. Here follows the list of the imbalance values of $x$ and $y$ in
$T_0$ and $T_1$ expressed as $(\Des_{T_0}(x), \Des_{T_0}(y)) \longrightarrow (\Des_{T_1}(x), \Des_{T_1}(y))$:
\begin{multicols}{3}
    \begin{enumerate}[label = (R\arabic*)]
        \item $(-1, -1) \longrightarrow ({\bf 1}, {\bf 1})$, \vspace*{1em} \label{item:CasRot1}
        \item $(0, -1) \longrightarrow ({\bf 1}, {\bf 0})$, \vspace*{1em} \label{item:CasRot2}
        \item $(0, 0) \longrightarrow (2, {\bf 1})$, \label{item:CasRot3}
        \item $(1, -1) \longrightarrow (2, {\bf 0})$, \vspace*{1em} \label{item:CasRot4}
        \item $(1, 0) \longrightarrow (3, {\bf 1})$, \vspace*{1em} \label{item:CasRot5}
        \item $(-1, 0) \longrightarrow (2, 2)$, \label{item:CasRot6}
        \item $(-1, 1) \longrightarrow (3, 3)$, \vspace*{1em} \label{item:CasRot7}
        \item $(0, 1) \longrightarrow (3, 2)$, \vspace*{1em} \label{item:CasRot8}
        \item $(1, 1) \longrightarrow (4, 2)$. \label{item:CasRot9}
    \end{enumerate}
\end{multicols}
Let us gather these nine sorts of rotations into three different groups,
taking into account if the nodes $x$ and $y$ are balanced in $T_1$.
\begin{itemize}
    \item Cases~\ref{item:CasRot1} and~\ref{item:CasRot2}, where $x$ and
    $y$ stay balanced are called \emph{conservative balancing rotations};
    \item Cases~\ref{item:CasRot3},~\ref{item:CasRot4} and~\ref{item:CasRot5},
    where $y$ stays balanced but $x$ not are called \emph{simply unbalancing rotations};
    \item Cases~\ref{item:CasRot6},~\ref{item:CasRot7},~\ref{item:CasRot8}
    and~\ref{item:CasRot9} where $x$ and $y$ are both unbalanced are called
    \emph{fully unbalancing rotations}.
\end{itemize}
\medskip

This leads to the following properties.
\begin{Proposition} \label{prop:HauteursEgalesEq}
    Let $T_0$ and $T_1$ be two balanced binary trees such that $T_0 \CouvTam T_1$.
    Then, $T_0$ and $T_1$ have the same height.
\end{Proposition}
\begin{proof}
    Since $T_0$ and $T_1$ are both balanced, the rotation modifies a subtree
    $S_0$ of $T_0$ such that the imbalance values of the root $y$ of $S_0$,
    and of the left child $x$ of $y$, satisfy~\ref{item:CasRot1} or~\ref{item:CasRot2}.
    Let $S_1$ be the binary tree obtained by the rotation of root $y$ from
    $S_0$. Computing the height of $S_0$ and $S_1$, we have $\Ht(S_0) = \Ht(S_1)$.
    Thus, since a rotation modifies a binary tree locally, we have $\Ht(T_0) = \Ht(T_1)$.
\end{proof}

\begin{Lemme} \label{lem:RotationDes}
    Let $T_0$ be a balanced binary tree and $T_1$ be an unbalanced binary
    tree such that $T_0 \CouvTam T_1$. Then, there exists a node $z$ in $T_1$
    such that $\Des_{T_1}(z) \geq 2$ and the left subtree and the right
    subtree of $z$ are both balanced.
\end{Lemme}
\begin{proof}
    Let $y$ be the node of $T_0$ which is the root of the rotation that
    transforms $T_0$ into $T_1$ and $x$ its left child in $T_0$. If this
    rotation is a simply unbalancing rotation, it satisfies~\ref{item:CasRot3},~\ref{item:CasRot4}
    or~\ref{item:CasRot5}, and the node $z := x$ satisfies the lemma. If
    this rotation is a fully unbalancing rotation, it satisfies~\ref{item:CasRot6},~\ref{item:CasRot7},~\ref{item:CasRot8}
    or~\ref{item:CasRot9}, and the node $z := y$ of $T_1$ agrees with the
    conclusion of the lemma.
\end{proof}

\begin{Lemme} \label{lem:RotationNonModifDroite}
    Let $T_0$ be a binary tree and $y$ be a node of $T_0$ such that all
    subtrees to the right \Wrt $y$ are balanced. Then, if the binary
    tree $T_1$ is obtained from $T_0$ by a rotation of root $y$, all subtrees
    of $T_1$ to the right \Wrt $y$ are balanced.
\end{Lemme}
\begin{proof}
    Since the rotation operation does not modify the infix order of the
    nodes and by definition of the relation $\ADroite$, if a subtree $S$
    is to the right \Wrt $y$ in $T_1$, then $S$ is also to the right
    \Wrt $y$ in $T_0$. By hypothesis, $S$ is balanced in $T_0$, and
    therefore, it is also balanced in~$T_1$.
\end{proof}

\subsection{Construction of an imbalance invariant} \label{subsec:Invariant}

Let $T$ be a binary tree, $x$ be a node of $T$ and $y$ be the leftmost node
of the subtree of root $x$ in $T$. We say that $x$ is a \emph{witness of imbalance}
if the following three conditions hold (see Figure~\ref{fig:Witness}):
\begin{figure}[ht]
    \centering
    \scalebox{.40}{
    \begin{tikzpicture}
        \node[Noeud, EtiqClair] (0) at (0, 0) {};
        \node[Noeud, EtiqClair] (y) at (-4, -6) {\Huge $y$};
        \node[Noeud, EtiqClair] (3) at (0, -2) {\Huge $x_\ell$};
        \node[Noeud, EtiqFonce, label=left:\huge $\Des_T(x) \geq 2$] (5) at (-2, -4) {\Huge $x$};
        \node (8) at (-3, -1) {};
        \node (9) at (3, -1) {};
        \node[Feuille] (feuille) at (-5, -7) {};
        \node[SArbre, label=right:\Huge $\in \EnsEq$] (Sy) at (-2, -8) {$S_y$};
        \node[SArbre, label=right:\Huge $\in \EnsEq$] (4) at (2, -4){$S_{x_\ell}$};
        \node[SArbre, label=right:\Huge $\in \EnsEq$] (6) at (0, -6){$S_x$};
        \draw[Arete] (y) -- (Sy);
        \draw[Arete] (3) -- (4);
        \draw[Arete] (5) -- (6);
        \draw[Arete] (0) -- (8);
        \draw[Arete] (0) -- (9);
        \draw[Arete] (feuille) -- (y);
        \draw[Arete, decorate, decoration = zigzag] (0) -- (3);
        \draw[Arete, decorate, decoration = zigzag] (3) -- (5);
        \draw[Arete, decorate, decoration = zigzag] (y) -- (5);
        \node at (-6, -3) {\scalebox{2.2}{$T = $}};
    \end{tikzpicture}}
    \caption{The node $x$ is a witness of imbalance of $T$. Note that the left subtree
    of $y$ is empty and thus $S_y$ has $0$ or $1$ node.}
    \label{fig:Witness}
\end{figure}
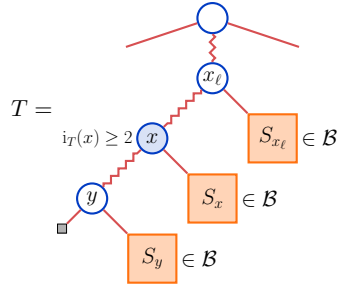
\begin{enumerate}[label = (W\arabic*)]
    \item The imbalance value of $x$ is greater than or equal to $2$; \label{item:W1}
    \item The left subtree of $x$ is balanced; \label{item:W2}
    \item The subtrees of $T$ which are to the right \Wrt $y$ are balanced. \label{item:W3}
\end{enumerate}

\begin{Remarque} \label{rem:TemoinDesImpliqueDes}
    If a binary tree $T$ has a witness of imbalance,~\ref{item:W1}
    guarantees that $T$ is unbalanced.
\end{Remarque}

The aim of this section is to define an additional property that $x$ and
$y$ must satisfy to ensure that any binary tree $T'$ such that $T \OrdTam T'$
has still a witness of imbalance. In this way, by showing that $T'$ also
satisfies this additional property, we will prove that it is impossible
to rebalance $T$ through rotations.
\medskip

Let us already give this property. In what follows, the concepts necessary
to understand it will be defined. If $y$ satisfies condition
\begin{enumerate}[label = (CC)]
    \item the height word of the node $y$ is admissible, \label{item:C}
\end{enumerate}
then, we say that $T$ satisfies the \emph{conservation condition}. Besides,
we say that $T$ has an \emph{imbalance invariant} if $T$ has a witness of
imbalance satisfying the conservation condition.

\subsubsection{Height words}
Let $T$ be a binary tree, $x_1$ be a node of $T$, $(x_1, x_2, \dots, x_\ell)$
be the sequence of all ancestors of $x_1$ whose are to the right \Wrt $x_1$
and ordered from bottom to top, and $(S_{x_i})_{1 \leq i \leq \ell}$
be the sequence of the right subtrees of the $x_i$ (see Figure~\ref{fig:ArbreDroite}).
The word $u_1 \dots u_\ell$ of $\EnsNat^*$ defined by $u_i := \Ht(S_{x_i})$
is called the \emph{height word} of $x_1$ and denoted by $\MotHt_T(x_1)$.
It is convenient to set $\MotHt_T(x) := \epsilon$ whenever $x$ is not a
node of $T$. See Figure~\ref{fig:ExempleHeightWords} for some examples of
height words associated with some nodes of a binary tree.
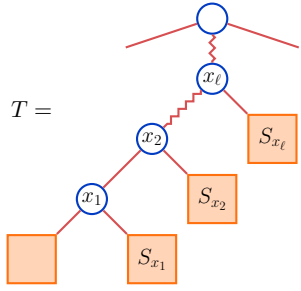
\begin{figure}[ht]
    \centering
    \scalebox{.40}{
    \begin{tikzpicture}
        \node[Noeud, EtiqClair] (0) at (0, 0) {};
        \node[Noeud, EtiqClair] (y) at (-4, -6) {\Huge $x_1$};
        \node[Noeud, EtiqClair] (3) at (0, -2) {\Huge $x_\ell$};
        \node[Noeud, EtiqClair] (5) at (-2, -4) {\Huge $x_2$};
        \node (8) at (-3, -1) {};
        \node (9) at (3, -1) {};
        \node[SArbre] (Sy) at (-2, -8) {$S_{x_1}$};
        \node[SArbre] (4) at (2, -4){$S_{x_\ell}$};
        \node[SArbre] (6) at (0, -6){$S_{x_2}$};
        \node[SArbre] (SAutre) at (-6, -8) {};
        \draw[Arete] (y) -- (Sy);
        \draw[Arete] (SAutre) -- (y);
        \draw[Arete] (3) -- (4);
        \draw[Arete] (5) -- (6);
        \draw[Arete] (0) -- (8);
        \draw[Arete] (0) -- (9);
        \draw[Arete, decorate, decoration = zigzag] (0) -- (3);
        \draw[Arete, decorate, decoration = zigzag] (3) -- (5);
        \draw[Arete] (y) -- (5);
        \node at (-6, -3) {\scalebox{2.2}{$T = $}};
    \end{tikzpicture}}
    \caption{The sequence $(S_{x_i})_{1 \leq i \leq \ell}$ associated with the node $x_1$.}
    \label{fig:ArbreDroite}
\end{figure}
\begin{figure}[ht]
    \centering
    \scalebox{.3}{
    \begin{tikzpicture}
        \node[Feuille](0)at(0,-3){};
        \node[Noeud, label=below:\scalebox{2.5}{$x$}](1)at(1,-2){};
        \node[Feuille](2)at(2,-5){};
        \node[Noeud, label=below:\scalebox{2.5}{$y$}](3)at(3,-4){};
        \node[Feuille](4)at(4,-5){};
        \draw[Arete](3)--(2);
        \draw[Arete](3)--(4);
        \node[Noeud](5)at(5,-3){};
        \node[Feuille](6)at(6,-4){};
        \draw[Arete](5)--(3);
        \draw[Arete](5)--(6);
        \draw[Arete](1)--(0);
        \draw[Arete](1)--(5);
        \node[Noeud](7)at(7,-1){};
        \node[Feuille](8)at(8,-3){};
        \node[Noeud](9)at(9,-2){};
        \node[Feuille](10)at(10,-4){};
        \node[Noeud, label=below:\scalebox{2.5}{$z$}](11)at(11,-3){};
        \node[Feuille](12)at(12,-4){};
        \draw[Arete](11)--(10);
        \draw[Arete](11)--(12);
        \draw[Arete](9)--(8);
        \draw[Arete](9)--(11);
        \draw[Arete](7)--(1);
        \draw[Arete](7)--(9);
        \node[Noeud](13)at(13,0){};
        \node[Feuille](14)at(14,-2){};
        \node[Noeud](15)at(15,-1){};
        \node[Feuille](16)at(16,-2){};
        \draw[Arete](15)--(14);
        \draw[Arete](15)--(16);
        \draw[Arete](13)--(7);
        \draw[Arete](13)--(15);
        \node at (-3, -1.5) {\scalebox{2.8}{$T = $}};
    \end{tikzpicture}}
    \caption{Examples of height words: $\MotHt_T(x) = 221$, $\MotHt_T(y) = 0021$, and $\MotHt_T(z) = 01$.}
    \label{fig:ExempleHeightWords}
\end{figure}

\subsubsection{Admissible words}
Let~$u := u_1 \dots u_n$ be a word. Let us denote by~$\ell(u)$ the
\emph{length}~$n$ of~$u$.
\medskip

Let $\Theta : \EnsNat^2 \rightarrow \EnsNat$ be the rewriting rule defined by
\begin{equation}
    \Theta(a.b) :=
    \begin{cases}
        \max\{a, b\}  + 1 & \mbox{if $b - a \in \{-1, 0, 1\}$,} \\
        \max\{a, b\}      & \mbox{otherwise.}
    \end{cases}
\end{equation}

Note that if $A \ABCons B$ is a balanced binary tree, then
$\Theta(\Ht(A).\Ht(B)) = \Ht(A \ABCons B)$. We shall use this simple
observation to establish the main result of this section.
\medskip

This rewriting rule is extended to words of $\EnsNat^*$ by
$\Theta(u) := \Theta(u_1 . u_2) . u_3 \dots u_{\ell(u)}$. If $0 \leq i \leq \ell(u) - 1$,
denote by $\Theta^i(u)$ the iterated application of $\Theta$ defined by
\begin{equation} \label{eq:DefThetaItere}
    \Theta^i(u) :=
    \begin{cases}
        u                                     & \mbox{if $i = 0$,} \\
        \Theta \left(\Theta^{i - 1}(u)\right) & \mbox{otherwise.}
    \end{cases}
\end{equation}

\begin{Definition} \label{def:AdmissibiliteMots}
    A word $u \in \EnsNat^*$ is \emph{admissible} if either $\ell(u) \leq 1$
    or all words $v$ of the set
    \begin{equation}
        \left\{ \Theta^i(u) : 0 \leq i \leq \ell(u) - 2 \right\} \label{eq:EnsembleTheta}
    \end{equation}
    satisfy $v_1 - 1 \leq v_2$.
\end{Definition}
The set of admissible words is denoted by $\EnsAdmissible$. To check
if a word $u$ is admissible, iteratively compute the elements of the
set~(\ref{eq:EnsembleTheta}) following~(\ref{eq:DefThetaItere}), and check
for each of these the inequality of the previous definition. For example,
by denoting by $\xrightarrow{\Theta}$ the rewriting rule $\Theta$, we can
check that $u := 00122$ is admissible. Indeed, we have
\begin{equation}
    00122 \xrightarrow{\Theta} 1122 \xrightarrow{\Theta} 222 \xrightarrow{\Theta} 32,
\end{equation}
and at each step, the condition $u_1 - 1 \leq u_2$ holds. The word $1234488$
is also admissible:
\begin{equation}
    01233778 \xrightarrow{\Theta} 2233778 \xrightarrow{\Theta}
    333778 \xrightarrow{\Theta} 43778 \xrightarrow{\Theta} 5778
    \xrightarrow{\Theta} 778 \xrightarrow{\Theta} 88.
\end{equation}
On the other hand, $3444$ is not admissible since we have
\begin{equation}
    3444 \xrightarrow{\Theta} 544 \xrightarrow{\Theta} 64,
\end{equation}
and $6 - 1 \nleq 4$.
\medskip

If $u$ is an nonempty word, let us denote by $\HtMot(u)$ the \emph{height}
of $u$, that is the one-letter word $\Theta^{\ell(u) - 1}(u)$. For example,
we have $\HtMot(00122) = 4$, $\HtMot(01233778) = 9$ and $\HtMot(3444) = 6$.
Note that one can deduce from Definition~\ref{def:AdmissibiliteMots} that
a word $u \in \EnsNat^*$ of length greater than $1$ is admissible if and
only if for each decomposition $u = v.a.w$ where $v \in \EnsNat^+$, $a \in \EnsNat$
and $w \in \EnsNat^*$, one has $\HtMot(v) - 1 \leq a$.

\subsubsection{Some properties of admissible words}
Let us establish three lemmas on admissible words that will be helpful later
to prove our main result.

\begin{Lemme} \label{lem:MotAdmCroissant}
    If $u$ is an admissible word, then, for all $1 \leq i \leq \ell(u) - 1$, one has
    $u_i - 1 \leq u_{i + 1}$.
\end{Lemme}
\begin{proof}
    Assume that $u$ is of the form $u = v.u_i.u_{i + 1}.w$ with $v, w \in \EnsNat^*$
    and $u_i - 1 > u_{i + 1}$. Since $\Theta$ changes a word $a.b \in \EnsNat^2$
    into a letter $c \in \EnsNat$ no smaller than both $a$ and $b$, we have
    $\HtMot(v.u_i) \geq u_i$. That implies that $\HtMot(v.u_i) - 1 > u_{i + 1}$,
    showing that $u \notin \EnsAdmissible$ and contradicting the hypothesis.
\end{proof}

\begin{Lemme} \label{lem:MotAdmPrefSuff}
    All prefixes and suffixes of an admissible word are admissible.
\end{Lemme}
\begin{proof}
    It is immediate, by definition, that all prefixes of an admissible word
    also are admissible.

    Let $u \in \EnsAdmissible$ such that $\ell(u) \geq 2$, and $w$ be a nonempty
    suffix of $u$. Assume that $w \notin \EnsAdmissible$. Hence, $w$ is of the
    form $w = x.a.y$ where $x \in \EnsNat^+$, $a \in \EnsNat$, $y \in \EnsNat^*$
    and $\HtMot(x) - 1 > a$. The word $u$ is of the form $u = v.x.a.y$ where
    $v \in \EnsNat^*$. Since $\Theta$ changes a word $a.b \in \EnsNat^2$
    into a letter $c \in \EnsNat$ no smaller than both $a$ and $b$, we have
    $\HtMot(v.x) \geq \HtMot(x)$. Therefore, we have $\HtMot(v.x) - 1 > a$,
    showing that $u \notin \EnsAdmissible$ and contradicting the hypothesis.
\end{proof}

\begin{Lemme} \label{lem:MotAdmSubs}
    If $u.v$ is an admissible word such that $\ell(v) \geq 2$, the word $u.\Theta(v)$
    is still admissible.
\end{Lemme}
\begin{proof}
    If $u$ is empty, the lemma follows immediately. Assume that $u$ is nonempty.
    The word $u.v$ is of the form $u.v = u.a.b.w$ where $a, b \in \EnsNat$
    and $w \in \EnsNat^*$. Set $c := \Theta(a.b) = \HtMot(a.b)$. The word
    $u.c.w = u.\Theta(v)$ is admissible if the two inequalities $\HtMot(u) - 1 \leq c$
    and $\HtMot(u.c) \leq \HtMot(u.a.b)$ hold. Since $u.a.b.w \in \EnsAdmissible$,
    we have $\HtMot(u) - 1 \leq a$, and since $c = \Theta(a.b)$, then $c \geq a$
    and thus, $\HtMot(u) - 1 \leq c$, showing the first inequality. Set $d := \HtMot(u)$.
    The second inequality amounts to prove that $\HtMot(d.c) \leq \HtMot(d.a.b)$,
    which is equivalent to prove $\HtMot(d.\HtMot(a.b)) \leq \HtMot(d.a.b)$.
    This relation holds in general for any letters $a, b, d \in \EnsNat$,
    showing that $u.\Theta(v) \in \EnsAdmissible$.
\end{proof}

\subsubsection{Admissible height words}
Let us prove two lemmas relating admissible words and height words.

\begin{Lemme} \label{lem:MotHauteursEquilibre}
    Let $T$ be a balanced binary tree, $x$ be a node of $T$, and $u$ be
    the height word of $x$. Then $u$ is admissible and $\HtMot(u) \leq \Ht(T)$.
\end{Lemme}
\begin{proof}
    We proceed by structural induction on the set of balanced binary trees.
    The lemma is true for the single element $T$ of the set $\EnsEq_1$ since
    by denoting $x$ its node, we have $u = \MotHt_T(x) = 0$ which is admissible
    and satisfies $0 = \HtMot(u) \leq \Ht(T) = 1$.

    Assume that $T = L \ABCons R$. If $x$ is a node of $R$, we have $u = \MotHt_T(x) = \MotHt_R(x)$,
    and by induction hypothesis, $u \in \EnsAdmissible$ and $\HtMot(u) \leq \Ht(R)$.
    Since $\Ht(R) < \Ht(T)$, the lemma is satisfied.

    If $x$ is a node of $L$, we have $u = \MotHt_T(x) = \MotHt_L(x) . \Ht(R)$.
    Since $T$ is balanced, $\Ht(R) - \Ht(L) \in \{-1, 0, 1\}$, and by
    induction hypothesis, $\HtMot\left(\MotHt_L(x)\right) \leq \Ht(L)$.
    Hence, $\HtMot\left(\MotHt_L(x)\right) - 1 \leq \Ht(R)$. Moreover,
    again by induction hypothesis, $\MotHt_L(x) \in \EnsAdmissible$, and hence,
    $u \in \EnsAdmissible$. Finally, since $\HtMot(u) \leq \Ht(R) + 1 \leq \Ht(T)$,
    the lemma is satisfied.
\end{proof}

\begin{Lemme} \label{lem:MotHauteursSousArbresDroits}
    Let $T$ be a binary tree and $y$ be a node of $T$ such that $\MotHt_T(y)$ is
    admissible and all subtrees of the sequence $(S_{y_i})_{1 \leq i \leq \ell}$
    are balanced. Then, for all node $x$ of $T$ such that $y \ADroite_T x$,
    the word $\MotHt_T(x)$ is admissible.
\end{Lemme}
\begin{proof}
    If $x$ is an ancestor of $y$, since $y \ADroite_T x$, $y$ belongs to
    the left subtree of $x$. Hence, $\MotHt_T(x)$ is a suffix of $\MotHt_T(y)$,
    and by Lemma~\ref{lem:MotAdmPrefSuff}, $\MotHt_T(x) \in \EnsAdmissible$.

    Otherwise, let $S$ be the subtree of $T$ such that $x$ is a node of
    $S$ and the parent of the root of $S$ in $T$ is an ancestor of $y$.
    The height word of $y$ is of the form $\MotHt_T(y) = u.\Ht(S).v$ where
    $u, v \in \EnsNat^*$. Since $y \ADroite_T S$, by hypothesis $S$ is balanced
    and thus by Lemma~\ref{lem:MotHauteursEquilibre}, $\MotHt_S(x) \in \EnsAdmissible$.
    Thanks to Lemma~\ref{lem:MotAdmPrefSuff}, $\Ht(S).v \in \EnsAdmissible$,
    and since, by Lemma~\ref{lem:MotHauteursEquilibre}, $\HtMot(\MotHt_S(x)) \leq \Ht(S)$,
    the word $\MotHt_T(x) = \MotHt_S(x).v$ is admissible too.
\end{proof}

\subsection{The main result}

We give and prove in this section the main result of this paper. For that,
we show through the next two Propositions, that the imbalance invariant
defined in Section~\ref{subsec:Invariant} is appropriate to prove that all
successors of a binary tree obtained from a balanced binary tree by an
unbalancing rotation cannot be rebalanced.
\medskip

Before going further, let us give one example of a binary that satisfies
the conservation condition. Let us consider the following binary tree~$T$~:
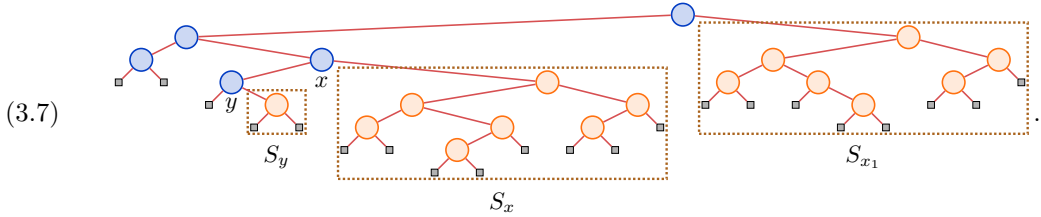
\begin{equation}\begin{split}
    \scalebox{.3}{\begin{tikzpicture}
        \node[Feuille](0)at(0.0,-3){};
        \node[Noeud](1)at(1.0,-2){};
        \node[Feuille](2)at(2.0,-3){};
        \draw[Arete](1)--(0);
        \draw[Arete](1)--(2);
        \node[Noeud](3)at(3.0,-1){};
        \node[Feuille](4)at(4.0,-4){};
        \node[Noeud](5)at(5.0,-3){};
        \node[below of=5]{\scalebox{3}{$y$}};
        \node[Feuille](6)at(6.0,-5){};
        \node[Noeud,Marque1](7)at(7.0,-4){};
        \node[Feuille](8)at(8.0,-5){};
        \node[draw=Marron!100,line width=3pt,dashed,fit=(6) (7) (8)] (Sy) {};
        \node[below of=Sy, node distance=2cm]{\scalebox{3}{$S_y$}};
        \draw[Arete](7)--(6);
        \draw[Arete](7)--(8);
        \draw[Arete](5)--(4);
        \draw[Arete](5)--(7);
        \node[Noeud](9)at(9.0,-2){};
        \node[below of=9]{\scalebox{3}{$x$}};
        \node[Feuille](10)at(10.0,-6){};
        \node[Noeud,Marque1](11)at(11.0,-5){};
        \node[Feuille](12)at(12.0,-6){};
        \draw[Arete](11)--(10);
        \draw[Arete](11)--(12);
        \node[Noeud,Marque1](13)at(13.0,-4){};
        \node[Feuille](14)at(14.0,-7){};
        \node[Noeud,Marque1](15)at(15.0,-6){};
        \node[Feuille](16)at(16.0,-7){};
        \draw[Arete](15)--(14);
        \draw[Arete](15)--(16);
        \node[Noeud,Marque1](17)at(17.0,-5){};
        \node[Feuille](18)at(18.0,-6){};
        \draw[Arete](17)--(15);
        \draw[Arete](17)--(18);
        \draw[Arete](13)--(11);
        \draw[Arete](13)--(17);
        \node[Noeud,Marque1](19)at(19.0,-3){};
        \node[Feuille](20)at(20.0,-6){};
        \node[Noeud,Marque1](21)at(21.0,-5){};
        \node[Feuille](22)at(22.0,-6){};
        \draw[Arete](21)--(20);
        \draw[Arete](21)--(22);
        \node[Noeud,Marque1](23)at(23.0,-4){};
        \node[Feuille](24)at(24.0,-5){};
        \draw[Arete](23)--(21);
        \draw[Arete](23)--(24);
        \draw[Arete](19)--(13);
        \draw[Arete](19)--(23);
        \draw[Arete](9)--(5);
        \draw[Arete](9)--(19);
        \draw[Arete](3)--(1);
        \draw[Arete](3)--(9);
        \node[draw=Marron!100,line width=3pt,dashed,fit=(10) (11) (12) (13) (14) (15) (16) (17) (18) (19) (20) (21) (22) (23) (24)] (Sx) {};
        \node[below of=Sx, node distance=3.5cm]{\scalebox{3}{$S_x$}};
        \node[Noeud](25)at(25.0,0){};
        \node[Feuille](26)at(26.0,-4){};
        \node[Noeud,Marque1](27)at(27.0,-3){};
        \node[Feuille](28)at(28.0,-4){};
        \draw[Arete](27)--(26);
        \draw[Arete](27)--(28);
        \node[Noeud,Marque1](29)at(29.0,-2){};
        \node[Feuille](30)at(30.0,-4){};
        \node[Noeud,Marque1](31)at(31.0,-3){};
        \node[Feuille](32)at(32.0,-5){};
        \node[Noeud,Marque1](33)at(33.0,-4){};
        \node[Feuille](34)at(34.0,-5){};
        \draw[Arete](33)--(32);
        \draw[Arete](33)--(34);
        \draw[Arete](31)--(30);
        \draw[Arete](31)--(33);
        \draw[Arete](29)--(27);
        \draw[Arete](29)--(31);
        \node[Noeud,Marque1](35)at(35.0,-1){};
        \node[Feuille](36)at(36.0,-4){};
        \node[Noeud,Marque1](37)at(37.0,-3){};
        \node[Feuille](38)at(38.0,-4){};
        \draw[Arete](37)--(36);
        \draw[Arete](37)--(38);
        \node[Noeud,Marque1](39)at(39.0,-2){};
        \node[Feuille](40)at(40.0,-3){};
        \draw[Arete](39)--(37);
        \draw[Arete](39)--(40);
        \draw[Arete](35)--(29);
        \draw[Arete](35)--(39);
        \draw[Arete](25)--(3);
        \draw[Arete](25)--(35);
        \node[draw=Marron!100, line width=3pt,dashed,fit=(26) (27) (28) (29) (30) (31) (32) (33) (34) (35) (36) (37) (38) (39) (40)] (Sx1) {};
        \node[below of=Sx1, node distance=3.5cm]{\scalebox{3}{$S_{x_1}$}};
    \end{tikzpicture}}
\end{split}\,.\end{equation}
One observes that the imbalance value of the node $x$ is $2$, that the left subtree
of $x$ is balanced, and that the subtrees to the right \Wrt $y$, namely
$S_y$, $S_x$, and $S_{x_1}$ are balanced. Hence, $x$ satisfies~\ref{item:W1},
\ref{item:W2}, and~\ref{item:W3} and is a witness
of imbalance of~$T$. Moreover, one has $\MotHt_T(y) = 144$. Since $144$
is an admissible word, $T$ satisfies the conservation condition~\ref{item:C}
and hence, has an imbalance invariant.

\begin{Proposition} \label{prop:ClotureCasInit}
    Let $T_0$ be a balanced binary tree and $T_1$ be an unbalanced binary
    tree such that $T_0 \CouvTam T_1$. Then, $T_1$ has an imbalance invariant.
\end{Proposition}
\begin{proof}
    Let $S_0 := (A \ABCons B) \ABCons C$ be the subtree of $T_0$ modified
    by the rotation transforming $T_0$ into $T_1$ and $S_1 := A \ABCons (B \ABCons C)$
    be the resulting subtree in $T_1$. Denote by $r$ the root of this rotation
    and by $q$ the left child of $r$ in $S_0$ (see Figure~\ref{fig:PreuveThmCasInitial}).
    We shall exhibit, in the rest of this proof, a witness of imbalance~$x$ of~$T_1$
    that satisfies the conservation condition.
    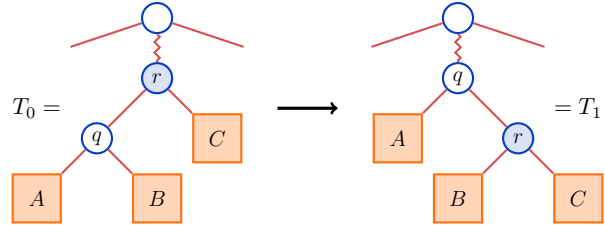
\begin{figure}[ht]
        \centering
        \scalebox{.4}{
        \begin{tikzpicture}
            \node[Noeud, EtiqClair] (racine) at (0, 0) {};
            \node (g) at (-3, -1) {};
            \node (d) at (3, -1) {};
            \node[Noeud, EtiqFonce] (r) at (0, -2) {\Huge $r$};
            \node[Noeud, EtiqClair] (q) at (-2, -4) {\Huge $q$};
            \node[SArbre] (A) at (-4, -6) {$A$};
            \node[SArbre] (B) at (0, -6) {$B$};
            \node[SArbre] (C) at (2, -4) {$C$};
            \draw[Arete] (racine) -- (g);
            \draw[Arete] (racine) -- (d);
            \draw[Arete, decorate, decoration = zigzag] (racine) -- (r);
            \draw[Arete] (r) -- (q);
            \draw[Arete] (r) -- (C);
            \draw[Arete] (q) -- (A);
            \draw[Arete] (q) -- (B);
            \node at (-4, -3) {\scalebox{2.2}{$T_0 = $}};
            \draw[line width=3pt, ->] (4, -3) -- (6, -3);
            \node[Noeud, EtiqClair] (racine') at (10, 0) {};
            \node (g') at (7, -1) {};
            \node (d') at (13, -1) {};
            \node[Noeud, EtiqFonce] (r') at (12, -4) {\Huge $r$};
            \node[Noeud, EtiqClair] (q') at (10, -2) {\Huge $q$};
            \node[SArbre] (A') at (8, -4) {$A$};
            \node[SArbre] (B') at (10, -6) {$B$};
            \node[SArbre] (C') at (14, -6) {$C$};
            \draw[Arete] (racine') -- (g');
            \draw[Arete] (racine') -- (d');
            \draw[Arete, decorate, decoration = zigzag] (racine') -- (q');
            \draw[Arete] (q') -- (r');
            \draw[Arete] (q') -- (A');
            \draw[Arete] (r') -- (B');
            \draw[Arete] (r') -- (C');
            \node at (14, -3) {\scalebox{2.2}{$ = T_1$}};
        \end{tikzpicture}}
        \caption{The initial case, an unbalancing rotation at root $r$ is
        performed into the balanced binary tree $T_0$.}
        \label{fig:PreuveThmCasInitial}
    \end{figure}
    By Lemma~\ref{lem:RotationDes}, $q$ or $r$ is unbalanced in $T_1$ and
    has a positive imbalance value. Therefore, we have to consider two
    cases, depending on the sort of unbalancing rotation which transforms
    $T_0$ into $T_1$.
    \begin{enumerate}[label = \underline{\bf Case \arabic*:}, fullwidth]
        \item If it is a simply unbalancing rotation, set $x := q$ and $y$
        as the leftmost node of the subtree of root $q$ in $T_1$. Since
        $\Des_{T_1}(x) \geq 2$,~\ref{item:W1} checks out. Moreover,
        since $T_0$ is balanced, by Lemma~\ref{lem:RotationNonModifDroite},
        the subtrees to the right \Wrt $r$ are balanced in $T_1$, and since
        $A$ and $B \ABCons C$ are balanced,~\ref{item:W2} and~\ref{item:W3}
        are established. Finally, since $T_0$ is balanced, Lemma~\ref{lem:MotHauteursEquilibre}
        shows that $\MotHt_{T_0}(y)$ is admissible. We have
        \begin{equation}
            \MotHt_{T_0}(y) = \MotHt_A(y) . \Ht(B) . \Ht(C) . v,
        \end{equation}
        where $v \in \EnsNat^*$. Besides, we have
        \begin{equation}
            \MotHt_{T_1}(y) = \MotHt_A(y) . \Ht(B \ABCons C) . v = \MotHt_A(y) . \Theta(\Ht(B) . \Ht(C)) . v,
        \end{equation}
        since $B \ABCons C$ is balanced. Hence, we have
        $\MotHt_{T_1}(y) = \MotHt_A(y) . \Theta(\Ht(B) . \Ht(C) . v)$, and
        since $\MotHt_A(y) . \Ht(B) . \Ht(C) . v$ is admissible,
        by Lemma~\ref{lem:MotAdmSubs}, $\MotHt_{T_1}(y)$ also is. That shows
        that~\ref{item:C} is satisfied.
        \smallskip

        \item Assume that the rotation is fully unbalancing. Set $x := r$
        and $y$ as the leftmost node of the subtree of root $r$ in $T_1$.
        Since $\Des_{T_1}(x) \geq 2$,~\ref{item:W1} checks out. Moreover,
        since $T_0$ is balanced, by Lemma~\ref{lem:RotationNonModifDroite},
        the subtrees to the right w.r.t $r$ are balanced in $T_1$, and since
        $B$ is balanced,~\ref{item:W2} and~\ref{item:W3} are established.
        Finally, since $T_0$ is balanced, Lemma~\ref{lem:MotHauteursEquilibre}
        shows that $\MotHt_{T_0}(y)$ is admissible. We have
        \begin{equation}
            \MotHt_{T_0}(y) = \MotHt_B(y) . \Ht(C) . v,
        \end{equation}
        where $v \in \EnsNat^*$. Besides,
        \begin{equation}
            \MotHt_{T_1}(y) = \MotHt_B(y) . \Ht(C) . v,
        \end{equation}
        and hence $\MotHt_{T_1}(y) = \MotHt_{T_0}(y)$, so that~\ref{item:C}
        checks out.
    \end{enumerate}
    \smallskip

    Thereby, we have shown that there exists a node $x$ in $T_1$ that is
    a witness of imbalance and satisfies the conservation condition in all
    case.
\end{proof}

\begin{Proposition} \label{prop:ClotureConservation}
    Let $T_1$ and $T_2$ be two binary trees such that~$T_1 \CouvTam T_2$
    and~$T_1$ has an imbalance invariant. Then,~$T_2$ has an imbalance
    invariant.
\end{Proposition}
\begin{proof}
    Let $x$ be a witness of imbalance of $T_1$ that satisfies the
    conservation condition, $y$ be the leftmost node of the subtree of root
    $x$ in $T_1$, $r$ be the root of the rotation that transforms $T_1$
    into $T_2$, and $q$ be the left child of $r$ in $T_1$. For all relative
    position of $r$ \Wrt $y$ in $T_1$, we shall exhibit a witness of imbalance
    $x'$ of $T_2$ that satisfies the conservation condition. If necessary,
    we shall also exhibit the node $y'$ of $T_2$ that is the leftmost node
    of the subtree of root $x'$.
    \smallskip

    There are exactly three cases to consider. Note first that since one
    can perform a rotation of root $r$, $r$ has a left son, and since $y$
    has no left son, $r \ne y$. The first case occurs when $r$ is to the
    left \Wrt $y$ (Case~1). Otherwise, when $r$ is to the right \Wrt $y$,
    the second case occurs when $r$ is a strict ancestor of $y$ (Case~2).
    In this case, $y$ is in the left subtree of $r$. Otherwise, when $r$
    is to the right \Wrt $y$ and $r$ is not a strict ancestor of $y$, the
    third case occurs (Case~3). In this last case, the subtree of root $r$
    is to the right \Wrt $y$.

    \begin{enumerate}[label = \underline{\bf Case \arabic*:}, fullwidth]
        \item If $r$ is to the left \Wrt $y$, the rotation of root $r$
        does not modify any of the subtrees to the right \Wrt $y$. Thus,
        $x' := x$ is a witness of imbalance of $T_2$ and satisfies the
        conservation condition.
        \smallskip

        \item If $r$ and $q$ are both ancestors of $y$ in $T_1$,
        set $C$ as the right subtree of $r$ and $B$ as the right subtree
        of $q$ in $T_1$. In this case, $T_2$ is obtained from $T_1$ by replacing
        the subtrees $B$ and $C$ by $B \ABCons C$ as shown in Figure~\ref{fig:PreuveThmCas2}.
        \begin{figure}[ht]
            \centering
            \scalebox{.40}{
            \begin{tikzpicture}
                \node[Noeud, EtiqClair](0) at (0, 0){};
                \node (8) at (-3, -1) {};
                \node (9) at (3, -1) {};
                \draw[Arete] (0) -- (8);
                \draw[Arete] (0) -- (9);
                \node[Noeud, EtiqFonce] (r) at (0, -2) {\Huge $r$};
                \draw[Arete, decorate, decoration = zigzag] (0) -- (r);
                \node[Noeud, EtiqClair] (fr) at (-4, -4) {\Huge $q$};
                \draw[Arete] (r) -- (fr);
                \node[Noeud, EtiqClair] (1) at (-4, -8){\Huge $y$};
                \draw[Arete, decorate, decoration = zigzag] (fr) -- (1);
                \node[SArbre] (B) at (2, -4){$C$};
                \node[SArbre] (A) at (-2, -6){$B$};
                \draw[Arete] (r) -- (B);
                \draw[Arete] (fr) -- (A);
                \node[SArbre] (Sx) at (-2, -10){};
                \draw[Arete] (1) -- (Sx);
                \node[Feuille] (feuille) at (-5, -9) {};
                \draw[Arete] (feuille) -- (1);
                \node at (-6, -3) {\scalebox{2.2}{$T_1 = $}};
                \draw[line width=3pt, ->] (4, -4) -- (7, -4);
                \node[Noeud, EtiqClair](0') at (10, 0){};
                \node (8') at (7, -1) {};
                \node (9') at (13, -1) {};
                \node[Noeud, EtiqFonce] (r') at (14, -4) {\Huge $r$};
                \node[Noeud, EtiqClair] (1') at (10, -8){\Huge $y$};
                \draw[Arete, decorate, decoration = zigzag] (0') -- (1');
                \node[Noeud, EtiqClair] (fr') at (10, -2) {\Huge $q$};
                \node[SArbre] (B') at (16, -6){$C$};
                \node[SArbre] (A') at (12, -6){$B$};
                \node[SArbre] (Sx') at (12, -10){};
                \draw[Arete] (0') -- (8');
                \draw[Arete] (0') -- (9');
                \draw[Arete] (r') -- (B');
                \draw[Arete] (r') -- (A');
                \draw[Arete] (r') -- (fr');
                \draw[Arete] (1') -- (Sx');
                \node[Feuille] (feuille') at (9, -9) {};
                \draw[Arete] (feuille') -- (1');
                \node at (16, -3) {\scalebox{2.2}{$ = T_2$}};
            \end{tikzpicture}}
            \caption{The second case, $r$ is an ancestor of $y$ and $y \ADroite_{T_1} r$.}
            \label{fig:PreuveThmCas2}
        \end{figure}
        We have now three possibilities whether $B \ABCons C$ is balanced
        and $r$ is an ancestor of $x$ in $T_1$.
        \smallskip

        \begin{enumerate}[label = \underline{\bf Case \arabic{enumi}.\arabic*:}, fullwidth]
            \item If $B \ABCons C$ is unbalanced, set $x' := r$ and $y'$
            as the leftmost node of $B \ABCons C$. One has
            \begin{equation}
                \MotHt_{T_1}(y) = u . \Ht(B) . \Ht(C) . v,
            \end{equation}
            where $u, v \in \EnsNat^*$.
            Since $x$ satisfies the conservation condition in $T_1$,
            $\MotHt_{T_1}(y) \in \EnsAdmissible$. Thus, by Lemma~\ref{lem:MotAdmCroissant},
            we have $\Ht(B) - 1 \leq \Ht(C)$ so that $\Des_{T_2}(x') \geq 2$
            and~\ref{item:W1} is satisfied. Moreover, since $B$ is balanced,
            and by Lemma~\ref{lem:RotationNonModifDroite}, all subtrees
            to the right \Wrt $x'$ are also balanced in $T_2$,~\ref{item:W2}
            and~\ref{item:W3} are established. Finally, by Lemma~\ref{lem:MotHauteursSousArbresDroits},
            $\MotHt_{T_1}(y') \in \EnsAdmissible$, and since $\MotHt_{T_2}(y') = \MotHt_{T_1}(y')$,~\ref{item:C}
            checks out.
            \smallskip

            \item If $B \ABCons C$ is balanced and $r$ is an ancestor of
            $x$ in $T_1$, set $x' := x$ and $y' := y$. One clearly has $\Des_{T_2}(x') \geq 2$,
            so that~\ref{item:W1} is satisfied. Moreover, since the left subtree
            of $x'$ in $T_2$ is not modified by the rotation and hence stays
            balanced, since $B \ABCons C$ is balanced, and since by Lemma~\ref{lem:RotationNonModifDroite},
            all subtrees to the right \Wrt $r$ are balanced in $T_2$,~\ref{item:W2}
            and~\ref{item:W3} check out. Finally, since $x$ satisfies
            the conservation condition in $T_1$, $\MotHt_{T_1}(y) \in \EnsAdmissible$
            and we have
            \begin{equation}
                \MotHt_{T_1}(y) = u . \Ht(B) . \Ht(C) . v,
            \end{equation}
            where $u, v \in \EnsNat^*$. Besides,
            \begin{equation}
                \MotHt_{T_2}(y') = u . \Ht(B \ABCons C) . v = u . \Theta(\Ht(B) . \Ht(C)) . v,
            \end{equation}
            since $B \ABCons C$ is balanced. Thus, by Lemma~\ref{lem:MotAdmSubs},
            $\MotHt_{T_2}(y') \in \EnsAdmissible$, so that~\ref{item:C}
            is satisfied.
            \smallskip

            \item If $B \ABCons C$ is balanced and $r$ is a descendant of
            $x$ in $T_1$, we have two possibilities whether $q$ is balanced
            in $T_2$. If it is, set $x' := x$. By Proposition~\ref{prop:HauteursEgalesEq},
            the left subtree of $x'$ stays balanced in $T_2$ and $\Des_{T_2}(x') \geq 2$.
            Thus,~\ref{item:W1} and~\ref{item:W2} are satisfied. Moreover,
            by Lemma~\ref{lem:RotationNonModifDroite}, all subtrees to the
            right \Wrt $x'$ stay balanced in $T_2$ so that~\ref{item:W3}
            checks out. Otherwise, if $q$ is not balanced, set $x' := q$.
            Since the left subtree of $x$ is balanced in $T_1$, by
            Lemma~\ref{lem:RotationDes}, $\Des_{T_2}(x') \geq 2$, and~\ref{item:W1} holds.
            Moreover, $q$ belongs to the left subtree of $x$ in $T_1$ which
            is balanced, and hence, the left subtree of $q$ is balanced in $T_2$,
            so that~\ref{item:W2} holds. Since $B \ABCons C$ is balanced
            and by Lemma~\ref{lem:RotationNonModifDroite},~\ref{item:W3}
            also holds. Set now for both cases $y'$ as the leftmost node
            of the subtree of root $x'$ in $T_2$. The word $\MotHt_{T_2}(y')$
            satisfies exactly same conditions as in the previous case, so that~\ref{item:C}
            is satisfied.
        \end{enumerate}
        \smallskip

        \item If the subtree $S_1 := (A \ABCons B) \ABCons C$ of root $r$
        in $T_1$ is to the right \Wrt $y$, set $S_2 := A \ABCons (B \ABCons C)$
        as the subtree of $T_2$ obtained by the rotation at root $r$
        which transforms $T_1$ into $T_2$ (see Figure~\ref{fig:PreuveThmCas3}).
        \begin{figure}[ht]
            \centering
            \scalebox{.40}{
            \begin{tikzpicture}
                \node[Noeud, EtiqClair](racine) at (0, 0){};
                \node (g) at (-3, -1) {};
                \node (d) at (3, -1) {};
                \draw[Arete] (racine) -- (g);
                \draw[Arete] (racine) -- (d);
                \node[Noeud, EtiqClair] (y) at (0, -9) {\Huge $y$};
                \draw[Arete, decorate, decoration = zigzag] (racine) -- (y);
                \node[Noeud, EtiqClair] (t1) at (0, -2) {};
                \node[SArbre] (Dy) at (2, -11){};
                \draw[Arete] (y) -- (Dy);
                \node[Feuille] (feuille) at (-1, -10) {};
                \draw[Arete] (y) -- (feuille);
                \node[Noeud, EtiqClair] (t2) at (5, -3) {};
                \draw[Arete] (t1) -- (t2);
                \node[Noeud, EtiqFonce] (r) at (5, -5) {\Huge $r$};
                \node[Noeud, EtiqClair] (p) at (3, -6) {\Huge $q$};
                \draw[Arete, decorate, decoration = zigzag] (t2) -- (r);
                \node[SArbre] (A) at (2, -8){$A$};
                \node[SArbre] (B) at (4, -8){$B$};
                \node[SArbre] (C) at (6, -7){$C$};
                \draw[Arete] (p) -- (A);
                \draw[Arete] (p) -- (B);
                \draw[Arete] (r) -- (C);
                \draw[Arete] (p) -- (r);
                \node at (-3, -4) {\scalebox{2.2}{$T_1 = $}};
                \draw[line width=3pt, ->] (8, -4) -- (11, -4);
                \node[Noeud, EtiqClair](racine') at (14, 0){};
                \node (g') at (11, -1) {};
                \node (d') at (17, -1) {};
                \draw[Arete] (racine') -- (g');
                \draw[Arete] (racine') -- (d');
                \node[Noeud, EtiqClair] (y') at (14, -9) {\Huge $y$};
                \draw[Arete, decorate, decoration = zigzag] (racine') -- (y');
                \node[Noeud, EtiqClair] (t1') at (14, -2) {};
                \node[SArbre] (Dy') at (16, -11){};
                \draw[Arete] (y') -- (Dy');
                \node[Feuille] (feuille') at (13, -10) {};
                \draw[Arete] (y') -- (feuille');
                \node[Noeud, EtiqClair] (t2') at (19, -3) {};
                \draw[Arete] (t1') -- (t2');
                \node[Noeud, EtiqFonce] (r') at (21, -6) {\Huge $r$};
                \node[Noeud, EtiqClair] (p') at (19, -5) {\Huge $q$};
                \draw[Arete, decorate, decoration = zigzag] (t2') -- (p');
                \node[SArbre] (A') at (18, -7){$A$};
                \node[SArbre] (B') at (20, -8){$B$};
                \node[SArbre] (C') at (22, -8){$C$};
                \draw[Arete] (p') -- (A');
                \draw[Arete] (r') -- (B');
                \draw[Arete] (r') -- (C');
                \draw[Arete] (p') -- (r');
                \node at (22, -4) {\scalebox{2.2}{$ = T_2$}};
            \end{tikzpicture}}
            \caption{The third case, $r$ is a node of a subtree $S_1$ of $T_1$ satisfying $y \ADroite_{T_1} S_1$.}
            \label{fig:PreuveThmCas3}
        \end{figure}
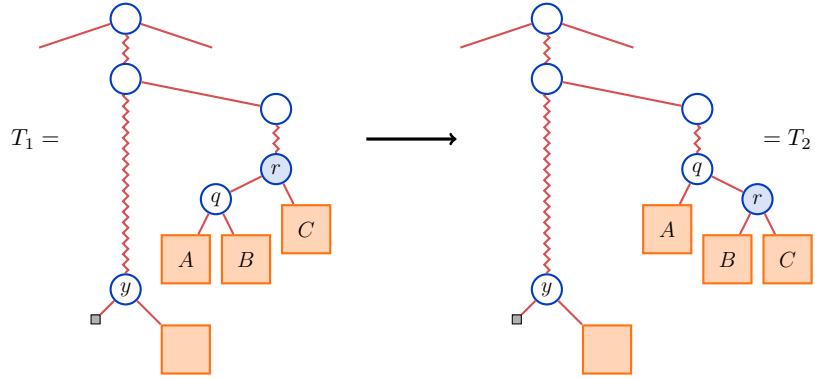
        We have now two cases to consider whether $S_2$ is balanced or not.
        \smallskip

        \begin{enumerate}[label = \underline{\bf Case \arabic{enumi}.\arabic*:}, fullwidth]
            \item If $S_2$ is balanced, by Proposition~\ref{prop:HauteursEgalesEq},
            $\Ht(S_2) = \Ht(S_1)$, and by setting $x' := x$ and $y' := y$
            one has $\Des_{T_2}(x') = \Des_{T_1}(x)$ so that~\ref{item:W1}
            is satisfied. Moreover, the left subtree of $x'$ stays balanced,
            and by Lemma~\ref{lem:RotationNonModifDroite}, the subtrees to the
            right \Wrt $x'$ in $T_2$ also, so that~\ref{item:W2} and~\ref{item:W3}
            check out. Finally, $x'$ also satisfies~\ref{item:C} in $T_2$
            since $\MotHt_{T_2}(y') = \MotHt_{T_1}(y)$.
            \smallskip

            \item If $S_2$ is not balanced, by Proposition~\ref{prop:ClotureCasInit},
            there exists a node $x'$ in $S_2$ which is a witness of imbalance
            satisfying the conservation condition, locally in $S_2$.
            Therefore, $x'$ satisfies~\ref{item:W1} and~\ref{item:W2}
            in $T_2$. It also satisfies~\ref{item:W3} in $T_2$ since, by
            Lemma~\ref{lem:RotationNonModifDroite}, the subtrees of $T_2$
            to the right \Wrt $r$ stay balanced. It remains to prove
            that $x'$ satisfies the conservation condition in the whole
            binary tree $T_2$. Set $y'$ as the leftmost node of the subtree
            of root $x'$ in $T_2$. By Proposition~\ref{prop:ClotureCasInit},
            $w := \MotHt_{S_2}(y') \in \EnsAdmissible$, and by Lemma~\ref{lem:MotHauteursEquilibre},
            $w$ satisfies $\HtMot(w) \leq \Ht(S_1)$. By hypothesis,
            $\MotHt_{T_1}(y) \in \EnsAdmissible$ and one has
            \begin{equation}
                \MotHt_{T_1}(y) = u . \Ht(S_1) . v,
            \end{equation}
            where $u, v \in \EnsNat^*$. Besides, since
            \begin{equation}
                \MotHt_{T_2}(y') = w . v,
            \end{equation}
            one has $\MotHt_{T_2}(y') \in \EnsAdmissible$, establishing~\ref{item:C}.
        \end{enumerate}
    \end{enumerate}
    \smallskip

    Thereby, we have shown that there exists a node $x'$ in $T_2$ that is
    a witness of imbalance and satisfies the conservation condition in all
    case.
\end{proof}

\begin{Theoreme} \label{thm:ClotureIntArbresEq}
    Let $T$ and $T'$ be two balanced binary trees such that $T \OrdTam T'$.
    Then, the interval $[T, T']$ only contains balanced binary trees. In
    other words, all successors of a binary tree obtained by an unbalancing
    rotation from a balanced binary tree are unbalanced.
\end{Theoreme}
\begin{proof}
    Let $T_0$ and $T_2$ be two balanced binary trees and $T_1$ be an
    unbalanced binary tree. Assume that
    \begin{equation}
        T_0 \CouvTam \cdots \CouvTam T_1 \CouvTam \cdots \CouvTam T_2.
    \end{equation}
    By Proposition~\ref{prop:ClotureCasInit}, $T_1$ satisfies the
    conservation condition. Moreover, by Proposition~\ref{prop:ClotureConservation},
    $T_2$ also satisfies the conservation condition. Hence, $T_2$ has
    a witness of imbalance and by Remark~\ref{rem:TemoinDesImpliqueDes},
    $T_2$ is unbalanced. This is contradictory with our hypothesis.
    \smallskip

    Therefore, the notion of imbalance invariant
    defined in Section~\ref{subsec:Invariant} is appropriate and hence
    the set of balanced binary trees is closed by interval in the Tamari lattice.
\end{proof}

\section{Synchronous grammars} \label{sec:GramSync}

In this section, we introduce synchronous grammars. These grammars allow
to generate planar rooted tree-like structures by allowing these to grow
from the root to the leaves step by step. Such trees grow from a single
node, the root, and by simultaneously substituting its nodes with no children
by new tree-like structures following some fixed substitution rules.
\smallskip

As we shall see, synchronous grammars are convenient tools to enumerate
some specified families of planar rooted tree-like structures. Indeed,
one can extract a fixed-point functional equation for the generating series
enumerating the specified objects from a synchronous grammar subject to
two precise conditions that we shall expose. We also present an algorithm to
compute the coefficients of this generating series.

\subsection{Definitions}

\subsubsection{Bud trees}
\begin{Definition}
    Let $B$ be a nonempty finite alphabet. A $B$-\emph{bud tree}, or simply
    a \emph{bud tree} if $B$ is fixed, is a nonempty incomplete rooted
    planar tree where the leaves, namely the \emph{buds}, are labeled on $B$.
\end{Definition}
Set for the sequel $B := \{\LettreB_1, \dots, \LettreB_k\}$ as a nonempty
finite alphabet. Denote by $\EnsBud_n$ the set of $B$-bud trees with $n$
buds and by $\EnsBud$ the set of all $B$-bud trees. The set of all buds of a
bud tree $D$ is denoted by $\Buds(D)$ and the \emph{frontier} of $D$ is
the sequence $(b_1, \dots, b_n)$ of its buds, read from left to right.
If $b$ is a bud, we shall denote by $\Eval(b)$ the \emph{evaluation} of
$b$, that is the element of $B$ labeling $b$. Moreover, the \emph{evaluation}
$\Eval(D)$ of $D$ is the monomial of $\EnsRel [B]$ defined by
\begin{equation}
    \Eval(D) := \prod_{b \in \Buds(D)} \Eval(b).
\end{equation}
For example,
\begin{equation}
    \Eval \left(
    \begin{split}
    \scalebox{0.45}{
    \begin{tikzpicture}
        \node[Noeud] (r) at (0, .25) {};
        \node[Bourgeon] (1) at (0, -1) {\huge $z$};
        \draw[Arete] (r) -- (1);
    \end{tikzpicture}}
    \end{split}
    \right) = z
    \qquad \mbox{and} \qquad
    \Eval \left(
    \begin{split}
    \scalebox{0.45}{
    \begin{tikzpicture}
        \node[Noeud] (r) at (0, .25) {};
        \node[Noeud] (fr) at (1, -1) {};
        \node[Bourgeon] (1) at (-1, -1) {\huge $x$};
        \node[Bourgeon] (2) at (.35, -2.25) {\huge $y$};
        \node[Bourgeon] (3) at (1.65, -2.25) {\huge $x$};
        \draw[Arete] (r) -- (fr);
        \draw[Arete] (r) -- (1);
        \draw[Arete] (fr) -- (2);
        \draw[Arete] (fr) -- (3);
    \end{tikzpicture}}
    \end{split}
    \right) = x^2 y.
\end{equation}

\subsubsection{Synchronous grammars}
\begin{Definition}
    A \emph{synchronous grammar} $S$ is a triple $(B, a, R)$ where:
    \begin{itemize}
        \item $B$ is a nonempty alphabet, the \emph{set of bud labels};
        \item $a$ is a bud labeled on $B$, the \emph{axiom} of $S$;
        \item $R \subseteq B \times \EnsBud$ is a finite set such that for
        all $\LettreB \in B$, there is at least one bud tree $D$ such
        that $(\LettreB, D) \in R$. This is the set of \emph{substitution rules}
        of $S$.
    \end{itemize}
\end{Definition}

Let $S := (B, a, R)$ be a synchronous grammar. For the sake of readability,
we will make use of the following notation for substitution rules: If $(\LettreB, D)$
is a substitution rule of $S$, we shall denote it by $\LettreB \longmapsto_S D$
or by $\LettreB \longmapsto D$ if $S$ is fixed. Moreover, we will abbreviate
the substitutions rules $\LettreB \longmapsto_S D_1$, \dots, $\LettreB \longmapsto_S D_n$
by
\begin{equation}
    \LettreB \longmapsto_S D_1 + \cdots + D_n.
\end{equation}

\begin{Definition} \label{def:Derivation}
    Let $S := (B, a, R)$ be a synchronous grammar and $D_0$ be a bud tree
    with frontier $(b_1, \dots, b_n)$ where $\Eval(b_i) = \LettreB_i$ for
    all $1 \leq i \leq n$. We say that the bud tree $D_1$ is \emph{derivable}
    from $D_0$ in $S$, and we denote that by $D_0 \xrightarrow{S} D_1$,
    if there exists a sequence of substitution rules
    $(\LettreB_1 \longmapsto T_1, \dots, \LettreB_n \longmapsto T_n)$ of
    $R^n$ such that, by simultaneously substituting the bud $b_i$ of $D_0$
    by the root of $T_i$ for all $1 \leq i \leq n$, one obtains $D_1$.
\end{Definition}

\begin{Definition} \label{def:Generation}
    A bud tree $D$ is \emph{generated} by a synchronous grammar $S := (B, a, R)$
    if there exists a sequence $(D_1, \dots, D_{\ell - 1})$ of bud trees such that
    \begin{equation} \label{eq:GramSyncGeneration}
        a \xrightarrow{S} D_1 \xrightarrow{S} \cdots \xrightarrow{S} D_{\ell - 1} \xrightarrow{S} D.
    \end{equation}
    Moreover, we say that $D$ is generated by a \emph{$\ell$-steps derivation}.
\end{Definition}

We denote by $\Lang_S^{(\ell)}$ the set of the bud trees generated by $\ell$-steps
derivations and by $\Lang_S$ the \emph{language} of $S$, that is the set of
all bud trees generated by $S$. We also say that $S$ is \emph{trim} if
for all $\LettreB \in B$ there exists at least one bud tree $D$ generated
by $S$ that contains a bud labeled by $\LettreB$. In the sequel, we shall
only consider trim synchronous grammars without mentioning it explicitly.
\medskip

We will illustrate most of the next definitions through the synchronous
grammar
\begin{equation}
    S_\textnormal{epl} := \left(\{x, y\}, \Bourgeon{.4}{$x$}, R\right),
\end{equation}
where $R$ contains the substitution rules
\begin{align}
    x & \longmapsto \raisebox{-0.65em}{\BourgeonA{.45}{$x$}{$2$}{$y$}} \quad + \quad
    \raisebox{-0.65em}{%
    \scalebox{.45}{%
    \begin{tikzpicture}
        \node[Bourgeon](1)at(-.15,0){\huge $x$};
        \node[Noeud,EtiqClair](2)at(1,1.25){$3$};
        \node[Bourgeon](3)at(2.15,0){\huge $x$};
        \draw[Arete](1)--(2);
        \draw[Arete](2)--(3);
        \node[Bourgeon](4)at(1,0){\huge $y$};
        \draw[Arete](2)--(4);
    \end{tikzpicture}}}\, , \label{eq:Regle1SEpl} \\
    y & \longmapsto \Bourgeon{.45}{$x$}\,. \label{eq:Regle2SEpl}
\end{align}
Figure~\ref{fig:UneDerivation} shows a derivation in $S_\textnormal{epl}$.
\begin{figure}[ht]
    \scalebox{.45}{%
    \raisebox{1em}{%
    \begin{tikzpicture}
        \node[Bourgeon](1)at(-.15,0){\huge $x$};
        \node[Noeud,EtiqClair](2)at(1,1.25){\huge $3$};
        \node[Bourgeon](3)at(2.15,0){\huge $x$};
        \draw[Arete](1)--(2);
        \draw[Arete](2)--(3);
        \node[Bourgeon](4)at(1,0){\huge $y$};
        \draw[Arete](2)--(4);
    \end{tikzpicture}}}
    \raisebox{1.5em}{$\xrightarrow{S_\textnormal{epl}}$}
    \scalebox{.45}{%
    \begin{tikzpicture}
        \node[Noeud,EtiqClair](1)at(-1,.25){\huge $2$};
        \node[Noeud,EtiqClair](2)at(1,1.5){\huge $3$};
        \node[Bourgeon](3)at(1,0){\huge $x$};
        \draw[Arete](1)--(2);
        \draw[Arete](2)--(3);
        \node[Noeud,EtiqClair](4)at(3,.25){\huge $3$};
        \draw[Arete](2)--(4);
        \node[Bourgeon](5)at(-1.65,-1){\huge $x$};
        \node[Bourgeon](6)at(-.35,-1){\huge $y$};
        \node[Bourgeon](7)at(1.85,-1){\huge $x$};
        \node[Bourgeon](8)at(3,-1){\huge $y$};
        \node[Bourgeon](9)at(4.15,-1){\huge $x$};
        \draw[Arete](1)--(5);
        \draw[Arete](1)--(6);
        \draw[Arete](4)--(7);
        \draw[Arete](4)--(8);
        \draw[Arete](4)--(9);
    \end{tikzpicture}}
    \caption{A $1$-step derivation in $S_\textnormal{epl}$.}
    \label{fig:UneDerivation}
\end{figure}

\subsubsection{Generating graph}
The \emph{$\ell$-generating graph} $\GenGraph_S^{(\ell)} := (V, E)$ of a
synchronous grammar $S$ is the directed graph defined by
\begin{equation}
    V := \bigcup_{0 \leq i \leq \ell} \Lang_S^{(i)},
\end{equation}
and
\begin{equation}
    E := \left\{(D_0, D_1) \in V^2: D_0 \xrightarrow{S} D_1\right\}.
\end{equation}
The \emph{generating graph} of $S$ is the possibly infinite graph $\GenGraph_S$
defined as above where $V := \Lang_S$. This graph is connected and has at
most one source, the axiom $a$. Figure~\ref{fig:GenGraphe} shows an example
of a $2$-generating graph.

\begin{figure}[ht]
    \centering
    \scalebox{.37}{%
    \begin{tikzpicture}
        \node[Bourgeon](a1n1)at(0,0){\huge $x$};
        \node[fit=(a1n1)] (a1) {};

        \node[Noeud,EtiqClair](a2n1)at(6,5){\huge $2$};
        \node[Bourgeon](a2n2)at(5,4){\huge $x$};
        \node[Bourgeon](a2n3)at(7,4){\huge $y$};
        \draw[Arete](a2n1)--(a2n2);
        \draw[Arete](a2n1)--(a2n3);
        \node[fit=(a2n1) (a2n2) (a2n3)] (a2) {};

        \node[Bourgeon](a3n1)at(4.85,-10){\huge $x$};
        \node[Noeud,EtiqClair](a3n2)at(6,-8.75){\huge $3$};
        \node[Bourgeon](a3n3)at(7.15,-10){\huge $x$};
        \node[Bourgeon](a3n4)at(6,-10){\huge $y$};
        \draw[Arete](a3n1)--(a3n2);
        \draw[Arete](a3n2)--(a3n3);
        \draw[Arete](a3n2)--(a3n4);
        \node[fit=(a3n1) (a3n2) (a3n3) (a3n4)] (a3) {};

        \node[Noeud,EtiqClair](a4n1)at(15,8){\huge $2$};
        \node[Noeud,EtiqClair](a4n2)at(14,7){\huge $2$};
        \node[Bourgeon](a4n3)at(16,7){\huge $x$};
        \draw[Arete](a4n1)--(a4n2);
        \draw[Arete](a4n1)--(a4n3);
        \node[Bourgeon](a4n4)at(13.25,6){\huge $x$};
        \node[Bourgeon](a4n5)at(14.75,6){\huge $y$};
        \draw[Arete](a4n2)--(a4n4);
        \draw[Arete](a4n2)--(a4n5);
        \node[fit=(a4n1) (a4n2) (a4n3) (a4n4) (a4n5)] (a4) {};

        \node[Noeud,EtiqClair](a5n1)at(15,4.25){\huge $2$};
        \node[Noeud,EtiqClair](a5n2)at(14,3.25){\huge $3$};
        \node[Bourgeon](a5n3)at(16,3){\huge $x$};
        \draw[Arete](a5n1)--(a5n2);
        \draw[Arete](a5n1)--(a5n3);
        \node[Bourgeon](a5n4)at(12.85,2){\huge $x$};
        \node[Bourgeon](a5n5)at(15.15,2){\huge $x$};
        \draw[Arete](a5n2)--(a5n4);
        \draw[Arete](a5n2)--(a5n5);
        \node[Bourgeon](a5n6)at(14,2){\huge $y$};
        \draw[Arete](a5n2)--(a5n6);
        \node[fit=(a5n1) (a5n2) (a5n3) (a5n4) (a5n5) (a5n6)] (a5) {};

        \node[Noeud,EtiqClair](a6n1)at(13,-.75){\huge $2$};
        \node[Noeud,EtiqClair](a6n2)at(15,.5){\huge $3$};
        \node[Bourgeon](a6n3)at(15,-1){\huge $x$};
        \draw[Arete](a6n1)--(a6n2);
        \draw[Arete](a6n2)--(a6n3);
        \node[Noeud,EtiqClair](a6n4)at(17,-.75){\huge $2$};
        \draw[Arete](a6n2)--(a6n4);
        \node[Bourgeon](a6n5)at(12.35,-2){\huge $x$};
        \node[Bourgeon](a6n6)at(13.65,-2){\huge $y$};
        \node[Bourgeon](a6n7)at(16.35,-2){\huge $x$};
        \node[Bourgeon](a6n8)at(17.65,-2){\huge $y$};
        \draw[Arete](a6n1)--(a6n5);
        \draw[Arete](a6n1)--(a6n6);
        \draw[Arete](a6n4)--(a6n7);
        \draw[Arete](a6n4)--(a6n8);
        \node[fit=(a6n1) (a6n2) (a6n3) (a6n4) (a6n5) (a6n6) (a6n7) (a6n8)] (a6) {};

        \node[Noeud,EtiqClair](a7n1)at(13,-6){\huge $3$};
        \node[Noeud,EtiqClair](a7n2)at(15,-5){\huge $3$};
        \node[Bourgeon](a7n3)at(15,-6.25){\huge $x$};
        \draw[Arete](a7n1)--(a7n2);
        \draw[Arete](a7n2)--(a7n3);
        \node[Noeud,EtiqClair](a7n4)at(17,-6){\huge $2$};
        \draw[Arete](a7n2)--(a7n4);
        \node[Bourgeon](a7n5)at(11.85,-7.25){\huge $x$};
        \node[Bourgeon](a7n6)at(14.15,-7.25){\huge $x$};
        \node[Bourgeon](a7n7)at(16.35,-7.25){\huge $x$};
        \node[Bourgeon](a7n8)at(17.65,-7.25){\huge $y$};
        \node[Bourgeon](a7n9)at(13,-7.25){\huge $y$};
        \draw[Arete](a7n1)--(a7n5);
        \draw[Arete](a7n1)--(a7n6);
        \draw[Arete](a7n1)--(a7n9);
        \draw[Arete](a7n4)--(a7n7);
        \draw[Arete](a7n4)--(a7n8);
        \node[fit=(a7n1) (a7n2) (a7n3) (a7n4) (a7n5) (a7n6) (a7n7) (a7n8) (a7n9)] (a7) {};

        \node[Noeud,EtiqClair](a8n1)at(13,-11.75){\huge $2$};
        \node[Noeud,EtiqClair](a8n2)at(15,-10.5){\huge $3$};
        \node[Bourgeon](a8n3)at(15,-12){\huge $x$};
        \draw[Arete](a8n1)--(a8n2);
        \draw[Arete](a8n2)--(a8n3);
        \node[Noeud,EtiqClair](a8n4)at(17,-11.75){\huge $3$};
        \draw[Arete](a8n2)--(a8n4);
        \node[Bourgeon](a8n5)at(12.35,-13){\huge $x$};
        \node[Bourgeon](a8n6)at(13.65,-13){\huge $y$};
        \node[Bourgeon](a8n7)at(15.85,-13){\huge $x$};
        \node[Bourgeon](a8n8)at(17,-13){\huge $y$};
        \node[Bourgeon](a8n9)at(18.15,-13){\huge $x$};
        \draw[Arete](a8n1)--(a8n5);
        \draw[Arete](a8n1)--(a8n6);
        \draw[Arete](a8n4)--(a8n7);
        \draw[Arete](a8n4)--(a8n8);
        \draw[Arete](a8n4)--(a8n9);
        \node[fit=(a8n1) (a8n2) (a8n3) (a8n4) (a8n5) (a8n6) (a8n7) (a8n8) (a8n9)] (a8) {};

        \node[Noeud,EtiqClair](a9n1)at(13,-17.75){\huge $3$};
        \node[Noeud,EtiqClair](a9n2)at(15,-16.5){\huge $3$};
        \node[Bourgeon](a9n3)at(15,-18){\huge $x$};
        \draw[Arete](a9n1)--(a9n2);
        \draw[Arete](a9n2)--(a9n3);
        \node[Noeud,EtiqClair](a9n4)at(17,-17.75){\huge $3$};
        \draw[Arete](a9n2)--(a9n4);
        \node[Bourgeon](a9n5)at(11.85,-19){\huge $x$};
        \node[Bourgeon](a9n6)at(13,-19){\huge $y$};
        \node[Bourgeon](a9n7)at(15.85,-19){\huge $x$};
        \node[Bourgeon](a9n8)at(17,-19){\huge $y$};
        \node[Bourgeon](a9n9)at(18.15,-19){\huge $x$};
        \node[Bourgeon](a9n10)at(14.15,-19){\huge $x$};
        \draw[Arete](a9n1)--(a9n5);
        \draw[Arete](a9n1)--(a9n6);
        \draw[Arete](a9n4)--(a9n7);
        \draw[Arete](a9n4)--(a9n8);
        \draw[Arete](a9n4)--(a9n9);
        \draw[Arete](a9n1)--(a9n10);
        \node[fit=(a9n1) (a9n2) (a9n3) (a9n4) (a9n5) (a9n6) (a9n7) (a9n8) (a9n9) (a9n10)] (a9) {};

        \draw[AreteGen] (a1)--(a2);
        \draw[AreteGen] (a1)--(a3);
        \draw[AreteGen] (a2)--(a4);
        \draw[AreteGen] (a2)--(a5);
        \draw[AreteGen] (a3)--(a6);
        \draw[AreteGen] (a3)--(a7);
        \draw[AreteGen] (a3)--(a8);
        \draw[AreteGen] (a3)--(a9);
    \end{tikzpicture}}

    \caption{The $2$-generating graph of $S_\textnormal{epl}$.}
    \label{fig:GenGraphe}
\end{figure}

\subsubsection{Strict and unambiguous synchronous grammars}

\begin{Definition} \label{def:GSStricte}
    A synchronous grammar $S := (B, a, R)$ is \emph{strict} if there exists
    a monomial order $\leq_B$ on the set of monomials of $\EnsRel [B]$ such
    that, for all bud tree $D_0$ generated by $S$ and all bud tree $D_1$
    derivable from $D_0$, we have $\Eval(D_0) <_B \Eval(D_1)$.
\end{Definition}

If $S$ is strict, since its set of substitution rules is finite, $S$ generates
only finitely many bud trees with a given evaluation, and since its set
of buds is finite, $S$ also generates only finitely many bud trees with a
given number of buds. Moreover, if $S$ is strict, its generating graph
$\GenGraph_S$ is acyclic.

\begin{Lemme} \label{lem:ConditionLocFin}
    Let $S := (B, a, R)$ be a synchronous grammar. If there exists a total
    order $\leq_B$ on $B$ such that, for all substitution rule $\LettreB \longmapsto D$
    of $R$ where $D \in \EnsBud_1$ we have $\LettreB <_B \Eval(D)$, then
    $S$ is strict.
\end{Lemme}
\begin{proof}
    We extend the total order $\leq_B$ defined on $B$ into a monomial order
    on the set of monomials of $\EnsRel [B]$ by considering the graded lexicographic
    order on monomials.

    Consider now a bud tree $D_0$ generated by $S$ and a bud tree $D_1$
    derivable from $D_0$. If there exists at least one bud of $D_0$ that
    is substituted by a bud tree with more than one bud, one has
    $\ell\left(\Eval(D_0)\right) < \ell\left(\Eval(D_1)\right)$ and hence $\Eval(D_0) <_B \Eval(D_1)$.
    Otherwise, $D_0$ and $D_1$ have the same number of buds. By hypothesis,
    all buds of the frontier $(b_1, \dots, b_n)$ of $D_0$ are substituted
    by $n$ bud trees each containing the buds $c_1$, \dots, $c_n$ such that
    $\Eval(b_i) <_B \Eval(c_i)$ for all $1 \leq i \leq n$. Hence, $\Eval(D_0) <_B \Eval(D_1)$,
    implying that $S$ is strict.
\end{proof}

For instance, $S_\textnormal{epl}$ is strict since the order $y \leq_B x$
meets the assumptions of Lemma~\ref{lem:ConditionLocFin}. This order can
be extended into the monomial order defined by
\begin{equation}
    x^i y^j \leq_B x^k y^\ell \qquad \mbox{if $i + j < k + l$ \quad or \quad $i + j = k + l$ and $i \leq k$}.
\end{equation}
\medskip

\begin{Definition} \label{def:GSNAmb}
    A synchronous grammar $S$ is \emph{unambiguous} if for all bud tree $D$,
    there exists at most one integer $\ell \geq 0$ and one sequence
    $(D_1, \dots, D_{\ell - 1})$ such that~(\ref{eq:GramSyncGeneration})
    holds.
\end{Definition}

The generating graph $\GenGraph_S$ is a tree if and only if $S$ is unambiguous.

\begin{Lemme} \label{lem:ConditionNAmb}
    Let $S := (B, a, R)$ be a strict synchronous grammar. If for
    all $\LettreB \in B$ and for all substitution rules $\LettreB \longmapsto T_0$
    and $\LettreB \longmapsto T_1$ of $R$ where $T_0 \ne T_1$ there are
    at the same location in $T_0$ and $T_1$ two non-bud nodes that are
    different, then $S$ is unambiguous.
\end{Lemme}
\begin{proof}
    Let $D$ be a bud tree generated by $S$ and $D_0$ and $D_1$ be two different
    bud trees derivable from $D$. Among other substitutions, the bud tree
    $D_0$ (resp. $D_1$) is obtained by replacing one of its buds by a bud
    tree $T_0$ (resp. $T_1$), and by hypothesis, there are at the same location
    in $T_0$ and $T_1$ two non-bud nodes that are different. Hence, there
    are at the same location in $D_0$ and $D_1$ two non-bud nodes that are
    different. This shows that all bud trees obtained by performing any
    sequence of derivations from $D_0$ and from $D_1$ are different since
    they differ by a non-bud node. Moreover, since $S$ is strict, its generating
    graph contains no cycle, and hence, $S$ is unambiguous.
\end{proof}

For instance, Lemma~\ref{lem:ConditionNAmb} shows that $S_\textnormal{epl}$
is unambiguous since it is strict and the bud \Bourgeon{.4}{$x$} can be
substituted by two buds trees with different roots: One of these is of
arity $2$ while the other one is of arity $3$.

\subsection{Synchronous grammars and generating series}

\begin{Definition}
    Let $S := (B, a, R)$ be a synchronous grammar. The \emph{$\ell$-generating series}
    $\GramSerie_S^{(\ell)}$ of $S$ is the polynomial of $\EnsRel [B]$ defined by
    \begin{equation}
        \GramSerie_S^{(\ell)}(\LettreB_1, \dots, \LettreB_k) := \sum_{
                a \xrightarrow{S} D_1 \xrightarrow{S} \cdots \xrightarrow{S} D_\ell} \Eval(D_\ell).
    \end{equation}
    Moreover, if $S$ is strict, the \emph{generating series} $\GramSerie_S$
    of $S$ is the element of $\EnsRel [[B]]$ defined by
    \begin{equation}
        \GramSerie_S(\LettreB_1, \dots, \LettreB_k) := \sum_{\ell \geq 0} \GramSerie_S^{(\ell)}(\LettreB_1, \dots, \LettreB_k).
    \end{equation}
\end{Definition}

Let $S := (B, a, R)$ be a strict synchronous grammar. The series $\GramSerie_S$
is well-defined since $S$ is strict. Moreover, if $S$ is also unambiguous, we have
\begin{equation}
    \GramSerie_S(\LettreB_1, \dots, \LettreB_k) = \sum_{D \; \in \; \Lang_S} \Eval(D),
\end{equation}
and for all monomial $u := \LettreB_1^{\alpha_1} \dots \LettreB_k^{\alpha_k}$,
the coefficient $[u] \GramSerie_S$ is the number of bud trees generated
by $S$ that have $u$ as evaluation, \emph{i.e.}, a frontier made of $\alpha_i$
occurrences of buds labeled by $\LettreB_i$, for all $1 \leq i \leq k$.

For example, the first $\ell$-generating series of $S_\textnormal{epl}$ are
\begin{align}
    \GramSerie_{S_\textnormal{epl}}^{(0)}(x, y) & = x, \\
    \GramSerie_{S_\textnormal{epl}}^{(1)}(x, y) & = xy + x^2y, \\
    \GramSerie_{S_\textnormal{epl}}^{(2)}(x, y) & = x^2y + x^3y + x^3y^2 + 2x^4y^2 + x^5y^2.
\end{align}
and its generating series is of the form
\begin{equation}
    \GramSerie_{S_\textnormal{epl}}(x, y) = x + xy + 2x^2y + x^3y + x^3y^2 + 2x^4y^2 + x^5y^2 + \cdots
\end{equation}
\medskip

For all $\LettreB \in B$ let us define the polynomials $\Expr(\LettreB)$
of $\EnsRel [B]$ by
\begin{equation}
    \Expr(\LettreB) := \sum_{(\LettreB, D) \; \in \; R} \Eval(D).
\end{equation}

For instance, for $S_\textnormal{epl}$ one directly obtains from~(\ref{eq:Regle1SEpl})
and~(\ref{eq:Regle2SEpl})
\begin{align}
    \Expr(x) & = xy + x^2y, \\
    \Expr(y) & = x.
\end{align}

\begin{Lemme} \label{lem:IterationGS}
    Let $S := (B, a, R)$ be a synchronous grammar. For all $\ell \geq 0$,
    $\GramSerie_S^{(\ell)}$ satisfies
    \begin{equation} \label{eq:IterationGS}
        \GramSerie_S^{(\ell)}(\LettreB_1, \dots, \LettreB_k) =
        \begin{cases}
            \Eval(a)                                                               & \mbox{if $\ell = 0$,} \\
            \GramSerie_S^{(\ell - 1)}(\Expr(\LettreB_1), \dots, \Expr(\LettreB_k)) & \mbox{otherwise.}
        \end{cases}
    \end{equation}
\end{Lemme}
\begin{proof}
    If $\ell = 0$, the only bud tree generated by $0$-step derivations is
    the axiom $a$ of $S$. Hence, the lemma is satisfied.

    Let $\ell \geq 1$. Assume that there exists the following sequence
    of derivations in $S$:
    \begin{equation}
        a \xrightarrow{S} D_1 \xrightarrow{S} \cdots \xrightarrow{S} D_{\ell-1} \xrightarrow{S} D_\ell.
    \end{equation}
    Then, by definition, $D_\ell$ is obtained by substituting the buds
    $b_i$ of $D_{\ell-1}$ by some buds trees $T_i$. From the polynomial
    point of view, the monomial $\Eval(D_\ell)$ is obtained by the polynomial
    substitutions $\Eval(b_i) \mapsfrom \Eval(T_i)$ in $\GramSerie_S^{(\ell-1)}$.
    Hence, $\GramSerie_S^{(\ell)}$ is obtained from $\GramSerie_S^{(\ell-1)}$
    by performing the polynomial substitution $\LettreB \leftarrow \Expr(\LettreB)$
    for each $\LettreB \in B$, showing~(\ref{eq:IterationGS}).
\end{proof}

\begin{Proposition} \label{prop:EquationFoncGS}
    Let $S := (B, a, R)$ be a strict synchronous grammar. The generating
    series $\GramSerie_S$ satisfies the fixed-point functional equation
    \begin{equation}
        \GramSerie_S(\LettreB_1, \dots, \LettreB_k) = \Eval(a) + \GramSerie_S(\Expr(\LettreB_1), \dots, \Expr(\LettreB_k)).
    \end{equation}
\end{Proposition}
\begin{proof}
    Using Lemma~\ref{lem:IterationGS}, we obtain
    \begin{align}
        \GramSerie_S(\LettreB_1, \dots, \LettreB_k) & = \sum_{\ell \geq 0} \GramSerie_S^{(\ell)}(\LettreB_1, \dots, \LettreB_k) \\
                     & = \GramSerie_S^{(0)}(\LettreB_1, \dots, \LettreB_k) +
                         \sum_{\ell \geq 1} \GramSerie_S^{(\ell)}(\LettreB_1, \dots, \LettreB_k) \\
                     & = \Eval(a) + \sum_{\ell \geq 0} \GramSerie_S^{(\ell + 1)}(\LettreB_1, \dots, \LettreB_k) \\
                     & = \Eval(a) + \sum_{\ell \geq 0} \GramSerie_S^{(\ell)}(\Expr(\LettreB_1), \dots, \Expr(\LettreB_k)) \\
                     & = \Eval(a) + \GramSerie_S(\Expr(\LettreB_1), \dots, \Expr(\LettreB_k)). \qedhere
    \end{align}
\end{proof}

Proposition~\ref{prop:EquationFoncGS} gives a formula to extract a fixed-point
functional equation for the generating series of a given strict synchronous grammar
$S := (B, a, R)$ and Lemma~\ref{lem:IterationGS} gives an algorithm to compute
its coefficients by \emph{iteration}, \emph{i.e.}, by computing the polynomials
$\GramSerie_S^{(\ell)}$ for $0 \leq \ell \leq n$ where $n$ is a desired order,
and then, by summing its terms.
\medskip

In our example, the generating series of $S_\textnormal{epl}$ satisfies
the fixed-point functional equation
\begin{equation}
    \GramSerie_{S_\textnormal{epl}}(x, y) = x + \GramSerie_{S_\textnormal{epl}}(xy + x^2y, x).
\end{equation}
\medskip

Note that in some cases it is useful to specialize the generating series
$\GramSerie_S$ associated with~$S$. For example, the specialization
of an element $\LettreB$ of $B$ to $0$ allows to annihilate some terms
of $\GramSerie_S$ corresponding to bud trees which have buds labeled by~$\LettreB$.
In this way, the enumeration provided by $\GramSerie_S$ with this specialization
takes into account only bud trees generated by $S$ that have no bud labeled by $\LettreB$.
\medskip

In the same way, it is possible to add some parameters to the substitution
rules of $S$ in order to refine the generating series $\GramSerie_S$. For
instance, to take into account the number of application of the substitution rule
\begin{equation}
    x \longmapsto
    \raisebox{-0.65em}{%
    \scalebox{.45}{%
    \begin{tikzpicture}
        \node[Bourgeon](1)at(-.15,0){\huge $x$};
        \node[Noeud,EtiqClair](2)at(1,1.25){$3$};
        \node[Bourgeon](3)at(2.15,0){\huge $x$};
        \draw[Arete](1)--(2);
        \draw[Arete](2)--(3);
        \node[Bourgeon](4)at(1,0){\huge $y$};
        \draw[Arete](2)--(4);
    \end{tikzpicture}}}\, ,
\end{equation}
in the bud trees generated by $S_\textnormal{epl}$, one has just to set
\begin{equation}
    \Expr(x) := xy + x^2y\xi,
\end{equation}
so that the parameter $\xi$ counts the number of application of this substitution
rule. In this way, one can enumerate tree-like structures according to some
statistics.

\subsection{Examples} \label{subsec:ExemplesGrammaires}
Let us consider three examples of synchronous grammars to illustrate
the concepts that we have presented. Let us start with a very simple example.

\subsubsection{Perfect binary trees}
Let the synchronous grammar $S_\textnormal{perf} := \left(\{x\}, \Bourgeon{.4}{$x$}, R\right)$ where
$R$ contains the unique substitution rule
\begin{equation}
    x \longmapsto \raisebox{-0.65em}{\BourgeonA{.45}{$x$}{}{$x$}}\, .
\end{equation}

By identifying the buds \Bourgeon{.4}{$x$} with leaves, the language
$\Lang_{S_\textnormal{perf}}$ is the set of \emph{perfect binary trees},
that are binary trees of the sequence $\left(T_i\right)_{i \geq 0}$ defined by
$T_0 := \ArbreVide$ and $T_{i + 1} := T_i \ABCons T_i$.
\medskip

This synchronous grammar is strict since the number of buds of all bud
trees generated by $S_\textnormal{perf}$ increases after each derivation.
Besides, since $S_\textnormal{perf}$ is strict and $R$ only contains one
substitution rule, the generating graph $\GenGraph_{S_\textnormal{perf}}$
only contains one maximal path and hence, $S_\textnormal{perf}$ is unambiguous.
Therefore, the series $\GramSerie_{S_\textnormal{perf}}$ is well-defined
and by Proposition~\ref{prop:EquationFoncGS}, it satisfies the fixed-point
functional equation
\begin{equation}
    \GramSerie_{S_\textnormal{perf}}(x) = x + \GramSerie_{S_\textnormal{perf}}(x^2),
\end{equation}
and enumerate perfect binary trees according to their number of leaves.
First $\GramSerie_{S_\textnormal{perf}}^{(\ell)}$ polynomials are
\begin{equation}
    \GramSerie_{S_\textnormal{perf}}^{(0)}(x) = x, \qquad
    \GramSerie_{S_\textnormal{perf}}^{(1)}(x) = x^2, \qquad
    \GramSerie_{S_\textnormal{perf}}^{(2)}(x) = x^4, \qquad
    \GramSerie_{S_\textnormal{perf}}^{(3)}(x) = x^8,
\end{equation}
so that
\begin{equation}
    \GramSerie_{S_\textnormal{perf}}(x) = \sum_{n \geq 0} x^{2^n} = x + x^2 + x^4 + x^8 + \cdots
\end{equation}

\subsubsection{Balanced 2-3 trees}
Let the synchronous grammar
$S_\textnormal{23} := \left(\{x\}, \Bourgeon{.4}{$x$}, R\right)$
where $R$ contains the substitution rules
\begin{equation}
    x \longmapsto \raisebox{-0.65em}{\BourgeonA{.45}{$x$}{$2$}{$x$}} \quad + \quad
    \raisebox{-0.65em}{%
    \scalebox{.45}{%
    \begin{tikzpicture}
        \node[Bourgeon](1)at(-.15,0){\huge $x$};
        \node[Noeud,EtiqClair](2)at(1,1.25){\huge $3$};
        \node[Bourgeon](3)at(2.15,0){\huge $x$};
        \draw[Arete](1)--(2);
        \draw[Arete](2)--(3);
        \node[Bourgeon](4)at(1,0){\huge $x$};
        \draw[Arete](2)--(4);
    \end{tikzpicture}}}\, .
\end{equation}

By identifying the buds \Bourgeon{.4}{$x$} with leaves, the language of
$S_\textnormal{23}$ is the set of \emph{balanced 2-3 trees}, that are complete
rooted planar trees such that each internal node has $2$ or $3$ children and all
paths leading to their leaves have same length (see~\cite{Odl82, FS09}).
\medskip

Since each step of derivation increases the number of buds of the generated
bud tree, $S_\textnormal{23}$ is strict. Moreover, $S_\textnormal{23}$ satisfies
the hypothesis of Lemma~\ref{lem:ConditionNAmb} and hence is unambiguous.
Indeed, the two bud trees appearing in the two substitution rules have a
different root: One is of arity $2$ and the other of arity $3$. Thus, the
series $\GramSerie_{S_\textnormal{23}}$ is well-defined and by
Proposition~\ref{prop:EquationFoncGS}, it satisfies the fixed-point functional equation
\begin{equation}
    \GramSerie_{S_\textnormal{23}}(x) = x + \GramSerie_{S_\textnormal{23}}(x^2 + x^3),
\end{equation}
and enumerate balanced 2-3 trees according to their number of leaves.
First polynomials $\GramSerie_{S_\textnormal{23}}^{(\ell)}$ are
\begin{align}
    \GramSerie_{S_\textnormal{23}}^{(0)}(x) & = x, \\
    \GramSerie_{S_\textnormal{23}}^{(1)}(x) & = x^2 + x^3, \\
    \GramSerie_{S_\textnormal{23}}^{(2)}(x) & = x^4 + 2x^5 + 2x^6 + 3x^7 + 3x^8 + x^9.
\end{align}
An interpretation of the polynomial $\GramSerie_{S_\textnormal{23}}^{(2)}(x)$ is
the following: By performing $2$-steps derivations, $S_\textnormal{23}$
generates one bud tree with 4 buds, two bud trees with 5 buds, two bud trees with 6 buds,
three bud trees with 7 buds, three bud trees with 8 buds and one bud tree with 9 buds.

\subsubsection{Balanced binary trees}
Consider now the synchronous grammar $S_\textnormal{bal} := \left(\{x, y\}, \Bourgeon{.4}{$x$}, R\right)$
where $R$ contains the substitution rules
\begin{align}
    x & \longmapsto
            \raisebox{-0.65em}{\BourgeonA{.45}{$x$}{$-1$}{$y$}} \quad + \quad
            \raisebox{-0.65em}{\BourgeonA{.45}{$x$}{$0$}{$x$}} \quad + \quad
            \raisebox{-0.65em}{\BourgeonA{.45}{$y$}{$1$}{$x$}}\, , \\
    y & \longmapsto \Bourgeon{.45}{$x$}\, .
\end{align}

As we shall show, by annihilating the bud trees containing some buds
\Bourgeon{.4}{$y$} and by replacing the buds \Bourgeon{.4}{$x$} by leaves,
the language of $S_\textnormal{bal}$ is the set of balanced binary trees.

\begin{Proposition} \label{prop:GrammaireEq}
    Let $D$ be a bud tree generated by $S_\textnormal{bal}$ only containing
    buds \Bourgeon{.4}{$x$}. Then, the non-bud nodes of $D$ are labeled
    by their imbalance value.
\end{Proposition}
\begin{proof}
    Each step of derivation leading to $D$ substitutes each \Bourgeon{.4}{$x$}
    by new bud trees of height two, and each \Bourgeon{.4}{$y$} by
    new bud trees of height one. Thus, each step of derivation increases
    by one the height of the subtrees containing a \Bourgeon{.4}{$x$}.
    Besides, the role of the \Bourgeon{.4}{$y$} is to delay, during
    one step of derivation, the growth of the branch containing these, to
    enable the creation of the imbalance values $-1$ and $1$. Since $D$
    does not have any \Bourgeon{.4}{$y$}, every growing delay is
    respected, so that imbalance values are its labels.
\end{proof}

Proposition~\ref{prop:GrammaireEq} shows that the bud trees generated by
$S_\textnormal{bal}$ only containing buds \Bourgeon{.4}{$x$} are balanced
binary trees. Moreover, a simple structural induction on balanced binary
trees shows that every balanced binary tree can be generated by $S_\textnormal{bal}$.
Indeed, the empty tree can be generated, and, if $T$ is a balanced binary
tree and $z$ its root, by induction hypothesis, its left subtree and its
right subtree can be generated by $S_\textnormal{bal}$. To generate $T$,
one just have to make the first derivation according to the imbalance value
of $z$. Figure~\ref{fig:ExempleGenerationEq} shows an example of generation
of a balanced binary tree.
\begin{figure}[ht]
    \centering
        \raisebox{1.9em}{%
        \scalebox{.4}{%
        \begin{tikzpicture}
            \node[Bourgeon]{\Huge $x$};
        \end{tikzpicture}}}
        \quad \raisebox{2em}{$\xrightarrow{S_\textnormal{bal}}$} \quad
        \raisebox{1.3em}{%
        \scalebox{.4}{%
        \begin{tikzpicture}
            \node[Bourgeon](1)at(0,0){\Huge $y$};
            \node[Noeud,EtiqClair](2)at(1,1){\LARGE $1$};
            \node[Bourgeon](3)at(2,0){\Huge $x$};
            \draw[Arete](1)--(2);
            \draw[Arete](2)--(3);
        \end{tikzpicture}}}
        \quad \raisebox{2em}{$\xrightarrow{S_\textnormal{bal}}$} \quad
        \raisebox{1.2em}{%
        \scalebox{.4}{%
        \begin{tikzpicture}
            \node[Bourgeon](0)at(0,-1){\Huge $x$};
            \node[Noeud,EtiqClair](1)at(1,0){\LARGE $1$};
            \draw[Arete](1)--(0);
            \node[Bourgeon](2)at(2,-2){\Huge $x$};
            \node[Noeud,EtiqClair](3)at(3,-1){\LARGE $-1$};
            \draw[Arete](3)--(2);
            \node[Bourgeon](4)at(4,-2){\Huge $y$};
            \draw[Arete](3)--(4);
            \draw[Arete](1)--(3);
        \end{tikzpicture}}}
        \quad \raisebox{2em}{$\xrightarrow{S_\textnormal{bal}}$} \quad
        \raisebox{0.7em}{%
        \scalebox{.4}{%
        \begin{tikzpicture}
            \node[Bourgeon](0)at(0,-2){\Huge $x$};
            \node[Noeud,EtiqClair](1)at(1,-1){\LARGE $-1$};
            \draw[Arete](1)--(0);
            \node[Bourgeon](2)at(2,-2){\Huge $y$};
            \draw[Arete](1)--(2);
            \node[Noeud,EtiqClair](3)at(3,0){\LARGE $1$};
            \draw[Arete](3)--(1);
            \node[Bourgeon](4)at(4,-3){\Huge $y$};
            \node[Noeud,EtiqClair](5)at(5,-2){\LARGE $1$};
            \draw[Arete](5)--(4);
            \node[Bourgeon](6)at(6,-3){\Huge $x$};
            \draw[Arete](5)--(6);
            \node[Noeud,EtiqClair](7)at(7,-1){\LARGE $-1$};
            \draw[Arete](7)--(5);
            \node[Bourgeon](8)at(8,-2){\Huge $x$};
            \draw[Arete](7)--(8);
            \draw[Arete](3)--(7);
        \end{tikzpicture}}} \\
        \quad \raisebox{2em}{$\xrightarrow{S_\textnormal{bal}}$} \quad
        \scalebox{.4}{%
        \begin{tikzpicture}
            \node[Bourgeon](0)at(0,-3){\Huge $x$};
            \node[Noeud,EtiqClair](1)at(1,-2){\LARGE $0$};
            \draw[Arete](1)--(0);
            \node[Bourgeon](2)at(2,-3){\Huge $x$};
            \draw[Arete](1)--(2);
            \node[Noeud,EtiqClair](3)at(3,-1){\LARGE $-1$};
            \draw[Arete](3)--(1);
            \node[Bourgeon](4)at(4,-2){\Huge $x$};
            \draw[Arete](3)--(4);
            \node[Noeud,EtiqClair](5)at(5,0){\LARGE $1$};
            \draw[Arete](5)--(3);
            \node[Bourgeon](6)at(6,-3){\Huge $x$};
            \node[Noeud,EtiqClair](7)at(7,-2){\LARGE $1$};
            \draw[Arete](7)--(6);
            \node[Bourgeon](8)at(8,-4){\Huge $x$};
            \node[Noeud,EtiqClair](9)at(9,-3){\LARGE $0$};
            \draw[Arete](9)--(8);
            \node[Bourgeon](10)at(10,-4){\Huge $x$};
            \draw[Arete](9)--(10);
            \draw[Arete](7)--(9);
            \node[Noeud,EtiqClair](11)at(11,-1){\LARGE $-1$};
            \draw[Arete](11)--(7);
            \node[Bourgeon](12)at(12,-3){\Huge $x$};
            \node[Noeud,EtiqClair](13)at(13,-2){\LARGE $0$};
            \draw[Arete](13)--(12);
            \node[Bourgeon](14)at(14,-3){\Huge $x$};
            \draw[Arete](13)--(14);
            \draw[Arete](11)--(13);
            \draw[Arete](5)--(11);
        \end{tikzpicture}}
        \quad \raisebox{2em}{$\approx$} \quad
        \scalebox{.25}{%
        \raisebox{2.5em}{%
        \begin{tikzpicture}
            \node[Feuille](0)at(0,-3){};
            \node[Noeud](1)at(1,-2){};
            \node[Feuille](2)at(2,-3){};
            \draw[Arete](1)--(0);
            \draw[Arete](1)--(2);
            \node[Noeud](3)at(3,-1){};
            \node[Feuille](4)at(4,-2){};
            \draw[Arete](3)--(1);
            \draw[Arete](3)--(4);
            \node[Noeud](5)at(5,0){};
            \node[Feuille](6)at(6,-3){};
            \node[Noeud](7)at(7,-2){};
            \node[Feuille](8)at(8,-4){};
            \node[Noeud](9)at(9,-3){};
            \node[Feuille](10)at(10,-4){};
            \draw[Arete](9)--(8);
            \draw[Arete](9)--(10);
            \draw[Arete](7)--(6);
            \draw[Arete](7)--(9);
            \node[Noeud](11)at(11,-1){};
            \node[Feuille](12)at(12,-3){};
            \node[Noeud](13)at(13,-2){};
            \node[Feuille](14)at(14,-3){};
            \draw[Arete](13)--(12);
            \draw[Arete](13)--(14);
            \draw[Arete](11)--(7);
            \draw[Arete](11)--(13);
            \draw[Arete](5)--(3);
            \draw[Arete](5)--(11);
        \end{tikzpicture}}}
    \caption{Generation of a balanced binary tree by the synchronous grammar $S_\textnormal{bal}$.}
    \label{fig:ExempleGenerationEq}
\end{figure}
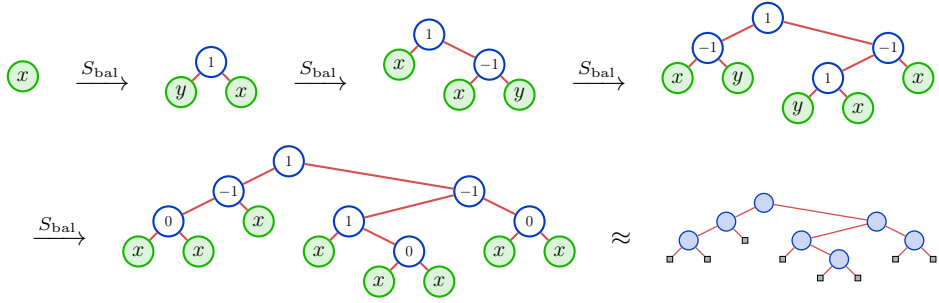
\medskip

By setting $y \leq_B x$, $S_\textnormal{bal}$ satisfies the hypothesis of
Lemma~\ref{lem:ConditionLocFin} and hence, is strict. Moreover, Lemma~\ref{lem:ConditionNAmb}
shows that $S_\textnormal{bal}$ is unambiguous since all bud trees arising
in a right member of the substitution rules of $S_\textnormal{bal}$ have
a different root since their labeling differ. Proposition~\ref{prop:GrammaireEq}
shows that this labeling is consistent. Hence $\GramSerie_{S_\textnormal{bal}}$
is well-defined. By Proposition~\ref{prop:EquationFoncGS}, the generating
series enumerating the elements of $\Lang_{S_\textnormal{bal}}$ satisfies
the fixed-point functional equation
\begin{equation}
    \GramSerie_{S_\textnormal{bal}}(x, y) = x + \GramSerie_{S_\textnormal{bal}}(x^2 + 2xy, x).
\end{equation}
First $\GramSerie_{S_\textnormal{bal}}^{(\ell)}$ polynomials are
\begin{align}
    \GramSerie_{S_\textnormal{bal}}^{(0)}(x, y) & = x, \\
    \GramSerie_{S_\textnormal{bal}}^{(1)}(x, y) & = 2xy + x^2, \\
    \GramSerie_{S_\textnormal{bal}}^{(2)}(x, y) & = 4x^2y + 2x^3 + 4x^2y^2 + 4x^3y + x^4.
\end{align}
As already mentioned, to enumerate balanced binary trees, we have to discard the
elements of $\Lang_{S_\textnormal{bal}}$ that contain a bud labelled by $y$.
Thus, the generating series enumerating balanced binary trees according to their
number of leaves is given by the specialization $\GramSerie_{S_\textnormal{bal}}(x, 0)$.
Note that this fixed-point functional equation is obtained in~\cite{BLR88, BLL94, KNU398}
by other methods.

\section{Imbalance tree patterns and balanced binary trees} \label{sec:Motifs}

Word patterns and permutations patterns are usually used to describe
languages or sets of permutations by considering the set of elements avoiding
these ones. We use the same idea to describe sets of binary trees by introducing
a notion of binary tree patterns and pattern avoidance.
\medskip

We show that we can describe in this way some interesting subsets of
the set of balanced binary trees according to their particular position
in the Tamari lattice. Afterwards, we use the methods developed in the
previous section to construct synchronous grammar generating the maximal
balanced binary trees in the Tamari lattice and get fixed-point functional
equation for the generating series enumerating these.

\subsection{Imbalance tree patterns}

\begin{Definition}
    An \emph{imbalance tree pattern} is a nonempty incomplete rooted planar
    binary tree with labels in $\EnsRel$.
\end{Definition}

Let $T$ be a binary tree. We denote by $T^{\Des}$ the labeled binary tree
of shape $T$ whose nodes are labeled by their imbalance value. We say that
$T$ admits an \emph{occurrence} of the imbalance tree pattern $p$ if a
connected component of $T^{\Des}$ has the same shape and same labels as $p$.
\medskip

Now, given a set $P$ of imbalance tree patterns, we can define the set
composed of the binary trees that \emph{avoid} $P$, \emph{i.e.}, the binary trees
that do not admit any occurrence of the elements of~$P$. For example, the set
\begin{equation}
    \left \{
        \scalebox{.45}{%
        \raisebox{-0.5em}{%
        \begin{tikzpicture}
            \node[Noeud,EtiqClair]{\huge $i$};
        \end{tikzpicture}}}
        : i \notin \{-1, 0, 1\}
    \right \}
\end{equation}
describes the set of balanced binary trees, the set
\begin{equation}
    \left \{
        \scalebox{.45}{%
        \raisebox{-0.5em}{%
            \begin{tikzpicture}
                \node[Noeud,EtiqClair]{\huge $i$};
            \end{tikzpicture}}}
        : i \ne 0
    \right \}
\end{equation}
describes the set of perfect binary trees and
\begin{equation}
    \left \{ \raisebox{-0.75em}{\MotifA{.45}{$i$}{$j$}} : i, j \in \EnsRel \right \}
\end{equation}
describes the set of \emph{right comb binary trees}, that are binary trees
such that each node has an empty left subtree.
\medskip

As exposed in Section~\ref{subsec:ExemplesGrammaires}, synchronous grammars
allow to generate binary trees by controlling the imbalance values of their
nodes. Hence, they allow to generate binary trees that avoid some imbalance
tree patterns.

\subsection{Minimal and maximal balanced binary trees in the Tamari lattice}

\subsubsection{Minimal and maximal balanced binary trees}
Let us first describe a set of balanced binary trees and its counterpart
whose elements are, roughly speaking, at the end of the balanced
binary trees subposet of the Tamari lattice.

\begin{Definition}
    A balanced binary tree $T_0$ (resp. $T_1$) is \emph{maximal} (resp. \emph{minimal})
    if, for all binary tree $T_1$ (resp. $T_0$) such that $T_0 \CouvTam T_1$,
    we have $T_1$ (resp. $T_0$) unbalanced.
\end{Definition}

By Theorem~\ref{thm:ClotureIntArbresEq}, if $T_0$ (resp. $T_1$) is a maximal
(resp. minimal) balanced binary tree, then there does not exist any balanced
binary tree $T_1$ (resp. $T_0$) such that $T_0 \OrdTam T_1$. Maximal (resp. minimal)
balanced binary trees are thus maximal (resp. minimal) elements in the Tamari
lattice restricted to balanced binary trees.

\begin{Proposition} \label{prop:MotifsArbresEqMax}
    A balanced binary tree $T$ is maximal if and only if it avoids the set
    of imbalance tree patterns
    \begin{equation}
        P_{\textnormal{max}} :=
        \left \{
            \raisebox{-0.75em}{\MotifA{.45}{$-1$}{$-1$}},\quad \raisebox{-0.75em}{\MotifA{.45}{$0$}{$-1$}}
        \right \}.
    \end{equation}
    Similarly, a balanced binary tree $T$ is minimal if and only if it
    avoids the set of imbalance tree patterns
    \begin{equation}
        P_{\textnormal{min}} :=
        \left \{
            \raisebox{-0.75em}{\MotifB{.45}{$1$}{$1$}},\quad \raisebox{-0.75em}{\MotifB{.45}{$1$}{$0$}}
        \right \}.
    \end{equation}
\end{Proposition}
\begin{proof}
    Assume that $T$ is maximal. Then, for all binary tree $T_1$ such that
    $T \CouvTam T_1$ we have $T_1$ unbalanced. Thus, it is impossible to
    do a conservative balancing rotation into $T$ and, looking at the
    different sorts of  rotations studied in Section~\ref{subsec:RotationsEquilibre}
    it avoids the set $P_{\textnormal{max}}$.

    Conversely, assume that $T$ is a balanced binary tree that avoids the
    two patterns of $P_{\textnormal{max}}$, then, for every binary tree
    $T_1$ such that $T \CouvTam T_1$, $T_1$ is unbalanced since for all
    node $y$ which has a left child $x$ in $T$, the imbalance values of
    $x$ and $y$ satisfy one of the seven cases~\ref{item:CasRot3}--\ref{item:CasRot9}
    of Section~\ref{subsec:RotationsEquilibre}. Thus, we can only do unbalancing
    rotations into $T$, implying that $T$ is maximal.

    The second part of the proposition is done in an analogous way, considering
    left rotations instead of right rotations.
\end{proof}

\begin{Proposition} \label{prop:SGenArbresEqMax}
    The generating series enumerating maximal balanced binary trees according
    to the number of leaves of the trees is $\GramSerie_{\textnormal{max}}(x, 0, 0)$
    where
    \begin{equation}
        \GramSerie_{\textnormal{max}}(x, y, z) = x + \GramSerie_{\textnormal{max}}(x^2 + xy + yz, x, xy).
    \end{equation}
\end{Proposition}
\begin{proof}
    To obtain this fixed-point functional equation, let us consider the synchronous grammar
    $S_{\textnormal{max}} := \left(\{x, y, z\}, \Bourgeon{.4}{$x$}, R\right)$ where $R$ contains the
    substitution rules
    \begin{align}
        x & \longmapsto \raisebox{-.5em}{\BourgeonA{.45}{$x$}{$0$}{$x$}} \quad \raisebox{0.75em}{$+$} \quad
                        \raisebox{-.5em}{\BourgeonA{.45}{$y$}{$1$}{$x$}} \quad \raisebox{0.75em}{$+$} \quad
                        \raisebox{-.5em}{\BourgeonA{.45}{$z$}{$-1$}{$y$}}\, , \\
        y & \longmapsto \Bourgeon{.45}{$x$}\, , \\
        z & \longmapsto \raisebox{-.5em}{\BourgeonA{.45}{$y$}{$1$}{$x$}}\, .
    \end{align}

    We can apply the same idea developed in the proof of Proposition~\ref{prop:GrammaireEq}
    to show that the bud trees generated by $S_{\textnormal{max}}$ that
    only contain buds \Bourgeon{.4}{$x$} have non-bud nodes labeled by
    their imbalance values. Hence, by identifying in such trees the
    \Bourgeon{.4}{$x$} with leaves, $S_{\textnormal{max}}$ only generates
    maximal balanced binary trees. Indeed, by Proposition~\ref{prop:MotifsArbresEqMax},
    the generated trees must avoid the two patterns of $P_{\textnormal{max}}$.
    To do that, we have to control the growth of the \Bourgeon{.4}{$x$}
    when they are substituted by bud trees $D$ whose roots have imbalance
    values of $-1$. Indeed, if the root of the left subtree of $D$ grows
    with an imbalance value of $-1$ or $0$, one of the two patterns is not
    avoided. The idea is to force the imbalance value of the root of the
    left subtree of $D$ to be $1$, role played by the bud \Bourgeon{.4}{$z$}.
    The role of the bud \Bourgeon{.4}{$y$} is to delay the growth of a branch
    of the generated bud tree in order to create the imbalance values $-1$
    and $1$. Moreover, by structural induction on maximal balanced binary trees,
    one can also prove that all maximal balanced binary trees can be generated
    by $S_{\textnormal{max}}$.

    By setting $y \leq_B x$, $S_{\textnormal{max}}$ satisfies the hypothesis
    of Lemma~\ref{lem:ConditionLocFin}, and hence, is strict. This synchronous
    grammar is also unambiguous since it satisfies the hypothesis of
    Lemma~\ref{lem:ConditionNAmb}. Indeed, the roots of all bud trees
    appearing in a right member of the substitution rules of $R$ are different
    to one other, due to their labeling.

    Finally, since $S_{\textnormal{max}}$ is strict and unambiguous, by
    Proposition~\ref{prop:EquationFoncGS}, we obtain the stated fixed-point
    functional equation, and the generating series is obtained by the
    specialization $y = 0$ and $z = 0$ in order to ignore bud trees containing
    a bud labelled by $y$ or by~$z$.
\end{proof}

The solution of this fixed-point functional equation gives us the following
first values for the number of maximal balanced binary trees in the Tamari lattice:
\begin{equation}
    \begin{split}
        & 1, 1, 1, 1, 2, 2, 2, 4, 6, 9, 11, 13, 22, 38, 60, 89, 128, 183, 256, 353, 512, 805, 1336, 2221, 3594, \\
        & 5665, 8774, 13433, 20359, 30550, 45437, 67086, 98491, 144492, 213876.
    \end{split}
\end{equation}

\subsubsection{Interior balanced binary trees}
Let us now describe a set of balanced binary trees and its counterpart
whose elements are, roughly speaking, in the heart of the balanced binary
trees subposet of the Tamari lattice.

\begin{Definition}
    A balanced binary tree $T_0$ (resp. $T_1$) is \emph{right interior}
    (resp. \emph{left interior}) if all binary tree $T_1$ (resp. $T_0$)
    such that $T_0 \CouvTam T_1$ is balanced.
\end{Definition}

\begin{Proposition} \label{prop:ArbresEqInterieurs}
    A balanced binary tree $T$ is right interior if and only if it avoids
    the set of imbalance tree patterns
    \begin{equation}
        P_{\textnormal{rint}} := \left\{
            \raisebox{-0.75em}{\MotifA{.45}{$-1$}{$0$}},\quad
            \raisebox{-0.75em}{\MotifA{.45}{$-1$}{$1$}},\quad
            \raisebox{-0.75em}{\MotifA{.45}{$0$}{$0$}},\quad
            \raisebox{-0.75em}{\MotifA{.45}{$0$}{$1$}},\quad
            \raisebox{-0.75em}{\MotifA{.45}{$1$}{$-1$}},\quad
            \raisebox{-0.75em}{\MotifA{.45}{$1$}{$0$}},\quad
            \raisebox{-0.75em}{\MotifA{.45}{$1$}{$1$}}
        \right\}.
    \end{equation}
    Similarly, a balanced binary tree $T$ is left interior if and only if
    it avoids the set of imbalance tree patterns
    \begin{equation}
        P_{\textnormal{lint}} := \left\{
            \raisebox{-0.75em}{\MotifB{.45}{$0$}{$1$}},\quad
            \raisebox{-0.75em}{\MotifB{.45}{$-1$}{$1$}},\quad
            \raisebox{-0.75em}{\MotifB{.45}{$0$}{$0$}},\quad
            \raisebox{-0.75em}{\MotifB{.45}{$-1$}{$0$}},\quad
            \raisebox{-0.75em}{\MotifB{.45}{$1$}{$-1$}},\quad
            \raisebox{-0.75em}{\MotifB{.45}{$0$}{$-1$}},\quad
            \raisebox{-0.75em}{\MotifB{.45}{$-1$}{$-1$}}
        \right\}.
    \end{equation}
\end{Proposition}
\begin{proof}
    Assume that $T$ is right interior. Then, for all binary tree $T_1$ such
    that $T \CouvTam T_1$, $T_1$ is balanced. Thus, for every node $y$ and
    its left child $x$ in $T$, the imbalance values of $x$ and $y$
    satisfy~\ref{item:CasRot1} or~\ref{item:CasRot2} of Section~\ref{subsec:RotationsEquilibre}
    since one can only do conservative balancing rotations in $T$. Hence,
    $T$ must avoid the seven given patterns.

    Conversely, assume that $T$ is a balanced binary tree that avoids the
    patterns of $P_{\textnormal{rint}}$. For every node $y$ which has a
    left child $x$ in $T$, the imbalance values of $x$ and $y$ satisfy~\ref{item:CasRot1}
    or~\ref{item:CasRot2}. Thus, the rotation of root $y$ in $T$ produces
    a balanced binary tree and implies that $T$ is interior.

    The second part of the proposition is done in an analogous way, considering
    left rotations instead of right rotations.
\end{proof}

In the sequel, we shall only consider right interior balanced binary trees
so we call these \emph{interior} balanced binary trees. This family of binary
trees is easily enumerable according to their height:
\begin{Proposition}
    The number $a_h$ of interior balanced binary trees of height $h$ is
    \begin{equation} \label{eq:ArbresInterieurs}
        a_h =
        \begin{cases}
            1                    & \mbox{if $h \in \{0, 1, 3\}$,} \\
            2                    & \mbox{if $h = 2$,} \\
            a_{h - 1} a_{h - 2}  & \mbox{otherwise.}
        \end{cases}
    \end{equation}
\end{Proposition}
\begin{proof}
    The values of $a_h$ for $0 \leq h \leq 3$ can easily be computed by hand.

    Let us first observe that if $T := L \ABCons R$ is an interior balanced
    binary tree of height $h \geq 3$, then $L$ and $R$ also are interior
    balanced binary trees and the imbalance value of the root of $T$ is $-1$.
    Indeed, if $L$ or $R$ is not an interior balanced binary tree, then,
    by Proposition~\ref{prop:ArbresEqInterieurs}, $L$ or $R$ would admit
    an occurrence of a pattern of $P_{\textnormal{rint}}$ and hence, would
    $T$. Moreover, if the imbalance value of $T$ is not $-1$, since $T$
    is balanced and $\Ht(T) \geq 3$, its left subtree $L$ is nonempty and
    $T$ would admit an occurrence of a pattern of $P_{\textnormal{rint}}$.

    Let us finally show that for all integer $h \geq 4$ and all interior
    balanced binary trees $L$ and $R$ such that $\Ht(L) = h - 1$ and $\Ht(R) = h - 2$,
    the binary tree $T := L \ABCons R$ is an interior balanced binary tree.
    Since $\Ht(L) \geq 3$, according to what we have just shown, the imbalance
    value of the root $x$ of $L$ is $-1$. The imbalance value of the root
    $y$ of $T$ also is $-1$ and thus, $x$ and $y$ do not form a pattern
    of $P_{\textnormal{rint}}$ in $T$. Moreover, the root of $R$ and the
    node $x$ in $T$ do neither form a pattern of $P_{\textnormal{rint}}$.
    Hence, $T$ is an interior balanced binary tree. That proves~(\ref{eq:ArbresInterieurs}).
\end{proof}

The first values of $(a_h)_{h \geq 0}$ are
\begin{equation}
    1, 1, 2, 1, 2, 2, 4, 8, 32, 256, 8192, 2097152, 17179869184.
\end{equation}
By forgetting the first three values, this is Sequence \Sloane{A000301}
of~\cite{SLO08}. Moreover, one has $a_{h} = 2^{f_{h-3}}$ for all $h \geq 3$,
where $f_i$ is the $i$-th Fibonacci number, defined by $f_i := i$ if
$i \in \{0, 1\}$, and $f_i := f_{i-1} + f_{i-2}$ otherwise.
\medskip

Recall that the set of \emph{Fibonacci binary trees}~\cite{CLR04} is formed
of the elements of the sequence $(T_i)_{i \geq 0}$ where $T_0 := T_1 := \ArbreVide$
and $T_{i + 2} := T_{i + 1} \ABCons T_i$. One can prove by structural induction on
the set of Fibonacci binary trees that these also are interior balanced binary trees.

\subsubsection{Mixed balanced binary trees}
Let us finally characterize balanced binary trees which are neither maximal nor interior.

\begin{Definition}
    A balanced binary tree $T_0$ is \emph{right mixed} (resp. \emph{left mixed})
    if there exists a balanced binary tree $T_1$ and an unbalanced binary
    tree $T'_1$ such that $T_0 \CouvTam T_1$ and $T_0 \CouvTam T'_1$
    (resp. $T_1 \CouvTam T_0$ and $T'_1 \CouvTam T_0$).
\end{Definition}

\begin{Proposition}
    A balanced binary tree $T$ is right mixed (resp. left mixed) if and only if it
    admits at least one occurrence of an imbalance tree pattern of the set
    $P_{\textnormal{max}}$ (resp. $P_{\textnormal{min}}$) and at least one
    occurrence of an imbalance tree pattern of the set $P_{\textnormal{rint}}$
    (resp. $P_{\textnormal{lint}}$).
\end{Proposition}
\begin{proof}
    Assume that $T$ is a mixed balanced binary tree. By definition, it is
    possible to perform a conservative balancing rotation into $T$. Hence,
    there are two nodes $x$ and $y$ in $T$ satisfying~\ref{item:CasRot1}
    or~\ref{item:CasRot2} of Section~\ref{subsec:RotationsEquilibre} and
    form an occurrence of a pattern of $P_{\textnormal{max}}$.
    Moreover, again by definition, it is possible to perform an unbalancing
    rotation into $T$. Hence, there are two nodes $x'$ and $y'$ in $T$ satisfying
    one of the seven cases ~\ref{item:CasRot3}--\ref{item:CasRot9} of Section~\ref{subsec:RotationsEquilibre}
    and form an occurrence of a pattern of $P_{\textnormal{rint}}$.

    Conversely, if $T$ admits some occurrences of patterns of both
    $P_{\textnormal{max}}$ and $P_{\textnormal{rint}}$, considering
    the nine cases of rotation in a balanced binary tree studied in
    Section~\ref{subsec:RotationsEquilibre}, we see that it is possible
    to make both a conservative and an unbalancing rotation into $T$, and
    hence $T$ is a right mixed balanced binary tree.

    The second part of the proposition is done in an analogous way, considering
    left rotations instead of right rotations.
\end{proof}

In the sequel, we shall only consider right mixed balanced binary trees,
so we call these \emph{mixed} balanced binary trees.
\medskip

Note that, for $n \geq 3$, the set $\EnsEq_n$ is a disjoint union of the
set $M$ of maximal balanced binary trees, the set $N$ of interior balanced
binary trees and the set $X$ of mixed balanced binary trees with $n$ nodes.
Indeed, by definition, $M$ and $X$ are disjoint, and in the same way, $N$
and $X$ also are. Consider now a balanced binary tree $T$ which is both
maximal and interior. That implies that $T$ is the maximal element of its
Tamari lattice, and hence, $T$ is a right comb binary tree. Since $T$ is
also balanced, it cannot have more than two nodes.

\section{The subposet of the Tamari lattice of balanced binary trees} \label{sec:Shape}

\subsection{Isomorphism between balanced binary tree intervals and hypercubes}

\begin{Lemme} \label{lem:CreationRotationEqImpossible}
    Let $T_0$ and $T_1$ be two balanced binary trees such that $T_0 \OrdTam T_1$
    and $y$ be a node of $T_0$. Then:
    \begin{enumerate}[label = {\it (\roman*)}]
        \item If the rotation of root $y$ in $T_0$ is an unbalancing rotation,
        then, if it exists, the rotation of root $y$ in $T_1$ is still an
        unbalancing rotation; \label{item:L1}
        \item If $y$ has no left child in $T_0$, then $y$ has no
        left child in $T_1$. \label{item:L2}
    \end{enumerate}
\end{Lemme}
\begin{proof}
    \ref{item:L1}: If the rotation of root $y$ in $T_0$ is an unbalancing
    rotation, it is because the imbalance values of $y$ and its left child
    $x$ do not satisfy~\ref{item:CasRot1} or~\ref{item:CasRot2} of
    Section~\ref{subsec:RotationsEquilibre}. Thus, to change these imbalance
    values, one has to perform rotations to change the height of some subtrees
    of $x$ and $y$. By Proposition~\ref{prop:HauteursEgalesEq}, these rotations
    necessarily unbalance the obtained binary tree. Moreover, by
    Theorem~\ref{thm:ClotureIntArbresEq}, it is impossible to make rotations
    to balance it again. This shows that if $y$ has a left child in $T_1$,
    it is necessarily a root of an unbalancing rotation.

    \ref{item:L2}: This is immediate from the definition of the rotation operation and
    by the fact that the rotation operation does not change the infix order
    of the nodes of a binary tree.
\end{proof}

Lemma~\ref{lem:CreationRotationEqImpossible} shows that for all balanced
binary trees $T_0$ and $T_1$ such that $T_0 \OrdTam T_1$, a node $y$
cannot become a root of a conservative balancing rotation in $T_1$ if
it is not a root of a conservative balancing rotation in $T_0$.

\begin{Lemme} \label{lem:RotationsEqIndependantes}
    Let $T_0$ and $T_1$ be two balanced binary trees and $y$ be a node of
    $T_0$ such that $T_1$ is obtained from $T_0$ by a rotation of root $y$.
    Then, denoting by $x$ the left child of $y$ in $T_0$, for all balanced
    binary tree $T_2$ such that $T_1 \OrdTam T_2$, $x$ and $y$
    cannot be roots of conservative balancing rotations in $T_2$.
\end{Lemme}
\begin{proof}
    Since $T_1$ is obtained by performing a conservative balancing
    rotation of root $y$ into $T_0$, we have two cases to consider, following
    the imbalance values of $x$ and $y$ in $T_0$. If $\Des_{T_0}(x) = \Des_{T_0}(y) = -1$,
    then $\Des_{T_1}(x) = \Des_{T_1}(y) = 1$ and $x$ and $y$ are not
    roots of conservative balancing rotations in $T_1$, so that, by
    Lemma~\ref{lem:CreationRotationEqImpossible}, $x$ and $y$ cannot be
    roots of conservative balancing rotations in $T_2$. If $\Des_{T_0}(x) = 0$
    and $\Des_{T_0}(y) -1$, then $\Des_{T_1}(x) = 1$ and $\Des_{T_1}(y) = 0$.
    For the same reason, $x$ and $y$ cannot be roots of conservative balancing
    rotations in $T_2$.
\end{proof}

A hypercube of dimension $k$ can be seen as a poset whose elements are
subsets of a set $\{e_1, \dots, e_k \}$, and ordered by the relation of inclusion.
Let us denote by $\HyperCube_k$ the hypercube poset of dimension $k$.
\medskip

We have the following characterization of the shape of balanced binary tree intervals:
\begin{Theoreme} \label{thm:IntEqHypercube}
    Let $T_0$ and $T_1$ be two balanced binary trees such that $T_0 \OrdTam T_1$.
    Then, the poset $([T_0, T_1], \OrdTam)$ is isomorphic to the hypercube $\HyperCube_k$,
    where $k$ is the number of rotations needed to transform $T_0$ into $T_1$.
\end{Theoreme}
\begin{proof}
    First, note by Theorem~\ref{thm:ClotureIntArbresEq}, that the interval
    $I := [T_0, T_1]$ only contains balanced binary trees. Hence, all
    covering relations in $I$ are conservative balancing rotations.

    Denote by $R$ the set of nodes $y$ of $T_0$ such that $y$ is a root of a
    rotation needed to transform $T_0$ into $T_1$. By Lemma~\ref{lem:RotationsEqIndependantes},
    $R$ is well defined --- it is not a multiset --- and if $y \in R$ then,
    denoting by $x$ the left child of $y$ in $T_0$, we have $x \notin R$.
    That implies that $T_1$ can be obtained from $T_0$ by performing, for
    all $y \in R$, a rotation of root $y$, independently of the order.

    Let us now define a bijection between the elements of $I$ and the set
    of the subsets of $R$. Let $T \in I$. By definition, it is possible to
    obtain $T$ by performing some rotations from $T_0$. Let $R_0$ be the set
    of nodes which are roots of these rotations. Besides, it is possible
    to obtain $T_1$ by performing some rotations from $T$. Let $R_1$ be the
    set of nodes which are roots of these rotations. By Lemma~\ref{lem:CreationRotationEqImpossible},
    we have $R = R_0 \uplus R_1$ and thus $R_0 \subset R$. The set $R_0$
    characterizes $T$. Conversely, for each subset $R_0 \subseteq R$ we
    can construct a unique binary tree $T \in I$. Indeed, $T$ is obtained
    by doing the rotations of root $y$ for all $y \in R_0$ into $T_0$, in
    any order. This is well-defined, by definition of $R$.

    This shows that the interval $I$ is isomorphic to the poset $\HyperCube_k$
    where $k$ is the number of rotations needed to transform $T_0$ into $T_1$.
\end{proof}

The first subposets of the Tamari lattice of balanced binary trees are depicted
in Figure~\ref{fig:InterEq}.

\begin{figure}[ht]
    \centering
    \subfigure[$(\EnsEq_0, \OrdTam)$]{\makebox[2cm]{\scalebox{.23}{
        \begin{tikzpicture}
            \node[Cercle] at (0, 0)  (000) {};
        \end{tikzpicture}
    }}}
    \subfigure[$(\EnsEq_1, \OrdTam)$]{\makebox[2cm]{\scalebox{.23}{
        \begin{tikzpicture}
            \node[Cercle] at (0, 0)  (000) {};
        \end{tikzpicture}
    }}}
    \subfigure[$(\EnsEq_2, \OrdTam)$]{\makebox[2cm]{\scalebox{.23}{
        \begin{tikzpicture}
            \node[Cercle] at (0, 0)  (000) {};
            \node[Cercle] at (0, 1)  (010) {};
            \draw[AreteCube] (000) -- (010);
        \end{tikzpicture}
    }}}
    \subfigure[$(\EnsEq_3, \OrdTam)$]{\makebox[2cm]{\scalebox{.23}{
        \begin{tikzpicture}
            \node[Cercle] at (0, 0)  (000) {};
        \end{tikzpicture}
    }}}
    \subfigure[$(\EnsEq_4, \OrdTam)$]{\makebox[2cm]{\scalebox{.23}{
        \begin{tikzpicture}
            \node[Cercle] at (0, 0) (1) {};
            \node[Cercle] at (1, 0) (2) {};
            \node[Cercle] at (0, 1) (3) {};
            \node[Cercle] at (1, 1) (4) {};
            \draw[AreteCube] (1) -- (3);
            \draw[AreteCube] (1) -- (4);
            \draw[AreteCube] (2) -- (4);
        \end{tikzpicture}
    }}}
    \subfigure[$(\EnsEq_5, \OrdTam)$]{\makebox[2cm]{\scalebox{.23}{
        \begin{tikzpicture}
            \node[Cercle] at (0, 0)  (1) {};
            \node[Cercle] at (-1, 1) (2) {};
            \node[Cercle] at (1, 1)  (3) {};
            \node[Cercle] at (0, 2)  (4) {};
            \draw[AreteCube] (1) -- (2);
            \draw[AreteCube] (1) -- (3);
            \draw[AreteCube] (2) -- (4);
            \draw[AreteCube] (3) -- (4);
            \node[Cercle] at (2, 0)  (000) {};
            \node[Cercle] at (2, 1)  (010) {};
            \draw[AreteCube] (000) -- (010);
        \end{tikzpicture}
    }}}
    \subfigure[$(\EnsEq_6, \OrdTam)$]{\makebox[2cm]{\scalebox{.23}{
        \begin{tikzpicture}
            \node[Cercle] at (0, 0)  (000) {};
            \node[Cercle] at (0, 1)  (010) {};
            \draw[AreteCube] (000) -- (010);
            \node[Cercle] at (1, 0)  (1) {};
            \node[Cercle] at (1, 1)  (2) {};
            \draw[AreteCube] (1) -- (2);
        \end{tikzpicture}
    }}}
    \subfigure[$(\EnsEq_7, \OrdTam)$]{\makebox[4cm]{\scalebox{.23}{
        \begin{tikzpicture}
            \node[Cercle] at (0, 0)  (000) {};
            \node[Cercle] at (-1, 1) (100) {};
            \node[Cercle] at (0, 1)  (010) {};
            \node[Cercle] at (1, 1)  (001) {};
            \node[Cercle] at (-1, 2) (110) {};
            \node[Cercle] at (1, 2)  (011) {};
            \node[Cercle] at (0, 2)  (101) {};
            \node[Cercle] at (0, 3)  (111) {};
            \draw[AreteCube] (000) -- (100);
            \draw[AreteCube] (000) -- (010);
            \draw[AreteCube] (000) -- (001);
            \draw[AreteCube] (100) -- (110);
            \draw[AreteCube] (010) -- (110);
            \draw[AreteCube] (010) -- (011);
            \draw[AreteCube] (001) -- (011);
            \draw[AreteCube] (100) -- (101);
            \draw[AreteCube] (001) -- (101);
            \draw[AreteCube] (101) -- (111);
            \draw[AreteCube] (110) -- (111);
            \draw[AreteCube] (011) -- (111);
            \node[Cercle] at (-2.5, 1.5)  (1) {};
            \node[Cercle] at (-3.5, 2.5) (2) {};
            \node[Cercle] at (-1.5, 2.5)  (3) {};
            \node[Cercle] at (-2.5, 3.5)  (4) {};
            \draw[AreteCube] (1) -- (2);
            \draw[AreteCube] (1) -- (3);
            \draw[AreteCube] (2) -- (4);
            \draw[AreteCube] (3) -- (4);
            \draw[AreteCube] (1) -- (110);
            \draw[AreteCube] (3) -- (111);
            \node[Cercle] at (2.5, -0.5)  (1a) {};
            \node[Cercle] at (1.5, 0.5) (2a) {};
            \node[Cercle] at (3.5, 0.5)  (3a) {};
            \node[Cercle] at (2.5, 1.5)  (4a) {};
            \draw[AreteCube] (1a) -- (2a);
            \draw[AreteCube] (1a) -- (3a);
            \draw[AreteCube] (2a) -- (4a);
            \draw[AreteCube] (3a) -- (4a);
            \draw[AreteCube] (000) -- (2a);
            \draw[AreteCube] (001) -- (4a);
            \node[Cercle] at (-2.5, 0)  (b) {};
        \end{tikzpicture}
    }}}
    \subfigure[$(\EnsEq_8, \OrdTam)$]{\makebox[4cm]{\scalebox{.23}{
        \begin{tikzpicture}
            \node[Cercle] at (4, 0)  (000) {};
            \node[Cercle] at (3, 1) (100) {};
            \node[Cercle] at (4, 1)  (010) {};
            \node[Cercle] at (5, 1)  (001) {};
            \node[Cercle] at (3, 2) (110) {};
            \node[Cercle] at (5, 2)  (011) {};
            \node[Cercle] at (4, 2)  (101) {};
            \node[Cercle] at (4, 3)  (111) {};
            \draw[AreteCube] (000) -- (100);
            \draw[AreteCube] (000) -- (010);
            \draw[AreteCube] (000) -- (001);
            \draw[AreteCube] (100) -- (110);
            \draw[AreteCube] (010) -- (110);
            \draw[AreteCube] (010) -- (011);
            \draw[AreteCube] (001) -- (011);
            \draw[AreteCube] (100) -- (101);
            \draw[AreteCube] (001) -- (101);
            \draw[AreteCube] (101) -- (111);
            \draw[AreteCube] (110) -- (111);
            \draw[AreteCube] (011) -- (111);
            \node[Cercle] at (7, 1)  (000b) {};
            \node[Cercle] at (6, 2) (100b) {};
            \node[Cercle] at (7, 2)  (010b) {};
            \node[Cercle] at (8, 2)  (001b) {};
            \node[Cercle] at (6, 3) (110b) {};
            \node[Cercle] at (8, 3)  (011b) {};
            \node[Cercle] at (7, 3)  (101b) {};
            \node[Cercle] at (7, 4)  (111b) {};
            \draw[AreteCube] (000b) -- (100b);
            \draw[AreteCube] (000b) -- (010b);
            \draw[AreteCube] (000b) -- (001b);
            \draw[AreteCube] (100b) -- (110b);
            \draw[AreteCube] (010b) -- (110b);
            \draw[AreteCube] (010b) -- (011b);
            \draw[AreteCube] (001b) -- (011b);
            \draw[AreteCube] (100b) -- (101b);
            \draw[AreteCube] (001b) -- (101b);
            \draw[AreteCube] (101b) -- (111b);
            \draw[AreteCube] (110b) -- (111b);
            \draw[AreteCube] (011b) -- (111b);
            \draw[AreteCube] (100) -- (100b);
            \draw[AreteCube] (010) -- (010b);
            \draw[AreteCube] (001) -- (001b);
            \draw[AreteCube] (110) -- (110b);
            \draw[AreteCube] (011) -- (011b);
            \draw[AreteCube] (101) -- (101b);
            \draw[AreteCube] (111) -- (111b);
            \draw[AreteCube] (000) -- (000b);
            \node[Cercle] at (0, 0)  (0) {};
            \node[Cercle] at (1, 1)  (1) {};
            \node[Cercle] at (2, 0)  (2) {};
            \node[Cercle] at (0, 2)  (3) {};
            \node[Cercle] at (2, 2)  (4) {};
            \node[Cercle] at (1, 3)  (5) {};
            \node[Cercle] at (-1, 1)  (6) {};
            \node[Cercle] at (-2, 0)  (7) {};
            \draw[AreteCube] (0) -- (1);
            \draw[AreteCube] (2) -- (1);
            \draw[AreteCube] (0) -- (3);
            \draw[AreteCube] (2) -- (4);
            \draw[AreteCube] (1) -- (5);
            \draw[AreteCube] (3) -- (5);
            \draw[AreteCube] (4) -- (5);
            \draw[AreteCube] (0) -- (6);
            \draw[AreteCube] (7) -- (6);
            \node[Cercle] at (11, 3)  (0a) {};
            \node[Cercle] at (10, 2) (1a) {};
            \node[Cercle] at (9, 3) (2a) {};
            \node[Cercle] at (11, 1)  (3a) {};
            \node[Cercle] at (9, 1) (4a) {};
            \node[Cercle] at (10, 0) (5a) {};
            \node[Cercle] at (12, 2)  (6a) {};
            \node[Cercle] at (13, 3)  (7a) {};
            \draw[AreteCube] (0a) -- (1a);
            \draw[AreteCube] (2a) -- (1a);
            \draw[AreteCube] (0a) -- (3a);
            \draw[AreteCube] (2a) -- (4a);
            \draw[AreteCube] (1a) -- (5a);
            \draw[AreteCube] (3a) -- (5a);
            \draw[AreteCube] (4a) -- (5a);
            \draw[AreteCube] (0a) -- (6a);
            \draw[AreteCube] (7a) -- (6a);
        \end{tikzpicture}
    }}}
    \subfigure[$(\EnsEq_9, \OrdTam)$]{\makebox[4cm]{\scalebox{.23}{
        \begin{tikzpicture}
            \node[Cercle] at (0, -2)  (1) {};
            \node[Cercle] at (-1, -1) (2) {};
            \node[Cercle] at (1, -1)  (3) {};
            \node[Cercle] at (0, 0)  (4) {};
            \draw[AreteCube] (1) -- (2);
            \draw[AreteCube] (1) -- (3);
            \draw[AreteCube] (2) -- (4);
            \draw[AreteCube] (3) -- (4);
            \node[Cercle] at (-2, 0) (1a) {};
            \node[Cercle] at (-3, 1) (2a) {};
            \node[Cercle] at (-1, 1) (3a) {};
            \node[Cercle] at (-2, 2) (4a) {};
            \draw[AreteCube] (1a) -- (2a);
            \draw[AreteCube] (1a) -- (3a);
            \draw[AreteCube] (2a) -- (4a);
            \draw[AreteCube] (3a) -- (4a);
            \node[Cercle] at (2, 0) (1b) {};
            \node[Cercle] at (1, 1) (2b) {};
            \node[Cercle] at (3, 1) (3b) {};
            \node[Cercle] at (2, 2) (4b) {};
            \draw[AreteCube] (1b) -- (2b);
            \draw[AreteCube] (1b) -- (3b);
            \draw[AreteCube] (2b) -- (4b);
            \draw[AreteCube] (3b) -- (4b);
            \draw[AreteCube] (1) -- (3a);
            \draw[AreteCube] (2) -- (4a);
            \draw[AreteCube] (1) -- (2b);
            \draw[AreteCube] (3) -- (4b);
            \node[Cercle] at (0, 2)  (1c) {};
            \node[Cercle] at (-1, 3) (2c) {};
            \node[Cercle] at (1, 3)  (3c) {};
            \node[Cercle] at (0, 4)  (4c) {};
            \draw[AreteCube] (1c) -- (2c);
            \draw[AreteCube] (1c) -- (3c);
            \draw[AreteCube] (2c) -- (4c);
            \draw[AreteCube] (3c) -- (4c);
            \draw[AreteCube] (1a) -- (2c);
            \draw[AreteCube] (3a) -- (4c);
            \draw[AreteCube] (1b) -- (3c);
            \draw[AreteCube] (2b) -- (4c);
            \node[Cercle] at (6, -4)  (000a) {};
            \node[Cercle] at (5, -3) (100a) {};
            \node[Cercle] at (6, -3)  (010a) {};
            \node[Cercle] at (7, -3)  (001a) {};
            \node[Cercle] at (5, -2) (110a) {};
            \node[Cercle] at (7, -2)  (011a) {};
            \node[Cercle] at (6, -2)  (101a) {};
            \node[Cercle] at (6, -1)  (111a) {};
            \draw[AreteCube] (000a) -- (100a);
            \draw[AreteCube] (000a) -- (010a);
            \draw[AreteCube] (000a) -- (001a);
            \draw[AreteCube] (100a) -- (110a);
            \draw[AreteCube] (010a) -- (110a);
            \draw[AreteCube] (010a) -- (011a);
            \draw[AreteCube] (001a) -- (011a);
            \draw[AreteCube] (100a) -- (101a);
            \draw[AreteCube] (001a) -- (101a);
            \draw[AreteCube] (101a) -- (111a);
            \draw[AreteCube] (110a) -- (111a);
            \draw[AreteCube] (011a) -- (111a);
            \node[Cercle] at (-3, -4)  (000b) {};
            \node[Cercle] at (-4, -3) (100b) {};
            \node[Cercle] at (-3, -3)  (010b) {};
            \node[Cercle] at (-2, -3)  (001b) {};
            \node[Cercle] at (-4, -2) (110b) {};
            \node[Cercle] at (-2, -2)  (011b) {};
            \node[Cercle] at (-3, -2)  (101b) {};
            \node[Cercle] at (-3, -1)  (111b) {};
            \draw[AreteCube] (000b) -- (100b);
            \draw[AreteCube] (000b) -- (010b);
            \draw[AreteCube] (000b) -- (001b);
            \draw[AreteCube] (100b) -- (110b);
            \draw[AreteCube] (010b) -- (110b);
            \draw[AreteCube] (010b) -- (011b);
            \draw[AreteCube] (001b) -- (011b);
            \draw[AreteCube] (100b) -- (101b);
            \draw[AreteCube] (001b) -- (101b);
            \draw[AreteCube] (101b) -- (111b);
            \draw[AreteCube] (110b) -- (111b);
            \draw[AreteCube] (011b) -- (111b);
            \node[Cercle] at (3, -4)  (000c) {};
            \node[Cercle] at (2, -3) (100c) {};
            \node[Cercle] at (3, -3)  (010c) {};
            \node[Cercle] at (4, -3)  (001c) {};
            \node[Cercle] at (2, -2) (110c) {};
            \node[Cercle] at (4, -2)  (011c) {};
            \node[Cercle] at (3, -2)  (101c) {};
            \node[Cercle] at (3, -1)  (111c) {};
            \draw[AreteCube] (000c) -- (100c);
            \draw[AreteCube] (000c) -- (010c);
            \draw[AreteCube] (000c) -- (001c);
            \draw[AreteCube] (100c) -- (110c);
            \draw[AreteCube] (010c) -- (110c);
            \draw[AreteCube] (010c) -- (011c);
            \draw[AreteCube] (001c) -- (011c);
            \draw[AreteCube] (100c) -- (101c);
            \draw[AreteCube] (001c) -- (101c);
            \draw[AreteCube] (101c) -- (111c);
            \draw[AreteCube] (110c) -- (111c);
            \draw[AreteCube] (011c) -- (111c);
            \node[Cercle] at (-.5, -4) (1d) {};
            \node[Cercle] at (.5, -4) (2d) {};
            \node[Cercle] at (-.5, -3) (3d) {};
            \node[Cercle] at (.5, -3) (4d) {};
            \draw[AreteCube] (1d) -- (3d);
            \draw[AreteCube] (1d) -- (4d);
            \draw[AreteCube] (2d) -- (4d);
        \end{tikzpicture}
    }}}
    \subfigure[$(\EnsEq_{10}, \OrdTam)$]{\makebox[4cm]{\scalebox{.23}{
        \begin{tikzpicture}
            \node[Cercle] at (0, 0)  (000a) {};
            \node[Cercle] at (-1, 1) (100a) {};
            \node[Cercle] at (0, 1)  (010a) {};
            \node[Cercle] at (1, 1)  (001a) {};
            \node[Cercle] at (-1, 2) (110a) {};
            \node[Cercle] at (1, 2)  (011a) {};
            \node[Cercle] at (0, 2)  (101a) {};
            \node[Cercle] at (0, 3)  (111a) {};
            \draw[AreteCube] (000a) -- (100a);
            \draw[AreteCube] (000a) -- (010a);
            \draw[AreteCube] (000a) -- (001a);
            \draw[AreteCube] (100a) -- (110a);
            \draw[AreteCube] (010a) -- (110a);
            \draw[AreteCube] (010a) -- (011a);
            \draw[AreteCube] (001a) -- (011a);
            \draw[AreteCube] (100a) -- (101a);
            \draw[AreteCube] (001a) -- (101a);
            \draw[AreteCube] (101a) -- (111a);
            \draw[AreteCube] (110a) -- (111a);
            \draw[AreteCube] (011a) -- (111a);
            \node[Cercle] at (3, 0.5)  (000b) {};
            \node[Cercle] at (2, 1.5) (100b) {};
            \node[Cercle] at (3, 1.5)  (010b) {};
            \node[Cercle] at (4, 1.5)  (001b) {};
            \node[Cercle] at (2, 2.5) (110b) {};
            \node[Cercle] at (4, 2.5)  (011b) {};
            \node[Cercle] at (3, 2.5)  (101b) {};
            \node[Cercle] at (3, 3.5)  (111b) {};
            \draw[AreteCube] (000b) -- (100b);
            \draw[AreteCube] (000b) -- (010b);
            \draw[AreteCube] (000b) -- (001b);
            \draw[AreteCube] (100b) -- (110b);
            \draw[AreteCube] (010b) -- (110b);
            \draw[AreteCube] (010b) -- (011b);
            \draw[AreteCube] (001b) -- (011b);
            \draw[AreteCube] (100b) -- (101b);
            \draw[AreteCube] (001b) -- (101b);
            \draw[AreteCube] (101b) -- (111b);
            \draw[AreteCube] (110b) -- (111b);
            \draw[AreteCube] (011b) -- (111b);
            \draw[AreteCube] (001a) -- (000b);
            \draw[AreteCube] (111a) -- (110b);
            \draw[AreteCube] (010b) -- (011a);
            \draw[AreteCube] (100b) -- (101a);
            \node[Cercle] at (6, 0)  (000a2) {};
            \node[Cercle] at (5, 1) (100a2) {};
            \node[Cercle] at (6, 1)  (010a2) {};
            \node[Cercle] at (7, 1)  (001a2) {};
            \node[Cercle] at (5, 2) (110a2) {};
            \node[Cercle] at (7, 2)  (011a2) {};
            \node[Cercle] at (6, 2)  (101a2) {};
            \node[Cercle] at (6, 3)  (111a2) {};
            \draw[AreteCube] (000a2) -- (100a2);
            \draw[AreteCube] (000a2) -- (010a2);
            \draw[AreteCube] (000a2) -- (001a2);
            \draw[AreteCube] (100a2) -- (110a2);
            \draw[AreteCube] (010a2) -- (110a2);
            \draw[AreteCube] (010a2) -- (011a2);
            \draw[AreteCube] (001a2) -- (011a2);
            \draw[AreteCube] (100a2) -- (101a2);
            \draw[AreteCube] (001a2) -- (101a2);
            \draw[AreteCube] (101a2) -- (111a2);
            \draw[AreteCube] (110a2) -- (111a2);
            \draw[AreteCube] (011a2) -- (111a2);
            \node[Cercle] at (9, 0.5)  (000b2) {};
            \node[Cercle] at (8, 1.5) (100b2) {};
            \node[Cercle] at (9, 1.5)  (010b2) {};
            \node[Cercle] at (10, 1.5)  (001b2) {};
            \node[Cercle] at (8, 2.5) (110b2) {};
            \node[Cercle] at (10, 2.5)  (011b2) {};
            \node[Cercle] at (9, 2.5)  (101b2) {};
            \node[Cercle] at (9, 3.5)  (111b2) {};
            \draw[AreteCube] (000b2) -- (100b2);
            \draw[AreteCube] (000b2) -- (010b2);
            \draw[AreteCube] (000b2) -- (001b2);
            \draw[AreteCube] (100b2) -- (110b2);
            \draw[AreteCube] (010b2) -- (110b2);
            \draw[AreteCube] (010b2) -- (011b2);
            \draw[AreteCube] (001b2) -- (011b2);
            \draw[AreteCube] (100b2) -- (101b2);
            \draw[AreteCube] (001b2) -- (101b2);
            \draw[AreteCube] (101b2) -- (111b2);
            \draw[AreteCube] (110b2) -- (111b2);
            \draw[AreteCube] (011b2) -- (111b2);
            \draw[AreteCube] (001a2) -- (000b2);
            \draw[AreteCube] (111a2) -- (110b2);
            \draw[AreteCube] (010b2) -- (011a2);
            \draw[AreteCube] (100b2) -- (101a2);
            \node[Cercle] at (0, 4)  (0) {};
            \node[Cercle] at (1, 5)  (1) {};
            \node[Cercle] at (2, 4)  (2) {};
            \node[Cercle] at (0, 6)  (3) {};
            \node[Cercle] at (2, 6)  (4) {};
            \node[Cercle] at (1, 7)  (5) {};
            \node[Cercle] at (3, 5)  (6) {};
            \node[Cercle] at (3, 7)  (7) {};
            \draw[AreteCube] (0) -- (1);
            \draw[AreteCube] (2) -- (1);
            \draw[AreteCube] (0) -- (3);
            \draw[AreteCube] (2) -- (4);
            \draw[AreteCube] (1) -- (5);
            \draw[AreteCube] (3) -- (5);
            \draw[AreteCube] (4) -- (5);
            \draw[AreteCube] (6) -- (7);
            \draw[AreteCube] (2) -- (6);
            \draw[AreteCube] (4) -- (7);
            \node[Cercle] at (5, 4)  (0a) {};
            \node[Cercle] at (6, 5)  (1a) {};
            \node[Cercle] at (7, 4)  (2a) {};
            \node[Cercle] at (5, 6)  (3a) {};
            \node[Cercle] at (7, 6)  (4a) {};
            \node[Cercle] at (6, 7)  (5a) {};
            \node[Cercle] at (8, 5)  (6a) {};
            \node[Cercle] at (8, 7)  (7a) {};
            \draw[AreteCube] (0a) -- (1a);
            \draw[AreteCube] (2a) -- (1a);
            \draw[AreteCube] (0a) -- (3a);
            \draw[AreteCube] (2a) -- (4a);
            \draw[AreteCube] (1a) -- (5a);
            \draw[AreteCube] (3a) -- (5a);
            \draw[AreteCube] (4a) -- (5a);
            \draw[AreteCube] (6a) -- (7a);
            \draw[AreteCube] (2a) -- (6a);
            \draw[AreteCube] (4a) -- (7a);
            \node[Cercle] at (1, 8)  (1b) {};
            \node[Cercle] at (0, 9) (2b) {};
            \node[Cercle] at (2, 9)  (3b) {};
            \node[Cercle] at (1, 10)  (4b) {};
            \draw[AreteCube] (1b) -- (2b);
            \draw[AreteCube] (1b) -- (3b);
            \draw[AreteCube] (2b) -- (4b);
            \draw[AreteCube] (3b) -- (4b);
            \node[Cercle] at (4, 8)  (1c) {};
            \node[Cercle] at (3, 9) (2c) {};
            \node[Cercle] at (5, 9)  (3c) {};
            \node[Cercle] at (4, 10)  (4c) {};
            \draw[AreteCube] (1c) -- (2c);
            \draw[AreteCube] (1c) -- (3c);
            \draw[AreteCube] (2c) -- (4c);
            \draw[AreteCube] (3c) -- (4c);
            \node[Cercle] at (7, 8)  (1d) {};
            \node[Cercle] at (6, 9) (2d) {};
            \node[Cercle] at (8, 9)  (3d) {};
            \node[Cercle] at (7, 10)  (4d) {};
            \draw[AreteCube] (1d) -- (2d);
            \draw[AreteCube] (1d) -- (3d);
            \draw[AreteCube] (2d) -- (4d);
            \draw[AreteCube] (3d) -- (4d);
        \end{tikzpicture}
    }}}
    \subfigure[$(\EnsEq_{11}, \OrdTam)$]{\makebox[4cm]{\scalebox{.23}{
        \begin{tikzpicture}
            \node[Cercle] at (4, 0)  (000) {};
            \node[Cercle] at (3, 1) (100) {};
            \node[Cercle] at (4, 1)  (010) {};
            \node[Cercle] at (5, 1)  (001) {};
            \node[Cercle] at (3, 2) (110) {};
            \node[Cercle] at (5, 2)  (011) {};
            \node[Cercle] at (4, 2)  (101) {};
            \node[Cercle] at (4, 3)  (111) {};
            \draw[AreteCube] (000) -- (100);
            \draw[AreteCube] (000) -- (010);
            \draw[AreteCube] (000) -- (001);
            \draw[AreteCube] (100) -- (110);
            \draw[AreteCube] (010) -- (110);
            \draw[AreteCube] (010) -- (011);
            \draw[AreteCube] (001) -- (011);
            \draw[AreteCube] (100) -- (101);
            \draw[AreteCube] (001) -- (101);
            \draw[AreteCube] (101) -- (111);
            \draw[AreteCube] (110) -- (111);
            \draw[AreteCube] (011) -- (111);
            \node[Cercle] at (7, 1)  (000b) {};
            \node[Cercle] at (6, 2) (100b) {};
            \node[Cercle] at (7, 2)  (010b) {};
            \node[Cercle] at (8, 2)  (001b) {};
            \node[Cercle] at (6, 3) (110b) {};
            \node[Cercle] at (8, 3)  (011b) {};
            \node[Cercle] at (7, 3)  (101b) {};
            \node[Cercle] at (7, 4)  (111b) {};
            \draw[AreteCube] (000b) -- (100b);
            \draw[AreteCube] (000b) -- (010b);
            \draw[AreteCube] (000b) -- (001b);
            \draw[AreteCube] (100b) -- (110b);
            \draw[AreteCube] (010b) -- (110b);
            \draw[AreteCube] (010b) -- (011b);
            \draw[AreteCube] (001b) -- (011b);
            \draw[AreteCube] (100b) -- (101b);
            \draw[AreteCube] (001b) -- (101b);
            \draw[AreteCube] (101b) -- (111b);
            \draw[AreteCube] (110b) -- (111b);
            \draw[AreteCube] (011b) -- (111b);
            \draw[AreteCube] (100) -- (100b);
            \draw[AreteCube] (010) -- (010b);
            \draw[AreteCube] (001) -- (001b);
            \draw[AreteCube] (110) -- (110b);
            \draw[AreteCube] (011) -- (011b);
            \draw[AreteCube] (101) -- (101b);
            \draw[AreteCube] (111) -- (111b);
            \draw[AreteCube] (000) -- (000b);
            \node[Cercle] at (1, 0)  (000c) {};
            \node[Cercle] at (0, 1) (100c) {};
            \node[Cercle] at (1, 1)  (010c) {};
            \node[Cercle] at (2, 1)  (001c) {};
            \node[Cercle] at (0, 2) (110c) {};
            \node[Cercle] at (2, 2)  (011c) {};
            \node[Cercle] at (1, 2)  (101c) {};
            \node[Cercle] at (1, 3)  (111c) {};
            \draw[AreteCube] (000c) -- (100c);
            \draw[AreteCube] (000c) -- (010c);
            \draw[AreteCube] (000c) -- (001c);
            \draw[AreteCube] (100c) -- (110c);
            \draw[AreteCube] (010c) -- (110c);
            \draw[AreteCube] (010c) -- (011c);
            \draw[AreteCube] (001c) -- (011c);
            \draw[AreteCube] (100c) -- (101c);
            \draw[AreteCube] (001c) -- (101c);
            \draw[AreteCube] (101c) -- (111c);
            \draw[AreteCube] (110c) -- (111c);
            \draw[AreteCube] (011c) -- (111c);
            \node[Cercle] at (10, 0)  (000d) {};
            \node[Cercle] at (9, 1) (100d) {};
            \node[Cercle] at (10, 1)  (010d) {};
            \node[Cercle] at (11, 1)  (001d) {};
            \node[Cercle] at (9, 2) (110d) {};
            \node[Cercle] at (11, 2)  (011d) {};
            \node[Cercle] at (10, 2)  (101d) {};
            \node[Cercle] at (10, 3)  (111d) {};
            \draw[AreteCube] (000d) -- (100d);
            \draw[AreteCube] (000d) -- (010d);
            \draw[AreteCube] (000d) -- (001d);
            \draw[AreteCube] (100d) -- (110d);
            \draw[AreteCube] (010d) -- (110d);
            \draw[AreteCube] (010d) -- (011d);
            \draw[AreteCube] (001d) -- (011d);
            \draw[AreteCube] (100d) -- (101d);
            \draw[AreteCube] (001d) -- (101d);
            \draw[AreteCube] (101d) -- (111d);
            \draw[AreteCube] (110d) -- (111d);
            \draw[AreteCube] (011d) -- (111d);
            \node[Cercle] at (0, 4)  (0e) {};
            \node[Cercle] at (1, 5)  (1e) {};
            \node[Cercle] at (2, 4)  (2e) {};
            \node[Cercle] at (0, 6)  (3e) {};
            \node[Cercle] at (2, 6)  (4e) {};
            \node[Cercle] at (1, 7)  (5e) {};
            \node[Cercle] at (3, 5)  (6e) {};
            \node[Cercle] at (3, 7)  (7e) {};
            \draw[AreteCube] (0e) -- (1e);
            \draw[AreteCube] (2e) -- (1e);
            \draw[AreteCube] (0e) -- (3e);
            \draw[AreteCube] (2e) -- (4e);
            \draw[AreteCube] (1e) -- (5e);
            \draw[AreteCube] (3e) -- (5e);
            \draw[AreteCube] (4e) -- (5e);
            \draw[AreteCube] (6e) -- (7e);
            \draw[AreteCube] (2e) -- (6e);
            \draw[AreteCube] (4e) -- (7e);
            \node[Cercle] at (4, 4)  (0f) {};
            \node[Cercle] at (5, 5)  (1f) {};
            \node[Cercle] at (6, 4)  (2f) {};
            \node[Cercle] at (4, 6)  (3f) {};
            \node[Cercle] at (6, 6)  (4f) {};
            \node[Cercle] at (5, 7)  (5f) {};
            \node[Cercle] at (7, 5)  (6f) {};
            \node[Cercle] at (7, 7)  (7f) {};
            \draw[AreteCube] (0f) -- (1f);
            \draw[AreteCube] (2f) -- (1f);
            \draw[AreteCube] (0f) -- (3f);
            \draw[AreteCube] (2f) -- (4f);
            \draw[AreteCube] (1f) -- (5f);
            \draw[AreteCube] (3f) -- (5f);
            \draw[AreteCube] (4f) -- (5f);
            \draw[AreteCube] (6f) -- (7f);
            \draw[AreteCube] (2f) -- (6f);
            \draw[AreteCube] (4f) -- (7f);
            \node[Cercle] at (8, 4)  (0g) {};
            \node[Cercle] at (9, 5)  (1g) {};
            \node[Cercle] at (10, 4)  (2g) {};
            \node[Cercle] at (8, 6)  (3g) {};
            \node[Cercle] at (10, 6)  (4g) {};
            \node[Cercle] at (9, 7)  (5g) {};
            \node[Cercle] at (11, 5)  (6g) {};
            \node[Cercle] at (11, 7)  (7g) {};
            \draw[AreteCube] (0g) -- (1g);
            \draw[AreteCube] (2g) -- (1g);
            \draw[AreteCube] (0g) -- (3g);
            \draw[AreteCube] (2g) -- (4g);
            \draw[AreteCube] (1g) -- (5g);
            \draw[AreteCube] (3g) -- (5g);
            \draw[AreteCube] (4g) -- (5g);
            \draw[AreteCube] (6g) -- (7g);
            \draw[AreteCube] (2g) -- (6g);
            \draw[AreteCube] (4g) -- (7g);
            \node[Cercle] at (0, 7)  (0h) {};
            \node[Cercle] at (1, 8)  (1h) {};
            \node[Cercle] at (2, 7)  (2h) {};
            \node[Cercle] at (0, 9)  (3h) {};
            \node[Cercle] at (2, 9)  (4h) {};
            \node[Cercle] at (1, 10)  (5h) {};
            \node[Cercle] at (3, 8)  (6h) {};
            \node[Cercle] at (3, 10)  (7h) {};
            \draw[AreteCube] (0h) -- (1h);
            \draw[AreteCube] (2h) -- (1h);
            \draw[AreteCube] (0h) -- (3h);
            \draw[AreteCube] (2h) -- (4h);
            \draw[AreteCube] (1h) -- (5h);
            \draw[AreteCube] (3h) -- (5h);
            \draw[AreteCube] (4h) -- (5h);
            \draw[AreteCube] (6h) -- (7h);
            \draw[AreteCube] (2h) -- (6h);
            \draw[AreteCube] (4h) -- (7h);
            \node[Cercle] at (6, 7)  (1i) {};
            \node[Cercle] at (5, 8) (2i) {};
            \node[Cercle] at (7, 8)  (3i) {};
            \node[Cercle] at (6, 9)  (4i) {};
            \draw[AreteCube] (1i) -- (2i);
            \draw[AreteCube] (1i) -- (3i);
            \draw[AreteCube] (2i) -- (4i);
            \draw[AreteCube] (3i) -- (4i);
            \node[Cercle] at (10, 7)  (000j) {};
            \node[Cercle] at (10, 8)  (010j) {};
            \draw[AreteCube] (000j) -- (010j);
        \end{tikzpicture}
    }}}
    \caption{Hasse diagrams of the first $(\EnsEq_n, \OrdTam)$ posets.}
    \label{fig:InterEq}
\end{figure}

\subsection{Enumeration of balanced binary tree intervals}

Let us make use again of the synchronous grammars to enumerate balanced
binary trees intervals.

\begin{Proposition} \label{prop:SerieGenIntEq}
    The generating series enumerating balanced binary tree intervals in the
    Tamari lattice according to the number of leaves of the trees is
    $\GramSerie_{\textnormal{bi}}(x, 0, 0)$ where
    \begin{equation}
        \GramSerie_{\textnormal{bi}}(x, y, z) = x + \GramSerie_{\textnormal{bi}}(x^2 + 2xy + yz, x, x^2 + xy).
    \end{equation}
\end{Proposition}
\begin{proof}
    Let $I := [T_0, T_1]$ be a balanced binary tree interval and $R$ be the set
    of nodes defined as in the proof of Theorem~\ref{thm:IntEqHypercube}
    associated with $I$. The proof of this theorem also shows
    that $I$ can be encoded by $T_0$ in which the nodes
    of $R$ are marked. To generate these objects, we consider a synchronous
    grammar which generates bud trees where (non-marked) nodes are \Noeud{.4}{}
    and marked nodes are \NoeudR{.4}{}. Let us consider the synchronous
    grammar $S_{\textnormal{bi}} := \left(\{x, y, z\}, \Bourgeon{.4}{$x$}, L\right)$ where $L$ contains
    the substitution rules
    \begin{align}
        x & \longmapsto
                \raisebox{-.7em}{\BourgeonA{.45}{$x$}{$-1$}{$y$}} \quad + \quad
                \raisebox{-.7em}{\BourgeonA{.45}{$x$}{$0$}{$x$}} \quad + \quad
                \raisebox{-.7em}{\BourgeonA{.45}{$y$}{$1$}{$x$}} \quad + \quad
                \raisebox{-.7em}{\BourgeonACarre{.45}{$z$}{$-1$}{$y$}}\, , \\
        y & \longmapsto \Bourgeon{.45}{$x$}\, , \\
        z & \longmapsto
                \raisebox{-.7em}{\BourgeonA{.45}{$x$}{$0$}{$x$}} \quad + \quad
                \raisebox{-.7em}{\BourgeonA{.45}{$x$}{$-1$}{$y$}}\, .
    \end{align}

    We can apply the same idea developed in the proof of Proposition~\ref{prop:GrammaireEq}
    to show that the bud trees generated by $S_{\textnormal{bi}}$ that only contain
    buds \Bourgeon{.4}{$x$} have non-bud nodes labeled by their imbalance
    values. Hence, identifying in such trees the \Bourgeon{.4}{$x$} with
    leaves, $S_{\textnormal{bi}}$ only generates balanced binary trees
    such that each of its node $r$ with $-1$ as imbalance value can be marked
    provided that its left child has $-1$ or $0$ as imbalance value and is
    not marked (recall that in this way, $r$ is a root of a conservative
    balancing rotation). Indeed, if a \Bourgeon{.4}{$x$} is substituted
    by a marked node, this marked node has a bud \Bourgeon{.4}{$z$} as
    left child and \Bourgeon{.4}{$z$} can only be substituted by a non-marked
    node with $-1$ or $0$ as imbalance value. The role of the bud \Bourgeon{.4}{$y$}
    is to delay the growth of a branch of the generated bud tree in order
    to create the imbalance values $-1$ and $1$.

    By setting $y \leq_B z \leq_B x$, $S_{\textnormal{bi}}$ satisfies the
    hypothesis of Lemma~\ref{lem:ConditionLocFin}, and hence, is strict.
    This synchronous grammar also is unambiguous since it satisfies the
    hypothesis of Lemma~\ref{lem:ConditionNAmb}. Indeed, the roots of the
    bud trees arising in a right member of the substitution rules of $L$
    are pairwise different, due to their labeling and their marking.

    Finally, since $S_{\textnormal{bi}}$ is strict and unambiguous, by
    Proposition~\ref{prop:EquationFoncGS}, we obtain the stated fixed-point
    functional equation, and the generating series is obtained by the specialization
    $y = 0$ and $z = 0$ in order to ignore bud trees that contain a bud
    labelled by~$y$ or by~$z$.
\end{proof}

The solution of this fixed-point functional equation gives us the following
first values for the number of balanced binary tree intervals in the Tamari lattice:
\begin{equation}
    \begin{split}
        & 1, 1, 3, 1, 7, 12, 6, 52, 119, 137, 195, 231, 1019, 3503, 6593, 12616, 26178, 43500, 64157, 94688, \\
        & 232560, 817757, 2233757, 5179734, 11676838, 24867480.
    \end{split}
\end{equation}
\medskip

The interval $[T_0, T_1]$ is a \emph{maximal balanced binary tree interval}
if $T_0$ (resp. $T_1$) is a minimal (resp. maximal) balanced binary tree.

\begin{Proposition} \label{prop:SerieGenIntArbresEqMax}
    The generating series enumerating maximal balanced binary tree intervals
    in the Tamari lattice according to the number of leaves of the trees
    is $\GramSerie_{\textnormal{mbi}}(x, 0, 0, 0)$ where
    \begin{equation}
        \GramSerie_{\textnormal{mbi}}(x, y, z, t) = x + \GramSerie_{\textnormal{mbi}}(x^2 + 2yt + yz, x, x^2 + xy, yt + yz).
    \end{equation}
\end{Proposition}
\begin{proof}
    Let $I := [T_0, T_1]$ be a maximal balanced binary tree interval. This
    interval can be encoded by the minimal balanced binary tree $T_0$ in
    which the nodes that are roots of the conservative balancing rotations
    needed to transform $T_0$ into $T_1$ are marked. Moreover, since $T_1$
    is a maximal balanced binary tree, by Proposition~\ref{prop:MotifsArbresEqMax},
    it avoids the patterns of $P_{\textnormal{max}}$. Hence, the tree-like
    structure that encodes $I$ must avoid the patterns of $P_{\textnormal{min}}$
    and not have a node which is root of a conservative balancing rotation
    not marked if its parent or its left child is not marked. To generate
    these objects, we use the synchronous grammar $S_{\textnormal{mbi}} := \left(\{x, y, z, u, v\}, \Bourgeon{.4}{$x$}, R\right)$
    where $R$ contains the substitution rules
    \begin{align}
        x & \longmapsto
                \raisebox{-.7em}{\BourgeonA{.45}{$v$}{$-1$}{$y$}} \quad + \quad
                \raisebox{-.7em}{\BourgeonA{.45}{$x$}{$0$}{$x$}} \quad + \quad
                \raisebox{-.7em}{\BourgeonA{.45}{$y$}{$1$}{$u$}} \quad + \quad
                \raisebox{-.7em}{\BourgeonACarre{.45}{$z$}{$-1$}{$y$}}\, , \\
        y & \longmapsto \Bourgeon{.45}{$x$}\, , \\
        z & \longmapsto
                \raisebox{-.7em}{\BourgeonA{.45}{$x$}{$-1$}{$y$}} \quad + \quad
                \raisebox{-.7em}{\BourgeonA{.45}{$x$}{$0$}{$x$}}\, , \\
        u & \longmapsto
                \raisebox{-.7em}{\BourgeonA{.45}{$v$}{$-1$}{$y$}} \quad + \quad
                \raisebox{-.7em}{\BourgeonACarre{.45}{$z$}{$-1$}{$y$}}\, , \\
        v & \longmapsto
                \raisebox{-1.4em}{\BourgeonA{.45}{$y$}{$1$}{$u$}} \quad + \quad
                \raisebox{-1.4em}{\BourgeonACarre{.45}{$z$}{$-1$}{$y$}}\, .
    \end{align}

    We can apply the same idea developed in the proof of Proposition~\ref{prop:GrammaireEq}
    to show that the bud trees generated by $S_{\textnormal{mbi}}$ that
    only contain buds \Bourgeon{.4}{$x$} have non-bud nodes labeled by
    their imbalance values. Hence, identifying in such trees the \Bourgeon{.4}{$x$}
    with leaves, $S_{\textnormal{mbi}}$ only generates minimal balanced binary trees
    that are maximally marked. Indeed, by Proposition~\ref{prop:MotifsArbresEqMax},
    the generated tree-like structures must avoid the two patterns of
    $P_{\textnormal{min}}$. To do that, we have to control the growth of
    the \Bourgeon{.4}{$x$} when they are substituted by bud trees $D$
    whose roots are not marked and have an imbalance value of $1$. Indeed,
    if the root of the right subtree of $D$ grows with an imbalance value
    of $1$ or $0$, one of the two patterns is not avoided. The idea is to
    force the imbalance value of the root of the right subtree of $D$ to
    be $-1$, role played by the bud \Bourgeon{.4}{$u$}. Moreover, if the
    \Bourgeon{.4}{$x$} are substituted by non-marked nodes $a$ labeled
    by $-1$, to generate trees that are maximally marked, the left child
    of $a$ has to be marked, or labeled by $1$ (in this case, $a$ is not
    root of a conservative balancing rotation). This is the role played
    by the bud \Bourgeon{.4}{$v$}. The bud \Bourgeon{.4}{$z$} appears
    in these substitution rules only as a left child of a marked node and
    it is substituted only by nodes with $-1$ or $0$ as imbalance value,
    that are the only ones authorized for a left child of a root of a
    conservative balancing rotation. As usual, the role of the bud
    \Bourgeon{.4}{$y$} is to delay the growth of a branch of the generated
    bud tree in order to create the imbalance values $-1$ and $1$.

    By setting $y \leq_B v \leq_B u \leq_B z \leq_B x$, $S_{\textnormal{mbi}}$
    satisfies the hypothesis of Lemma~\ref{lem:ConditionLocFin}, and hence,
    is strict. This synchronous grammar also is unambiguous since it satisfies
    the hypothesis of Lemma~\ref{lem:ConditionNAmb}. Indeed, the roots of
    each bud trees arising in a right member of the substitution rules of
    $R$ are different to one other, due to their labeling and their marking.

    By Proposition~\ref{prop:EquationFoncGS}, the fixed-point functional
    equation $F$ associated with $\GramSerie_{\textnormal{mbi}}$ is
    \begin{equation}
        F(x, y, z, u, v) = x + F(x^2 + yu + yv + yz, x, x^2 + xy, yv + yz, yu + yz),
    \end{equation}
    and, since the variables $u$ and $v$ play the same role, we obtain the
    stated fixed-point functional equation. The generating series is obtained
    by the specialization $y = 0$, $z = 0$ and $t = 0$ in order to ignore bud
    trees that contain a bud labelled by $y$, $z$, $u$, or by~$v$.
\end{proof}

The solution of this fixed-point functional equation gives us the following
first values for the number of maximal balanced binary tree intervals in
the Tamari lattice:
\begin{equation}
    \begin{split}
        & 1, 1, 1, 1, 3, 2, 2, 6, 9, 15, 15, 17, 41, 77, 125, 178, 252, 376, 531, 740, 1192, 2179, 4273, 7738, \\
        & 13012, 20776, 32389, 49841, 75457, 113011, 168888, 252881, 379348.
    \end{split}
\end{equation}
\medskip

We can slightly modify $S_{\textnormal{mbi}}$
to take into consideration the dimensions of the hypercubes isomorphic to
the enumerated maximal balanced binary tree intervals. For that, we have to
count the number of applications of substitution rules that generate
a marked node. Let us use for that a parameter $\xi$. Whence we obtain
the generating series defined by the fixed-point functional equation
\begin{equation}
    \GramSerie_{\textnormal{mbi}}(x, y, z, t, \xi) = x + \GramSerie_{\textnormal{mbi}}(x^2 + 2yt + yz\xi, x, x^2 + xy, yt + yz\xi, \xi).
\end{equation}
First coefficients of $x^i$ in $P := \GramSerie_{\textnormal{mbi}}(x, 0, 0, 0, \xi)$ are
\begin{multicols}{2}
    \small
    \begin{equation}
    \left[x^1\right] P  = 1,
    \end{equation}
    \begin{equation}
    \left[x^2\right] P = 1,
    \end{equation}
    \begin{equation}
    \left[x^3\right] P  = \xi,
    \end{equation}
    \begin{equation}
    \left[x^4\right] P  = 1,
    \end{equation}
    \begin{equation}
    \left[x^5\right] P = 3\xi,
    \end{equation}
    \begin{equation}
    \left[x^6\right] P = \xi + \xi^2,
    \end{equation}
    \begin{equation}
    \left[x^7\right] P = 2\xi,
    \end{equation}

    \begin{equation}
    \left[x^8\right] P = 1 + 4\xi^2 + \xi^3,
    \end{equation}
    \begin{equation}
    \left[x^9\right] P = 4\xi + 4\xi^2 + \xi^4,
    \end{equation}
    \begin{equation}
    \left[x^{10}\right] P = 3\xi + 9\xi^2 + 3\xi^3,
    \end{equation}
    \begin{equation}
    \left[x^{11}\right] P = 9\xi^2 + 6\xi^3,
    \end{equation}
    \begin{equation}
    \left[x^{12}\right] P = \xi + 13\xi^2 + 2\xi^3 + \xi^4,
    \end{equation}
    \begin{equation}
    \left[x^{13}\right] P = 6\xi + 4\xi^2 + 16\xi^3 + 15\xi^4,
    \end{equation}
    \begin{equation}
    \left[x^{14}\right] P = 2\xi + 18\xi^2 + 31\xi^3 + 12\xi^4 + 14\xi^5.
    \end{equation}
\end{multicols}
As example, the coefficient of $x^{12}$ of $\GramSerie_{\textnormal{mbi}}(x, 0, 0, 0, \xi)$
says that in the poset $(\EnsEq_{11}, \OrdTam)$, there is one maximal $1$-dimensional
hypercube, thirteen maximal $2$-dimensional hypercubes, two maximal $3$-dimensional
hypercubes and one maximal $4$-dimensional hypercube (see Figure~\ref{fig:InterEq}).
\medskip

Note that Proposition~\ref{prop:HauteursEgalesEq} implies that all binary
trees of the connected components of the posets $(\EnsEq_n, \OrdTam)$
have same height. However, the converse is false: There is two connected
components in the poset $(\EnsEq_5, \OrdTam)$ and each binary tree of $\EnsEq_5$
has same height.

\section{Intervals of other binary trees families in the Tamari lattice} \label{sec:AutresFamillesCloses}

\subsection{Generalized balanced binary trees}

\subsubsection{Definitions}
Let $V$ be a subset of $\EnsRel$. We say that a binary tree $T$ is \emph{$V$-balanced}
if for all node $x$ of $T$, $\Des_T(x) \in V$. Let us denote by
$\EnsEq^V$ the set of $V$-balanced binary trees. Note that the set of
balanced binary trees is $\EnsEq^{[-1, 1]}$. It is clear that $0$ must always
belongs to $V$ since a binary tree necessarily has a node with both
empty left and right subtrees; Otherwise, $\EnsEq^V$ would be empty.
A natural question about $V$-balanced binary trees demands to characterize
the sets $V$ such that $\EnsEq^V$ is closed by interval in the Tamari lattice.
\medskip

Let $T$ be a binary tree. Denote by $T^\InvolAB$ the binary tree obtained by
exchanging the right and left subtrees of each of its nodes. More formally,
\begin{equation}
    T^\InvolAB :=
    \begin{cases}
        R^\InvolAB \ABCons L^\InvolAB & \mbox{if $T = L \ABCons R$,} \\
        \ArbreVide                    & \mbox{otherwise ($T = \ArbreVide$).}
    \end{cases}
\end{equation}
For instance, one has
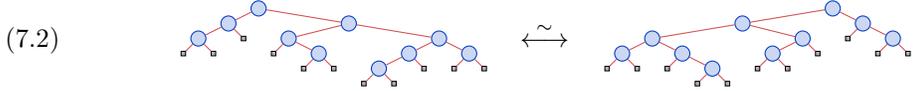
\begin{equation}
    \scalebox{.2}{\raisebox{-7.5em}{
    \begin{tikzpicture}
        \node[Feuille](0)at(0,-3){};
        \node[Noeud](1)at(1,-2){};
        \node[Feuille](2)at(2,-3){};
        \draw[Arete](1)--(0);
        \draw[Arete](1)--(2);
        \node[Noeud](3)at(3,-1){};
        \node[Feuille](4)at(4,-2){};
        \draw[Arete](3)--(1);
        \draw[Arete](3)--(4);
        \node[Noeud](5)at(5,0){};
        \node[Feuille](6)at(6,-3){};
        \node[Noeud](7)at(7,-2){};
        \node[Feuille](8)at(8,-4){};
        \node[Noeud](9)at(9,-3){};
        \node[Feuille](10)at(10,-4){};
        \draw[Arete](9)--(8);
        \draw[Arete](9)--(10);
        \draw[Arete](7)--(6);
        \draw[Arete](7)--(9);
        \node[Noeud](11)at(11,-1){};
        \node[Feuille](12)at(12,-5){};
        \node[Noeud](13)at(13,-4){};
        \node[Feuille](14)at(14,-5){};
        \draw[Arete](13)--(12);
        \draw[Arete](13)--(14);
        \node[Noeud](15)at(15,-3){};
        \node[Feuille](16)at(16,-4){};
        \draw[Arete](15)--(13);
        \draw[Arete](15)--(16);
        \node[Noeud](17)at(17,-2){};
        \node[Feuille](18)at(18,-4){};
        \node[Noeud](19)at(19,-3){};
        \node[Feuille](20)at(20,-4){};
        \draw[Arete](19)--(18);
        \draw[Arete](19)--(20);
        \draw[Arete](17)--(15);
        \draw[Arete](17)--(19);
        \draw[Arete](11)--(7);
        \draw[Arete](11)--(17);
        \draw[Arete](5)--(3);
        \draw[Arete](5)--(11);
    \end{tikzpicture}}}
    \quad \overset{\InvolAB}{\longleftrightarrow} \quad
    \scalebox{.2}{\raisebox{-7.5em}{
    \begin{tikzpicture}
        \node[Feuille](0)at(0,-4){};
        \node[Noeud](1)at(1,-3){};
        \node[Feuille](2)at(2,-4){};
        \draw[Arete](1)--(0);
        \draw[Arete](1)--(2);
        \node[Noeud](3)at(3,-2){};
        \node[Feuille](4)at(4,-4){};
        \node[Noeud](5)at(5,-3){};
        \node[Feuille](6)at(6,-5){};
        \node[Noeud](7)at(7,-4){};
        \node[Feuille](8)at(8,-5){};
        \draw[Arete](7)--(6);
        \draw[Arete](7)--(8);
        \draw[Arete](5)--(4);
        \draw[Arete](5)--(7);
        \draw[Arete](3)--(1);
        \draw[Arete](3)--(5);
        \node[Noeud](9)at(9,-1){};
        \node[Feuille](10)at(10,-4){};
        \node[Noeud](11)at(11,-3){};
        \node[Feuille](12)at(12,-4){};
        \draw[Arete](11)--(10);
        \draw[Arete](11)--(12);
        \node[Noeud](13)at(13,-2){};
        \node[Feuille](14)at(14,-3){};
        \draw[Arete](13)--(11);
        \draw[Arete](13)--(14);
        \draw[Arete](9)--(3);
        \draw[Arete](9)--(13);
        \node[Noeud](15)at(15,0){};
        \node[Feuille](16)at(16,-2){};
        \node[Noeud](17)at(17,-1){};
        \node[Feuille](18)at(18,-3){};
        \node[Noeud](19)at(19,-2){};
        \node[Feuille](20)at(20,-3){};
        \draw[Arete](19)--(18);
        \draw[Arete](19)--(20);
        \draw[Arete](17)--(16);
        \draw[Arete](17)--(19);
        \draw[Arete](15)--(9);
        \draw[Arete](15)--(17);
    \end{tikzpicture}}}\, .
\end{equation}
If $V$ is a subset of $\EnsRel$, let us also denote by $V^\InvolAB$ the set
$\{-v : v \in V\}$.

\subsubsection{A symmetry}

\begin{Lemme} \label{lem:OrdreDual}
    Let $T_0$ and $T_1$ be two binary trees such that $T_0 \OrdTam T_1$.
    Then, $T_1^\InvolAB \OrdTam T_0^\InvolAB$.
\end{Lemme}
\begin{proof}
    Assume that $S_0 \CouvTam S_1$ where $S_0 = (A \ABCons B) \ABCons C$
    and $S_1 = A \ABCons (B \ABCons C)$. Hence, we have
    $S_1^\InvolAB = (C^\InvolAB \ABCons B^\InvolAB) \ABCons A^\InvolAB$
    and $S_0^\InvolAB = C^\InvolAB \ABCons (B^\InvolAB \ABCons A^\InvolAB)$.
    Thus, $S_1^\InvolAB \CouvTam S_0^\InvolAB$, and the result follows
    from the fact that $\OrdTam$ is the reflexive and transitive closure of $\CouvTam$.
\end{proof}

\begin{Lemme} \label{lem:BijectionInvol}
    For all $V \subseteq \EnsRel$, the application $\InvolAB$ yields a bijection
    between the sets $\EnsEq^V$ and $\EnsEq^{V^\InvolAB}$.
\end{Lemme}
\begin{proof}
    It is immediate, from the definition of $\InvolAB$, that the application
    $\InvolAB$ is an involution. It then remains to show that if
    $T \in \EnsEq^V$, then $T^\InvolAB \in \EnsEq^{V^\InvolAB}$. Let $x$
    be a node of $T$ and $L$ (resp. $R$) be the left (resp. right) subtree
    of $x$. We have $v := \Des_T(x) = \Ht(R) - \Ht(L) \in V$. In $T^\InvolAB$,
    one has $\Des_{T^\InvolAB}(x) = \Ht(L^\InvolAB) - \Ht(R^\InvolAB) = \Ht(L) - \Ht(R) = -v \in V^\InvolAB$.
    Hence, $T^\InvolAB \in \EnsEq^{V^\InvolAB}$.
\end{proof}

\begin{Proposition} \label{prop:EnsEqVrV}
    For all $V \subseteq \EnsRel$, the set $\EnsEq^V$ is closed by interval
    in the Tamari lattice if and only if the set $\EnsEq^{V^\InvolAB}$ also is.
\end{Proposition}
\begin{proof}
    Assume that $\EnsEq^{V^\InvolAB}$ is closed by interval in the
    Tamari lattice. By contradiction, assume that there exist
    $T_0, T_2 \in \EnsEq^V$ and $T_1 \notin \EnsEq^V$ such that
    $T_0 \OrdTam T_1 \OrdTam T_2$. By Lemma~\ref{lem:OrdreDual}, we have
    $T_2^\InvolAB \OrdTam T_1^\InvolAB \OrdTam T_0^\InvolAB$, and by
    Lemma~\ref{lem:BijectionInvol}, $T_0^\InvolAB, T_2^\InvolAB \in \EnsEq^{V^\InvolAB}$
    and $T_1^\InvolAB \notin \EnsEq^{V^\InvolAB}$. That implies that $\EnsEq^{V^\InvolAB}$
    is not closed by interval in the Tamari lattice, which is contradictory
    with our hypothesis.
\end{proof}

\subsubsection{\texorpdfstring{$\{0, 1\}$-balanced binary trees}{0, 1-balanced binary trees}}
Using the methods developed in Section~\ref{sec:GramSync}, one can enumerate
$\{0, 1\}$-balanced binary trees according to their number of leaves, and
obtain the fixed-point functional equation
\begin{equation}
    \GramSerie_{01}(x, y) = x + \GramSerie_{01}(x^2 + xy, x),
\end{equation}
where the generating series of $\{0, 1\}$-balanced binary trees is the specialization
$\GramSerie_{01}(x, 0)$. First values are
\begin{equation}
    1, 1, 1, 1, 1, 2, 2, 2, 3, 5, 7, 9, 11, 13, 17, 26, 42, 66, 97, 134, 180, 241, 321, 424, 564, 774, 1111.
\end{equation}

\begin{Proposition} \label{prop:Arbres01Clos}
    The set of $\{0, 1\}$-balanced binary trees is closed by interval in
    The Tamari lattice.
\end{Proposition}
\begin{proof}
    Let $T_0 \in \EnsEq^{\{0, 1\}}$. Since $T_0$ is only composed of nodes
    with $0$ or $1$ as imbalance value, one can only perform into $T_0$ rotations
    of the kind~\ref{item:CasRot3},~\ref{item:CasRot5},~\ref{item:CasRot8}
    or~\ref{item:CasRot9} studied in Section~\ref{subsec:RotationsEquilibre}.
    Since these rotations are unbalancing rotations, for all binary tree
    $T_1$ such that $T_0 \CouvTam T_1$, $T_1$ is not balanced and hence,
    $T_1 \notin \EnsEq^{\{0, 1\}}$. By Theorem~\ref{thm:ClotureIntArbresEq},
    for all binary tree $T_2$ such that $T_1 \OrdTam T_2$, $T_2$ is not
    balanced, and with greater reason, $T_2 \notin \EnsEq^{\{0, 1\}}$.
    Therefore, $\EnsEq^{\{0, 1\}}$ is closed by interval in the Tamari lattice.
\end{proof}

The proof of Proposition~\ref{prop:Arbres01Clos} also shows that every
rotation performed into a $\{0, 1\}$-balanced binary tree gives a
$\{0, 1\}$-unbalanced binary tree. That implies that any pair of elements
of $\EnsEq^{\{0, 1\}}$ is incomparable.
\medskip

Computer trials suggest that for all $\beta \in \EnsRel$, any pair of elements
of $\EnsEq^{\{0, \beta\}}$ is incomparable. Hence, the sets $\EnsEq^{\{0, \beta\}}$
seem to be closed by interval in the Tamari lattice.

\subsubsection{\texorpdfstring{$[-\alpha, \beta]$-balanced binary trees}{-a, b-balanced binary trees}}
\begin{Lemme} \label{lem:MoinsV0NonClos}
    For all $\alpha \geq 2$, the sets $\EnsEq^{[-\alpha, 0]}$ and $\EnsEq^{]-\infty, 0]}$
    are not closed by interval in the Tamari lattice.
\end{Lemme}
\begin{proof}
    It is enough to exhibit a chain of the sort $T_0 \CouvTam T_1 \CouvTam T_2$
    where $T_0, T_2 \in \EnsEq^{[-\alpha, 0]} \cap \EnsEq^{]-\infty, 0]}$ and
    $T_1 \notin \EnsEq^{[-\alpha, 0]} \cup \EnsEq^{]-\infty, 0]}$. The following chain,
    where nodes are labeled by their imbalance values, is the case:
    \begin{equation}
        \scalebox{.25}{\raisebox{-4.5em}
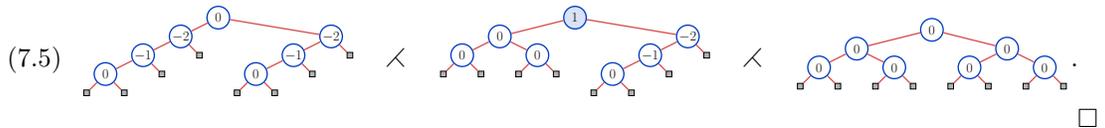
{
        \begin{tikzpicture}
            \node[Feuille](0)at(0,-4){};
            \node[Noeud,EtiqClair,minimum size = 12mm](1)at(1,-3){\Huge $0$};
            \node[Feuille](2)at(2,-4){};
            \draw[Arete](1)--(0);
            \draw[Arete](1)--(2);
            \node[Noeud,EtiqClair,minimum size = 12mm](3)at(3,-2){\Huge $-1$};
            \node[Feuille](4)at(4,-3){};
            \draw[Arete](3)--(1);
            \draw[Arete](3)--(4);
            \node[Noeud,EtiqClair,minimum size = 12mm](5)at(5,-1){\Huge $-2$};
            \node[Feuille](6)at(6,-2){};
            \draw[Arete](5)--(3);
            \draw[Arete](5)--(6);
            \node[Noeud,EtiqClair,minimum size = 12mm](7)at(7,0){\Huge $0$};
            \node[Feuille](8)at(8,-4){};
            \node[Noeud,EtiqClair,minimum size = 12mm](9)at(9,-3){\Huge $0$};
            \node[Feuille](10)at(10,-4){};
            \draw[Arete](9)--(8);
            \draw[Arete](9)--(10);
            \node[Noeud,EtiqClair,minimum size = 12mm](11)at(11,-2){\Huge $-1$};
            \node[Feuille](12)at(12,-3){};
            \draw[Arete](11)--(9);
            \draw[Arete](11)--(12);
            \node[Noeud,EtiqClair,minimum size = 12mm](13)at(13,-1){\Huge $-2$};
            \node[Feuille](14)at(14,-2){};
            \draw[Arete](13)--(11);
            \draw[Arete](13)--(14);
            \draw[Arete](7)--(5);
            \draw[Arete](7)--(13);
        \end{tikzpicture}}}
        \quad \CouvTam \quad
        \scalebox{.25}{\raisebox{-4.5em}{
        \begin{tikzpicture}
            \node[Feuille](0)at(0,-3){};
            \node[Noeud,EtiqClair,minimum size = 12mm](1)at(1,-2){\Huge $0$};
            \node[Feuille](2)at(2,-3){};
            \draw[Arete](1)--(0);
            \draw[Arete](1)--(2);
            \node[Noeud,EtiqClair,minimum size = 12mm](3)at(3,-1){\Huge $0$};
            \node[Feuille](4)at(4,-3){};
            \node[Noeud,EtiqClair,minimum size = 12mm](5)at(5,-2){\Huge $0$};
            \node[Feuille](6)at(6,-3){};
            \draw[Arete](5)--(4);
            \draw[Arete](5)--(6);
            \draw[Arete](3)--(1);
            \draw[Arete](3)--(5);
            \node[Noeud,EtiqFonce,minimum size = 12mm](7)at(7,0){\Huge $1$};
            \node[Feuille](8)at(8,-4){};
            \node[Noeud,EtiqClair,minimum size = 12mm](9)at(9,-3){\Huge $0$};
            \node[Feuille](10)at(10,-4){};
            \draw[Arete](9)--(8);
            \draw[Arete](9)--(10);
            \node[Noeud,EtiqClair,minimum size = 12mm](11)at(11,-2){\Huge $-1$};
            \node[Feuille](12)at(12,-3){};
            \draw[Arete](11)--(9);
            \draw[Arete](11)--(12);
            \node[Noeud,EtiqClair,minimum size = 12mm](13)at(13,-1){\Huge $-2$};
            \node[Feuille](14)at(14,-2){};
            \draw[Arete](13)--(11);
            \draw[Arete](13)--(14);
            \draw[Arete](7)--(3);
            \draw[Arete](7)--(13);
        \end{tikzpicture}}}
        \quad \CouvTam \quad
        \scalebox{.25}{\raisebox{-3.5em}{
        \begin{tikzpicture}
            \node[Feuille](0)at(0,-3){};
            \node[Noeud,EtiqClair,minimum size = 12mm](1)at(1,-2){\Huge $0$};
            \node[Feuille](2)at(2,-3){};
            \draw[Arete](1)--(0);
            \draw[Arete](1)--(2);
            \node[Noeud,EtiqClair,minimum size = 12mm](3)at(3,-1){\Huge $0$};
            \node[Feuille](4)at(4,-3){};
            \node[Noeud,EtiqClair,minimum size = 12mm](5)at(5,-2){\Huge $0$};
            \node[Feuille](6)at(6,-3){};
            \draw[Arete](5)--(4);
            \draw[Arete](5)--(6);
            \draw[Arete](3)--(1);
            \draw[Arete](3)--(5);
            \node[Noeud,EtiqClair,minimum size = 12mm](7)at(7,0){\Huge $0$};
            \node[Feuille](8)at(8,-3){};
            \node[Noeud,EtiqClair,minimum size = 12mm](9)at(9,-2){\Huge $0$};
            \node[Feuille](10)at(10,-3){};
            \draw[Arete](9)--(8);
            \draw[Arete](9)--(10);
            \node[Noeud,EtiqClair,minimum size = 12mm](11)at(11,-1){\Huge $0$};
            \node[Feuille](12)at(12,-3){};
            \node[Noeud,EtiqClair,minimum size = 12mm](13)at(13,-2){\Huge $0$};
            \node[Feuille](14)at(14,-3){};
            \draw[Arete](13)--(12);
            \draw[Arete](13)--(14);
            \draw[Arete](11)--(9);
            \draw[Arete](11)--(13);
            \draw[Arete](7)--(3);
            \draw[Arete](7)--(11);
        \end{tikzpicture}}}\, .
    \end{equation}
\end{proof}

\begin{Lemme} \label{lem:MoinsV1NonClos}
    For all $\alpha \geq 2$, the sets $\EnsEq^{[-\alpha, 1]}$ and $\EnsEq^{]-\infty, 1]}$
    are not closed by interval in the Tamari lattice.
\end{Lemme}
\begin{proof}
    It is enough to exhibit a chain of the sort $T_0 \CouvTam T_1 \CouvTam T_2$
    where $T_0, T_2 \in \EnsEq^{[-\alpha, 1]} \cap \EnsEq^{]-\infty, 1]}$ and
    $T_1 \notin \EnsEq^{[-\alpha, 1]} \cup \EnsEq^{]-\infty, 1]}$. The following
    chain, where nodes are labeled by their imbalance values, is the case:
    \begin{equation}
        \scalebox{.25}{\raisebox{-4.5em}
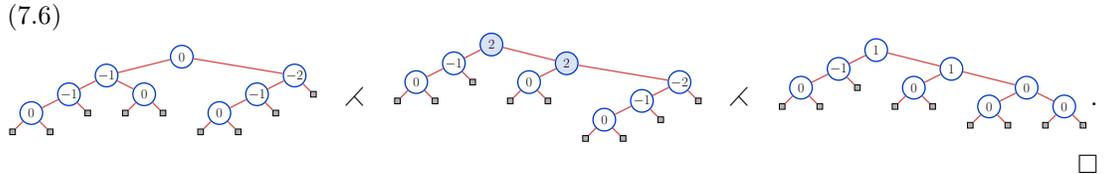
{
        \begin{tikzpicture}
            \node[Feuille](0)at(0,-4){};
            \node[Noeud,EtiqClair,minimum size = 12mm](1)at(1,-3){\Huge $0$};
            \node[Feuille](2)at(2,-4){};
            \draw[Arete](1)--(0);
            \draw[Arete](1)--(2);
            \node[Noeud,EtiqClair,minimum size = 12mm](3)at(3,-2){\Huge $-1$};
            \node[Feuille](4)at(4,-3){};
            \draw[Arete](3)--(1);
            \draw[Arete](3)--(4);
            \node[Noeud,EtiqClair,minimum size = 12mm](5)at(5,-1){\Huge $-1$};
            \node[Feuille](6)at(6,-3){};
            \node[Noeud,EtiqClair,minimum size = 12mm](7)at(7,-2){\Huge $0$};
            \node[Feuille](8)at(8,-3){};
            \draw[Arete](7)--(6);
            \draw[Arete](7)--(8);
            \draw[Arete](5)--(3);
            \draw[Arete](5)--(7);
            \node[Noeud,EtiqClair,minimum size = 12mm](9)at(9,0){\Huge $0$};
            \node[Feuille](10)at(10,-4){};
            \node[Noeud,EtiqClair,minimum size = 12mm](11)at(11,-3){\Huge $0$};
            \node[Feuille](12)at(12,-4){};
            \draw[Arete](11)--(10);
            \draw[Arete](11)--(12);
            \node[Noeud,EtiqClair,minimum size = 12mm](13)at(13,-2){\Huge $-1$};
            \node[Feuille](14)at(14,-3){};
            \draw[Arete](13)--(11);
            \draw[Arete](13)--(14);
            \node[Noeud,EtiqClair,minimum size = 12mm](15)at(15,-1){\Huge $-2$};
            \node[Feuille](16)at(16,-2){};
            \draw[Arete](15)--(13);
            \draw[Arete](15)--(16);
            \draw[Arete](9)--(5);
            \draw[Arete](9)--(15);
        \end{tikzpicture}}}
        \quad \CouvTam \quad
        \scalebox{.25}{\raisebox{-5.5em}{
        \begin{tikzpicture}
            \node[Feuille](0)at(0,-3){};
            \node[Noeud,EtiqClair,minimum size = 12mm](1)at(1,-2){\Huge $0$};
            \node[Feuille](2)at(2,-3){};
            \draw[Arete](1)--(0);
            \draw[Arete](1)--(2);
            \node[Noeud,EtiqClair,minimum size = 12mm](3)at(3,-1){\Huge $-1$};
            \node[Feuille](4)at(4,-2){};
            \draw[Arete](3)--(1);
            \draw[Arete](3)--(4);
            \node[Noeud,EtiqFonce,minimum size = 12mm](5)at(5,0){\Huge $2$};
            \node[Feuille](6)at(6,-3){};
            \node[Noeud,EtiqClair,minimum size = 12mm](7)at(7,-2){\Huge $0$};
            \node[Feuille](8)at(8,-3){};
            \draw[Arete](7)--(6);
            \draw[Arete](7)--(8);
            \node[Noeud,EtiqFonce,minimum size = 12mm](9)at(9,-1){\Huge $2$};
            \node[Feuille](10)at(10,-5){};
            \node[Noeud,EtiqClair,minimum size = 12mm](11)at(11,-4){\Huge $0$};
            \node[Feuille](12)at(12,-5){};
            \draw[Arete](11)--(10);
            \draw[Arete](11)--(12);
            \node[Noeud,EtiqClair,minimum size = 12mm](13)at(13,-3){\Huge $-1$};
            \node[Feuille](14)at(14,-4){};
            \draw[Arete](13)--(11);
            \draw[Arete](13)--(14);
            \node[Noeud,EtiqClair,minimum size = 12mm](15)at(15,-2){\Huge $-2$};
            \node[Feuille](16)at(16,-3){};
            \draw[Arete](15)--(13);
            \draw[Arete](15)--(16);
            \draw[Arete](9)--(7);
            \draw[Arete](9)--(15);
            \draw[Arete](5)--(3);
            \draw[Arete](5)--(9);
        \end{tikzpicture}}}
        \quad \CouvTam \quad
        \scalebox{.25}{\raisebox{-3.5em}{
        \begin{tikzpicture}
            \node[Feuille](0)at(0,-3){};
            \node[Noeud,EtiqClair,minimum size = 12mm](1)at(1,-2){\Huge $0$};
            \node[Feuille](2)at(2,-3){};
            \draw[Arete](1)--(0);
            \draw[Arete](1)--(2);
            \node[Noeud,EtiqClair,minimum size = 12mm](3)at(3,-1){\Huge $-1$};
            \node[Feuille](4)at(4,-2){};
            \draw[Arete](3)--(1);
            \draw[Arete](3)--(4);
            \node[Noeud,EtiqClair,minimum size = 12mm](5)at(5,0){\Huge $1$};
            \node[Feuille](6)at(6,-3){};
            \node[Noeud,EtiqClair,minimum size = 12mm](7)at(7,-2){\Huge $0$};
            \node[Feuille](8)at(8,-3){};
            \draw[Arete](7)--(6);
            \draw[Arete](7)--(8);
            \node[Noeud,EtiqClair,minimum size = 12mm](9)at(9,-1){\Huge $1$};
            \node[Feuille](10)at(10,-4){};
            \node[Noeud,EtiqClair,minimum size = 12mm](11)at(11,-3){\Huge $0$};
            \node[Feuille](12)at(12,-4){};
            \draw[Arete](11)--(10);
            \draw[Arete](11)--(12);
            \node[Noeud,EtiqClair,minimum size = 12mm](13)at(13,-2){\Huge $0$};
            \node[Feuille](14)at(14,-4){};
            \node[Noeud,EtiqClair,minimum size = 12mm](15)at(15,-3){\Huge $0$};
            \node[Feuille](16)at(16,-4){};
            \draw[Arete](15)--(14);
            \draw[Arete](15)--(16);
            \draw[Arete](13)--(11);
            \draw[Arete](13)--(15);
            \draw[Arete](9)--(7);
            \draw[Arete](9)--(13);
            \draw[Arete](5)--(3);
            \draw[Arete](5)--(9);
        \end{tikzpicture}}}\, .
    \end{equation}
\end{proof}

\begin{Lemme} \label{lem:Moins22NonClos}
    For all $\alpha \geq 2$, the sets $\EnsEq^{[-\alpha, 2]}$ and $\EnsEq^{]-\infty, 2]}$
    are not closed by interval in the Tamari lattice.
\end{Lemme}
\begin{proof}
    It is enough to exhibit a chain of the sort $T_0 \CouvTam T_1 \CouvTam T_2$
    where $T_0, T_2 \in \EnsEq^{[-\alpha, 2]} \cap \EnsEq^{]-\infty, 2]}$ and
    $T_1 \notin \EnsEq^{[-\alpha, 2]} \cup \EnsEq^{]-\infty, 2]}$. The following
    chain, where nodes are labeled by their imbalance values, is the case:
    \begin{equation}
        \scalebox{.25}{\raisebox{-3em}
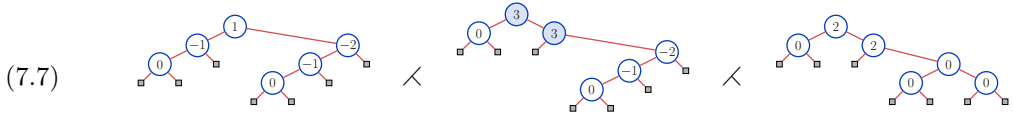
{
        \begin{tikzpicture}
            \node[Feuille](0)at(0,-3){};
            \node[Noeud,EtiqClair,minimum size = 12mm](1)at(1,-2){\Huge $0$};
            \node[Feuille](2)at(2,-3){};
            \draw[Arete](1)--(0);
            \draw[Arete](1)--(2);
            \node[Noeud,EtiqClair,minimum size = 12mm](3)at(3,-1){\Huge $-1$};
            \node[Feuille](4)at(4,-2){};
            \draw[Arete](3)--(1);
            \draw[Arete](3)--(4);
            \node[Noeud,EtiqClair,minimum size = 12mm](5)at(5,0){\Huge $1$};
            \node[Feuille](6)at(6,-4){};
            \node[Noeud,EtiqClair,minimum size = 12mm](7)at(7,-3){\Huge $0$};
            \node[Feuille](8)at(8,-4){};
            \draw[Arete](7)--(6);
            \draw[Arete](7)--(8);
            \node[Noeud,EtiqClair,minimum size = 12mm](9)at(9,-2){\Huge $-1$};
            \node[Feuille](10)at(10,-3){};
            \draw[Arete](9)--(7);
            \draw[Arete](9)--(10);
            \node[Noeud,EtiqClair,minimum size = 12mm](11)at(11,-1){\Huge $-2$};
            \node[Feuille](12)at(12,-2){};
            \draw[Arete](11)--(9);
            \draw[Arete](11)--(12);
            \draw[Arete](5)--(3);
            \draw[Arete](5)--(11);
        \end{tikzpicture}}}
        \quad \CouvTam \quad
        \scalebox{.25}{\raisebox{-4em}{
        \begin{tikzpicture}
            \node[Feuille](0)at(0,-2){};
            \node[Noeud,EtiqClair,minimum size = 12mm](1)at(1,-1){\Huge $0$};
            \node[Feuille](2)at(2,-2){};
            \draw[Arete](1)--(0);
            \draw[Arete](1)--(2);
            \node[Noeud,EtiqFonce,minimum size = 12mm](3)at(3,0){\Huge $3$};
            \node[Feuille](4)at(4,-2){};
            \node[Noeud,EtiqFonce,minimum size = 12mm](5)at(5,-1){\Huge $3$};
            \node[Feuille](6)at(6,-5){};
            \node[Noeud,EtiqClair,minimum size = 12mm](7)at(7,-4){\Huge $0$};
            \node[Feuille](8)at(8,-5){};
            \draw[Arete](7)--(6);
            \draw[Arete](7)--(8);
            \node[Noeud,EtiqClair,minimum size = 12mm](9)at(9,-3){\Huge $-1$};
            \node[Feuille](10)at(10,-4){};
            \draw[Arete](9)--(7);
            \draw[Arete](9)--(10);
            \node[Noeud,EtiqClair,minimum size = 12mm](11)at(11,-2){\Huge $-2$};
            \node[Feuille](12)at(12,-3){};
            \draw[Arete](11)--(9);
            \draw[Arete](11)--(12);
            \draw[Arete](5)--(4);
            \draw[Arete](5)--(11);
            \draw[Arete](3)--(1);
            \draw[Arete](3)--(5);
        \end{tikzpicture}}}
        \quad \CouvTam \quad
        \scalebox{.25}{\raisebox{-3em}{
        \begin{tikzpicture}
            \node[Feuille](0)at(0,-2){};
            \node[Noeud,EtiqClair,minimum size = 12mm](1)at(1,-1){\Huge $0$};
            \node[Feuille](2)at(2,-2){};
            \draw[Arete](1)--(0);
            \draw[Arete](1)--(2);
            \node[Noeud,EtiqClair,minimum size = 12mm](3)at(3,0){\Huge $2$};
            \node[Feuille](4)at(4,-2){};
            \node[Noeud,EtiqClair,minimum size = 12mm](5)at(5,-1){\Huge $2$};
            \node[Feuille](6)at(6,-4){};
            \node[Noeud,EtiqClair,minimum size = 12mm](7)at(7,-3){\Huge $0$};
            \node[Feuille](8)at(8,-4){};
            \draw[Arete](7)--(6);
            \draw[Arete](7)--(8);
            \node[Noeud,EtiqClair,minimum size = 12mm](9)at(9,-2){\Huge $0$};
            \node[Feuille](10)at(10,-4){};
            \node[Noeud,EtiqClair,minimum size = 12mm](11)at(11,-3){\Huge $0$};
            \node[Feuille](12)at(12,-4){};
            \draw[Arete](11)--(10);
            \draw[Arete](11)--(12);
            \draw[Arete](9)--(7);
            \draw[Arete](9)--(11);
            \draw[Arete](5)--(4);
            \draw[Arete](5)--(9);
            \draw[Arete](3)--(1);
            \draw[Arete](3)--(5);
        \end{tikzpicture}}}\, .
    \end{equation}
\end{proof}

\begin{Lemme} \label{lem:MoinsVWNonClos}
    For all $\alpha \geq 2$ and $\beta \geq 3$, the sets $\EnsEq^{[-\alpha, \beta]}$ and
    $\EnsEq^{]-\infty, \beta]}$ are not closed by interval in the Tamari lattice.
\end{Lemme}
\begin{proof}
    It is enough to exhibit a chain of the sort $T_0 \CouvTam T_1 \CouvTam T_2$
    where $T_0, T_2 \in \EnsEq^{[-\alpha, \beta]} \cap \EnsEq^{]-\infty, \beta]}$
    and $T_1 \notin \EnsEq^{[-\alpha, \beta]} \cup \EnsEq^{]-\infty, \beta]}$.
    By setting $\beta' := \beta - 1$ and $\beta'' := \beta + 1$, the following generic chain,
    where nodes are labeled by their imbalance values, and where the edges
    depicted by
    \scalebox{.4}{
    \begin{tikzpicture}
        \draw[Arete, decorate, decoration = zigzag](0,0)--(1,-.6);
    \end{tikzpicture}}
    denote a right comb binary tree with $\beta - 3$ nodes, is the case:
    \begin{equation}
        \scalebox{.25}{\raisebox{-3em}
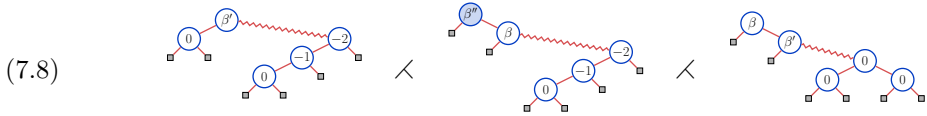
{
        \begin{tikzpicture}
            \node[Feuille](0)at(0.0,-2){};
            \node[Noeud,EtiqClair,minimum size = 12mm](1)at(1.0,-1){\Huge $0$};
            \node[Feuille](2)at(2.0,-2){};
            \draw[Arete](1)--(0);
            \draw[Arete](1)--(2);
            \node[Noeud,EtiqClair,minimum size = 12mm](3)at(3.0,0){\Huge $\beta'$};
            \node[Feuille](4)at(4.0,-4){};
            \node[Noeud,EtiqClair,minimum size = 12mm](5)at(5.0,-3){\Huge $0$};
            \node[Feuille](6)at(6.0,-4){};
            \draw[Arete](5)--(4);
            \draw[Arete](5)--(6);
            \node[Noeud,EtiqClair,minimum size = 12mm](7)at(7.0,-2){\Huge $-1$};
            \node[Feuille](8)at(8.0,-3){};
            \draw[Arete](7)--(5);
            \draw[Arete](7)--(8);
            \node[Noeud,EtiqClair,minimum size = 12mm](9)at(9.0,-1){\Huge $-2$};
            \node[Feuille](10)at(10.0,-2){};
            \draw[Arete](9)--(7);
            \draw[Arete](9)--(10);
            \draw[Arete](3)--(1);
            \draw[Arete, decorate, decoration = zigzag](3)--(9);
        \end{tikzpicture}}}
        \quad \CouvTam \quad
        \scalebox{.25}{\raisebox{-5em}{
        \begin{tikzpicture}
            \node[Feuille](0)at(0.0,-1){};
            \node[Noeud,minimum size = 12mm](1)at(1.0,0){\Huge $\beta''$};
            \node[Feuille](2)at(2.0,-2){};
            \node[Noeud,EtiqClair,minimum size = 12mm](3)at(3.0,-1){\Huge $\beta$};
            \node[Feuille](4)at(4.0,-5){};
            \node[Noeud,EtiqClair,minimum size = 12mm](5)at(5.0,-4){\Huge $0$};
            \node[Feuille](6)at(6.0,-5){};
            \draw[Arete](5)--(4);
            \draw[Arete](5)--(6);
            \node[Noeud,EtiqClair,minimum size = 12mm](7)at(7.0,-3){\Huge $-1$};
            \node[Feuille](8)at(8.0,-4){};
            \draw[Arete](7)--(5);
            \draw[Arete](7)--(8);
            \node[Noeud,EtiqClair,minimum size = 12mm](9)at(9.0,-2){\Huge $-2$};
            \node[Feuille](10)at(10.0,-3){};
            \draw[Arete](9)--(7);
            \draw[Arete](9)--(10);
            \draw[Arete](3)--(2);
            \draw[Arete, decorate, decoration = zigzag](3)--(9);
            \draw[Arete](1)--(0);
            \draw[Arete](1)--(3);
        \end{tikzpicture}}}
        \quad \CouvTam \quad
        \scalebox{.25}{\raisebox{-3.5em}{
        \begin{tikzpicture}
            \node[Feuille](0)at(0.0,-1){};
            \node[Noeud,EtiqClair,minimum size = 12mm](1)at(1.0,0){\Huge $\beta$};
            \node[Feuille](2)at(2.0,-2){};
            \node[Noeud,EtiqClair,minimum size = 12mm](3)at(3.0,-1){\Huge $\beta'$};
            \node[Feuille](4)at(4.0,-4){};
            \node[Noeud,EtiqClair,minimum size = 12mm](5)at(5.0,-3){\Huge $0$};
            \node[Feuille](6)at(6.0,-4){};
            \draw[Arete](5)--(4);
            \draw[Arete](5)--(6);
            \node[Noeud,EtiqClair,minimum size = 12mm](7)at(7.0,-2){\Huge $0$};
            \node[Feuille](8)at(8.0,-4){};
            \node[Noeud,EtiqClair,minimum size = 12mm](9)at(9.0,-3){\Huge $0$};
            \node[Feuille](10)at(10.0,-4){};
            \draw[Arete](9)--(8);
            \draw[Arete](9)--(10);
            \draw[Arete](7)--(5);
            \draw[Arete](7)--(9);
            \draw[Arete](3)--(2);
            \draw[Arete, decorate, decoration = zigzag](3)--(7);
            \draw[Arete](1)--(0);
            \draw[Arete](1)--(3);
        \end{tikzpicture}}}\, .
    \end{equation}
\end{proof}

\begin{Theoreme} \label{thm:EnsClosInterDeZ}
    Let $V$ be an interval of $\EnsRel$ containing $0$. The set $\EnsEq^V$
    is closed by interval in the Tamari lattice if and only if
    $V \in \left\{ \{0\}, \{-1, 0\}, \{0, 1\}, \{-1, 0, 1\}, \EnsRel\right\}$.
\end{Theoreme}
\begin{proof}
    Since $\EnsEq^{\{0\}}$ only contains perfect binary trees and there
    is at most one such element with a given number of nodes, $\EnsEq^{\{0\}}$
    is closed by interval. Moreover, by Proposition~\ref{prop:Arbres01Clos},
    $\EnsEq^{\{0, 1\}}$ is closed by interval, and by Proposition~\ref{prop:EnsEqVrV},
    $\EnsEq^{\{-1, 0\}}$ also is. By Theorem~\ref{thm:ClotureIntArbresEq},
    $\EnsEq^{\{-1, 0, 1\}}$ is closed by interval. Finally, since
    $\EnsEq^\EnsRel = \EnsAB$, $\EnsEq^\EnsRel$ is obviously closed
    by interval.

    If $V$ is an interval of $\EnsRel$ containing $0$ and that does not fit
    into the previous cases, necessarily $V$ or $V^\InvolAB$ satisfies the
    assumptions of Lemma~\ref{lem:MoinsV0NonClos},~\ref{lem:MoinsV1NonClos},~\ref{lem:Moins22NonClos},
    or~\ref{lem:MoinsVWNonClos}. Thus, by Proposition~\ref{prop:EnsEqVrV},
    $\EnsEq^V$ is not closed by interval.
\end{proof}

Theorem~\ref{thm:EnsClosInterDeZ} emphasizes the special role played by
balanced binary trees in the Tamari lattice. Indeed, the interval $V := [-1, 1]$
of $\EnsRel$ is the only interval different from $\EnsRel$ such that $\EnsEq^V$ is closed
by interval in the Tamari lattice and such that the subposet of the Tamari
lattice induced by $\EnsEq^V$ contains nontrivial intervals (see Theorem~\ref{thm:IntEqHypercube}
and Figure~\ref{fig:InterEq}).

\subsection{Weight balanced binary trees}

Denote by $\Nd(T)$ the number of nodes of the binary tree $T$. Let us define
the \emph{weight imbalance mapping} $\WDes_T$ which associates an element
of $\EnsRel$ with a node $x$ of $T$. It is defined by
\begin{equation}
    \WDes_T(x) := \Nd(R) - \Nd(L),
\end{equation}
where $L$ (resp. $R$) is the left (resp. right) subtree of $x$. A node $x$
is \emph{weight balanced} if
\begin{equation}
    \WDes_T(x) \in \{-1, 0, 1\}.
\end{equation}

\begin{Definition}
    A binary tree $T$ is \emph{weight balanced} if all nodes of $T$ are weight balanced.
\end{Definition}

The sequence $(w_n)_{n \geq 0}$ of the number of weight balanced binary
trees with $n$ nodes satisfies straightforwardly the recurrence relation
\begin{equation}
    w_n =
    \begin{cases}
        1               & \mbox{if $n \in \{0, 1\}$}, \\
        2 w_k w_{k - 1} & \mbox{if $n = 2k$,} \\
        w^2_k           & \mbox{where $n = 2k + 1$, otherwise.}
    \end{cases}
\end{equation}
This is Sequence \Sloane{A110316} of~\cite{SLO08}. First values are
\begin{equation}
    1, 1, 2, 1, 4, 4, 4, 1, 8, 16, 32, 16, 32, 16, 8, 1, 16, 64, 256, 256, 1024, 1024.
\end{equation}

\begin{Lemme} \label{lem:RelationHtNdWEq}
    For all nonempty weight balanced binary tree $T$, the following relation
    between its height and its number of nodes holds
    \begin{equation} \label{eq:RelationHtNdWEq}
        \Ht(T) = \left \lfloor \log_2(\Nd(T)) \right \rfloor + 1.
    \end{equation}
\end{Lemme}
\begin{proof}
    We proceed by structural induction on the set of nonempty weight balanced binary trees.
    The lemma is true for the one-node binary tree. Assume now that~(\ref{eq:RelationHtNdWEq})
    holds for both the weight balanced binary trees $L$ and $R$ such that $T := L \ABCons R$
    is weight balanced. We have now two cases to consider, depending if
    $L$ and $R$ have the same number of nodes or not. If $\Nd(L) = \Nd(R)$, set $k := \Nd(L)$. We have
    \begin{align}
        \left \lfloor \log_2(\Nd(T)) \right \rfloor + 1 & = \left \lfloor \log_2(2k + 1) \right \rfloor + 1 \\
                                                        & = \left \lfloor \log_2(2) + \log_2\left(k + 1/2\right) \right \rfloor + 1 \\
                                                        & = \left \lfloor \log_2\left(k + 1/2\right) \right \rfloor + 2 \label{eq:PreuveNdWeq1} \\
                                                        & = \left \lfloor \log_2(k) \right \rfloor + 2 \label{eq:PreuveNdWeq2} \\
                                                        & = \Ht(L) + 1 \label{eq:PreuveNdWeq3} \\
                                                        & = \Ht(R) + 1 \\
                                                        & = \Ht(T).
    \end{align}
    The equality between~(\ref{eq:PreuveNdWeq1}) and~(\ref{eq:PreuveNdWeq2})
    is provided by the fact that $k$ is an integer. The equality between~(\ref{eq:PreuveNdWeq2})
    and~(\ref{eq:PreuveNdWeq3}) follows by induction hypothesis.

    If $\Nd(L) \ne \Nd(R)$, assume without lost of generality that
    $\Nd(L) = \Nd(R) + 1$ and set $k := \Nd(L)$. An analog computation
    as above implies~(\ref{eq:RelationHtNdWEq}).
\end{proof}

\begin{Proposition} \label{prop:WEqSousEnsembleEq}
    The set of weight balanced binary trees is a subset of the set of the
    (height) balanced binary trees.
\end{Proposition}
\begin{proof}
    We proceed by structural induction on the set of weight balanced binary trees
    to show that each weight balanced binary tree is also (height) balanced.
    This property is true for both the empty tree and the one-node binary tree.
    Assume now that this property holds for two weight balanced binary trees $L$
    and $R$ such that $T := L \ABCons R$ is weight balanced. By Lemma~\ref{lem:RelationHtNdWEq},
    we have
    \begin{equation}
        \Ht(R) - \Ht(L) = \lfloor \log_2(\Nd(R)) \rfloor - \lfloor \log_2(\Nd(L)) \rfloor,
    \end{equation}
    and since $T$ is weight balanced, we have $|\Nd(R) - \Nd(L)| \leq 1$
    so that $|\Ht(R) - \Ht(L)| \leq 1$. By induction hypothesis, $L$ and
    $R$ are (height) balanced, proving that $T$ also is.
\end{proof}

\begin{Proposition} \label{prop:ClotureIntArbresWEq}
    Let $T_0$ and $T_1$ be two weight balanced binary trees such that
    $T_0 \OrdTam T_1$. Then, the interval $[T_0, T_1]$ only contains weight
    balanced binary trees.
\end{Proposition}
\begin{proof}
    Let us show that for all binary tree $T$, any rotation operation performed
    into $T$ does not decrease the weight imbalance values of any node of $T$.
    Let $y$ be a node in $T$ and $x$ its left child. Let $(A \ABCons B) \ABCons C$
    be the subtree of root $y$ in $T$. Let $T'$ be the binary tree obtained by
    the rotation of root $y$ from $T$. We have the following weight imbalance
    values:
    \begin{equation}
        \begin{cases}
            \WDes_{T}(x) = \Nd(B) - \Nd(A), \\
            \WDes_{T}(y) = \Nd(C) - \Nd(B) - \Nd(A) - 1,
        \end{cases}
    \end{equation}
    and
    \begin{equation}
        \begin{cases}
            \WDes_{T'}(x) = \Nd(B) + \Nd(C) + 1 - \Nd(A), \\
            \WDes_{T'}(y) = \Nd(C) - \Nd(B),
        \end{cases}
    \end{equation}
    showing that $\WDes_{T'}(x) > \WDes_{T}(x)$ and $\WDes_{T'}(y) > \WDes_{T}(y)$.
    Besides, note that the rotation does not change the weight imbalance
    values of the other nodes of $T$.

    This shows that the set of weight balanced binary trees is closed by interval
    in the Tamari lattice since, by starting from a weight balanced binary tree $T$
    and by performing a rotation that gives a weight unbalanced binary tree $T'$, there
    exists a node $z$ of $T'$ such that $\WDes_{T'}(z) \geq 2$ and it is
    impossible to decrease this value so that each binary tree greater than $T'$
    is not weight balanced.
\end{proof}

Note that the proof of Proposition~\ref{prop:ClotureIntArbresWEq} also proves
that for all $k \geq 0$, the sets of \emph{$k$-weight balanced binary trees},
that are the sets of binary trees $T$ such that for all node $x$ of $T$, the
relation $\left|\WDes_T(x)\right| \leq k$ holds, are closed by interval in the Tamari
lattice.
\medskip

Since by Proposition~\ref{prop:WEqSousEnsembleEq}, weight balanced binary
trees are also (height) balanced, by Proposition~\ref{prop:ClotureIntArbresWEq}
and Theorem~\ref{thm:IntEqHypercube}, the intervals of weight balanced binary
trees are isomorphic to a hypercube. However, the set of weight balanced
binary trees has an additional property:

\begin{Proposition} \label{prop:ArbresWEqGradue}
    The restriction of the Tamari order on the set of weight balanced binary
    trees is a graded poset.
\end{Proposition}
\begin{proof}
    Let us characterize the conservative weight balancing rotations. Let
    $T_0 := (A \ABCons B) \ABCons C$ and $T_1 := A \ABCons (B \ABCons C)$
    be two weight balanced binary trees such that $T_1$ is obtained by
    a rotation at the root $y$ of $T_0$. Denote by $x$ the left child of
    $y$ in $T_0$. Note that the rotation that transforms $T_0$ into $T_1$
    cannot be a conservative weight balancing rotation if $\WDes_{T_0}(x) = 1$
    or $\WDes_{T_0}(y) = 1$ since, following the proof of Proposition~\ref{prop:ClotureIntArbresWEq},
    the imbalance values of $x$ and $y$ both increase after a rotation. Here
    follows the list of the weight imbalance values of the nodes $x$ and
    $y$ in $T_0$ and $T_1$ expressed as
    $(\WDes_{T_0}(x), \WDes_{T_0}(y)) \longrightarrow (\WDes_{T_1}(x), \WDes_{T_1}(y))$:
    \begin{multicols}{2}
        \begin{enumerate}[label = (R'\arabic*)]
            \item $(-1, -1) \longrightarrow (2\Nd(A) - 1, \Nd(A))$, \vspace*{1em} \label{eq:CasWRot1}
            \item $(0, -1) \longrightarrow (2\Nd(A) + 1, \Nd(A))$, \label{eq:CasWRot2}
            \item $(-1, 0) \longrightarrow (2\Nd(A), \Nd(A) + 1)$, \vspace*{1em} \label{eq:CasWRot3}
            \item $(0, 0) \longrightarrow (2\Nd(A) + 2, \Nd(A) + 1)$. \label{eq:CasWRot4}
        \end{enumerate}
    \end{multicols}
    Hence, we have four kind of rotations to explore:
    \begin{enumerate}[label = \underline{\bf Case \arabic*:}, fullwidth]
        \item Regarding~\ref{eq:CasWRot1}, we necessarily have $\Nd(A) = 1$. Indeed,
        if $\Nd(A) \geq 2$, $y$  would not be weight balanced in $T_1$, and if
        $\Nd(A) = 0$, since $\WDes_{T_0}(x) = -1$, that would imply that $\Nd(B) = -1$,
        which is absurd. Hence, since $\Nd(A) = 1$, we have $\Nd(B) = 0$ and $\Nd(C) = 1$.
        Thus, there is only one pair $(T_0, T_1)$ satisfying this kind of conservative
        weight balancing rotation:
        \begin{equation} \label{eq:RotationConsW1}
            T_0 =
            \scalebox{.25}{%
            \raisebox{-4em}{%
            \begin{tikzpicture}
                \node[Feuille](0)at(0,-3){};
                \node[Noeud](1)at(1,-2){};
                \node[Feuille](2)at(2,-3){};
                \draw[Arete](1)--(0);
                \draw[Arete](1)--(2);
                \node[Noeud](3)at(3,-1){};
                \node[Feuille](4)at(4,-2){};
                \draw[Arete](3)--(1);
                \draw[Arete](3)--(4);
                \node[Noeud](5)at(5,0){};
                \node[Feuille](6)at(6,-2){};
                \node[Noeud](7)at(7,-1){};
                \node[Feuille](8)at(8,-2){};
                \draw[Arete](7)--(6);
                \draw[Arete](7)--(8);
                \draw[Arete](5)--(3);
                \draw[Arete](5)--(7);
            \end{tikzpicture}}}
            \quad \longrightarrow \quad
            \scalebox{.25}{%
            \raisebox{-4em}{%
            \begin{tikzpicture}
                \node[Feuille](0)at(0,-2){};
                \node[Noeud](1)at(1,-1){};
                \node[Feuille](2)at(2,-2){};
                \draw[Arete](1)--(0);
                \draw[Arete](1)--(2);
                \node[Noeud](3)at(3,0){};
                \node[Feuille](4)at(4,-2){};
                \node[Noeud](5)at(5,-1){};
                \node[Feuille](6)at(6,-3){};
                \node[Noeud](7)at(7,-2){};
                \node[Feuille](8)at(8,-3){};
                \draw[Arete](7)--(6);
                \draw[Arete](7)--(8);
                \draw[Arete](5)--(4);
                \draw[Arete](5)--(7);
                \draw[Arete](3)--(1);
                \draw[Arete](3)--(5);
            \end{tikzpicture}}}
            = T_1.
        \end{equation}
        \item Concerning~\ref{eq:CasWRot2}, we necessarily have $\Nd(A) = 0$.
        Indeed, if $\Nd(A) \geq 1$, $x$ would not be weight balanced in
        $T_1$. Hence, since $\Nd(A) = 0$, we have $\Nd(B) = 0$ and $\Nd(C) = 0$.
        Thus, there is only one pair $(S_0, S_1)$ that satisfies this kind
        of conservative weight balancing rotation:
        \begin{equation}
            S_0 = \label{eq:RotationConsW2}
            \scalebox{.25}{%
            \raisebox{-2em}{%
            \begin{tikzpicture}
                \node[Feuille](0)at(0,-2){};
                \node[Noeud](1)at(1,-1){};
                \node[Feuille](2)at(2,-2){};
                \draw[Arete](1)--(0);
                \draw[Arete](1)--(2);
                \node[Noeud](3)at(3,0){};
                \node[Feuille](4)at(4,-1){};
                \draw[Arete](3)--(1);
                \draw[Arete](3)--(4);
            \end{tikzpicture}}}
            \quad \longrightarrow \quad
            \scalebox{.25}{%
            \raisebox{-2em}{%
            \begin{tikzpicture}
                \node[Feuille](0)at(0,-1){};
                \node[Noeud](1)at(1,0){};
                \node[Feuille](2)at(2,-2){};
                \node[Noeud](3)at(3,-1){};
                \node[Feuille](4)at(4,-2){};
                \draw[Arete](3)--(2);
                \draw[Arete](3)--(4);
                \draw[Arete](1)--(0);
                \draw[Arete](1)--(3);
            \end{tikzpicture}}}
            = S_1.
        \end{equation}
        \item Regarding~\ref{eq:CasWRot3}, we necessarily have $\Nd(A) = 0$.
        That implies $\Ht(B) = -1$, which is absurd. Hence,~\ref{eq:CasWRot3}
        cannot be a conservative weight balancing rotation.
        \item Concerning~\ref{eq:CasWRot4}, $x$ satisfies $\WDes_{T_1}(x) \geq 2$,
        and thus~\ref{eq:CasWRot4} is not a case of a conservative weight
        balancing rotation.
    \end{enumerate}
    \smallskip

    Hence, we only have two sorts of conservative weight balancing rotations.
    They are the ones depicted in~(\ref{eq:RotationConsW1}) and~(\ref{eq:RotationConsW2}).

    Since each such rotation suppresses exactly one subtree of the form
    $S_0$ and adds exactly one subtree of the form $S_1$, we can define a
    map $\phi : \EnsAB \rightarrow \EnsNat$ where $\phi(T)$
    is the number of subtrees of the form $S_1$ in $T$. Hence, since
    by Proposition~\ref{prop:ClotureIntArbresWEq} the covering relations
    of the Tamari lattice restricted to the weight balanced binary trees
    are only conservative weight balancing rotations, the statistic
    $\phi$ is a ranking function of the Tamari lattice restricted to
    these elements, and shows that this poset is graded.
\end{proof}

\subsection{Binary trees with fixed canopy}
The \emph{canopy} $\Canop(T)$ (see~\cite{LR98} and~\cite{V04}) of a binary
tree $T$ is the word on the alphabet $\{0, 1\}$ obtained by browsing
the leaves of $T$ from left to right except the first and the last one,
writing $0$ if the considered leaf is oriented to the right, $1$ otherwise
(see Figure~\ref{fig:ExempleCanopee}).
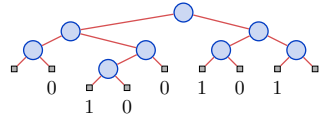
\begin{figure}[ht]
    \centering
    \scalebox{.25}{%
    \begin{tikzpicture}
        \node[Feuille](0)at(0,-3){};
        \node[Noeud](1)at(1,-2){};
        \node[Feuille](2)at(2,-3){};
        \node[] (2') [below of = 2] {\scalebox{3}{$0$}};
        \draw[Arete](1)--(0);
        \draw[Arete](1)--(2);
        \node[Noeud](3)at(3,-1){};
        \node[Feuille](4)at(4,-4){};
        \node[] (4') [below of = 4] {\scalebox{3}{$1$}};
        \node[Noeud](5)at(5,-3){};
        \node[Feuille](6)at(6,-4){};
        \node[] (6') [below of = 6] {\scalebox{3}{$0$}};
        \draw[Arete](5)--(4);
        \draw[Arete](5)--(6);
        \node[Noeud](7)at(7,-2){};
        \node[Feuille](8)at(8,-3){};
        \node[] (8') [below of = 8] {\scalebox{3}{$0$}};
        \draw[Arete](7)--(5);
        \draw[Arete](7)--(8);
        \draw[Arete](3)--(1);
        \draw[Arete](3)--(7);
        \node[Noeud](9)at(9,0){};
        \node[Feuille](10)at(10,-3){};
        \node[] (10') [below of = 10] {\scalebox{3}{$1$}};
        \node[Noeud](11)at(11,-2){};
        \node[Feuille](12)at(12,-3){};
        \node[] (12') [below of = 12] {\scalebox{3}{$0$}};
        \draw[Arete](11)--(10);
        \draw[Arete](11)--(12);
        \node[Noeud](13)at(13,-1){};
        \node[Feuille](14)at(14,-3){};
        \node[] (14') [below of = 14] {\scalebox{3}{$1$}};
        \node[Noeud](15)at(15,-2){};
        \node[Feuille](16)at(16,-3){};
        \draw[Arete](15)--(14);
        \draw[Arete](15)--(16);
        \draw[Arete](13)--(11);
        \draw[Arete](13)--(15);
        \draw[Arete](9)--(3);
        \draw[Arete](9)--(13);
    \end{tikzpicture}}
    \caption{The canopy of this binary tree is $0100101$.}
    \label{fig:ExempleCanopee}
\end{figure}

For all $u \in \{0, 1\}^*$, define the set $\EnsCanop_u$ by
\begin{equation}
    \EnsCanop_u := \left\{T \in \EnsAB : \Canop(T) = u\right\}.
\end{equation}

Note that the sets of binary trees with a given canopy play a role in a injective
Hopf morphism relating the Hopf algebra of noncommutative symmetric functions
$\Sym$~\cite{GKDLLRT94} and the Hopf algebra of binary trees $\PBT$~\cite{LR98, HNT05-2}.
Recall that the fundamental basis of $\PBT$ is $\left\{\PP_T\right\}_{T \in \EnsAB}$
and is indexed by binary trees. One can see the fundamental basis of $\Sym$
as a basis $\left\{\PP_u\right\}_{u \in \{0, 1\}^*}$ indexed by binary words.
The injective Hopf morphism $\beta : \Sym \hookrightarrow \PBT$ also satisfies
(see~\cite{GirBX})
\begin{equation}
    \beta(\PP_u) = \sum_{T \; \in \; \EnsCanop_u} \PP_T.
\end{equation}

\begin{Proposition} \label{prop:ClotureIntArbresCanopee}
    For all $u \in \{0, 1\}^*$, the set $\EnsCanop_u$ is an interval of
    the Tamari lattice.
\end{Proposition}
\begin{proof}
    Let us prove first that $\EnsCanop_u$ is closed by interval in the
    Tamari lattice. Consider a binary tree $T_0$ and $y$ one of its nodes.
    Let $(A \ABCons B) \ABCons C$ the subtree of $T_0$ of root $y$ and $T_1$
    be the binary tree obtained by the rotation of root $y$ from $T_0$.
    Regardless $A$ and $C$, if $B$ is not empty, we have $\Canop(T_0) = \Canop(T_1)$;
    Otherwise, $B$ is a leaf and its orientation changes from right to left.
    Thus, $\Canop(T_1)$ is lexicographically not smaller than $\Canop(T_0)$,
    which proves that $\EnsCanop_u$ is closed by interval.

    We give now a counting argument to prove that $\EnsCanop_u$ also is
    an interval of the Tamari lattice. Let $T$ be a maximal element among
    $\EnsCanop_u$. Thus, each rotation changes the canopy of $T$, and hence,
    for every node $y$ which has a left child $x$ in $T$, $x$ has no right
    child. The set of such maximal binary trees, denoted $\mathcal{M}$, is
    characterized by the following regular specification (see~\cite{FS09}
    for a general survey on regular specifications):
    \begin{equation}
        \mathcal{M} = \mathcal{L} \times \left\{\Noeud{.4}{}\right\} \times \mathcal{M} \enspace + \enspace \{\ArbreVide\},
    \end{equation}
    where $\mathcal{L}$ is the set of left comb binary trees. It admits the
    following generating series $M(x)$, which enumerates the elements of
    $\mathcal{M}$ according to their number of nodes:
    \begin{equation}
        M(x) = \frac{1 - x}{1 - 2x} = 1 + \sum_{n \geq 1} 2^{n - 1} x ^n.
    \end{equation}
    Moreover, for all $n \geq 1$ there are exactly $2^{n - 1}$ sets $\EnsCanop_u$
    where $\ell(u) = n - 1$, and there are the same number of such maximal binary
    trees. That implies that there is exactly one maximal element in each
    $\EnsCanop_u$. By the same reasoning, we can show that there is exactly
    one minimal tree in each $\EnsCanop_u$, proving the result.
\end{proof}

The statement of Proposition~\ref{prop:ClotureIntArbresCanopee} is
already known~\cite{LR02}, only our proof is new.

\subsection{Narayana binary trees}

Let $T$ be a binary tree. Denote by $\Nar(T)$ the number of nodes of $T$
that have a nonempty right child. We say that $T$ is a \emph{$k$-Narayana binary tree}
if $\Nar(T) = k$. These binary trees are enumerated by the Narayana numbers~\cite{Nar55}
(see Sequence \Sloane{A001263} of~\cite{SLO08}). First values are
\begin{center}
    \begin{tabular}{c|llllllll}
        $n$ & \multicolumn{8}{c}{$\#\{T \in \EnsAB_n : \Nar(T) = k\}$, $k = 0, \dots, n\!-\!1$} \\ \hline
        $1$ & $1$ \\
        $2$ & $1$ & $1$ \\
        $3$ & $1$ & $3$  & $1$ \\
        $4$ & $1$ & $6$  & $6$   & $1$ \\
        $5$ & $1$ & $10$ & $20$  & $10$  & $1$ \\
        $6$ & $1$ & $15$ & $50$  & $50$  & $15$  & $1$ \\
        $7$ & $1$ & $21$ & $105$ & $175$ & $105$ & $21$ &  $1$ \\
        $8$ & $1$ & $28$ & $196$ & $490$ & $490$ & $196$ & $28$ & $1$
    \end{tabular}
\end{center}

\begin{Proposition}
    For all $k \geq 0$ and $T_0$ and $T_1$ two $k$-Narayana binary trees
    such that $T_0 \OrdTam T_1$, the interval $[T_0, T_1]$ only contains
    $k$-Narayana binary trees.
\end{Proposition}
\begin{proof}
    Consider a node $y$ of $T_0$ and let $(A \ABCons B) \ABCons C$ the
    subtree of $T_0$ of root $y$ and $T_1$ be the binary tree obtained by
    the rotation of root $y$ from $T_0$. Regardless $A$ and $C$, if $B$
    is not empty, $T_0$ and $T_1$ have the same number of nodes that have
    a right child; Otherwise, the number of right children increases by
    one in $T_1$. Hence, in every chain $T_0 \OrdTam T_1 \OrdTam \dots \OrdTam T_\ell$,
    we have $\Nar(T_0) \leq \Nar(T_1) \leq \dots \leq \Nar(T_\ell)$.
    That proves that the set of $k$-Narayana binary trees is closed by
    interval in the Tamari lattice.
\end{proof}

\begin{Proposition} \label{prop:ArbresNarayanaUnionCanop}
    For all $k \geq 0$, the set of $k$-Narayana binary trees with $n$
    nodes is the disjoint union of the sets $\EnsCanop_u$ where
    $\ell(u) = n - 1$ and $u$ contains $k$ occurrences of $1$.
\end{Proposition}
\begin{proof}
    It is enough to show that for all binary tree $T$ of canopy $u$,
    the number of $1$ in $u$ is $\Nar(u)$. Let us show this property by
    structural induction on the set of binary trees. If $T$ is empty,
    this property is clearly satisfied. Assume now that $T := L \ABCons R$,
    and set $v := \Canop(L)$ and $w := \Canop(R)$. We have now to deal
    four cases whether $L$ and $R$ are empty or not.
    \begin{enumerate}[label = \underline{\bf Case \arabic*:}, fullwidth]
        \item If $L$ and $R$ are empty, $T$ is the one-node binary tree
        and the property is satisfied.
        \item If $L$ and $R$ are both not empty, then $\Canop(T) = v.0.1.w$.
        Since $\Nar(T) = \Nar(L) + \Nar(R) + 1$, by induction hypothesis,
        the property is satisfied.
        \item If $L$ is empty and $R$ not, then $\Canop(T) = 1.w$. Since
        $\Nar(T) = \Nar(R) + 1$, by induction hypothesis, the property
        is satisfied.
        \item If $R$ is empty and $L$ not, then $\Canop(T) = v.0$. Since
        $\Nar(T) = \Nar(L)$, by induction hypothesis, the property is
        satisfied.
    \end{enumerate}
\end{proof}

\begin{Corollaire}
    For all $k \geq 0$, the set of $k$-Narayana binary trees with $n$ nodes
    is a disjoint union of intervals in the Tamari lattice.
\end{Corollaire}
\begin{proof}
    The property follows from the fact that the set of $k$-Narayana binary
    trees with $n$ nodes is the union of some binary trees with a given
    canopy (Proposition~\ref{prop:ArbresNarayanaUnionCanop}) and that the
    sets of binary trees with a given canopy are intervals of the Tamari lattice
    (Proposition~\ref{prop:ClotureIntArbresCanopee}).
\end{proof}

\bibliographystyle{alpha}
\bibliography{Bibliographie}

\end{document}